\documentclass[11pt,reqno]{amsart}
\usepackage{stmaryrd}
\usepackage{mathrsfs}
\usepackage{amsfonts,amssymb,amsmath,amsthm}
\usepackage[colorlinks=true,citecolor=red,linkcolor=blue]{hyperref}
\usepackage{titletoc}
\usepackage{geometry}
\usepackage{bm}
\usepackage{indentfirst}
\usepackage{graphicx}
\usepackage{stmaryrd}
\usepackage{tikz}
\usepackage{color,soul}
\usepackage{setspace}
\usepackage[pagewise]{lineno}
\usepackage{float}
\usepackage{appendix}

\newcommand{\blackhyperref}[2]{%
  {%
    \hypersetup{linkcolor=black}%
    \hyperref[#1]{#2}%
  }%
}

\usepackage{enumitem}

\usepackage{todonotes}
\usepackage{cleveref}
\hypersetup{urlcolor=blue, citecolor=red, linkcolor=blue}
\usepackage[square,numbers]{natbib}

\numberwithin{equation}{section}
\theoremstyle{plain} 

\allowdisplaybreaks

\newcommand{\R}{ {\mathbb{R}}}
\newcommand{\C}{ {\mathbb{C}}}

\newcommand{\slab}{\mathcal{S}}
\newcommand{\I}{\mathrm i}
\newcommand{\e}{{\mathrm e}}

\newcommand{\Diff}{{\partial_{x_1}}}
\newcommand{\diff}{d}
\newcommand{\p}{ {\partial} }

\newcommand{\cuad}{{\sqcap\kern-.68em\sqcup}}

\newcommand{\be}{\begin{equation}}
\newcommand{\ee}{\end{equation}}

\newcommand{\pd}{\partial}

\theoremstyle{plain}
\newtheorem{theorem}{Theorem}[section]
\newtheorem{lemma}[theorem]{Lemma}
\newtheorem{prop}[theorem]{Proposition}
\newtheorem{corollary}[theorem]{Corollary}
\newtheorem{remark}{Remark}[theorem]
\newcommand{\bremark}{\begin{remark} \em}
\newcommand{\eremark}{\end{remark} }

\title[Steady water waves with localized vorticities]{Smooth stationary water waves with multiple vortex spikes}
\author[Z. Fan]{Zhiqiang Fan}
\address{\noindent School of Mathematics, Tianjin University, Tianjin 300072, P. R. China}
\email{fanzq@tju.edu.cn}
\author[Y. Wang]{Yuchen Wang}
\address{\noindent School of Mathematical Science, Tianjin Normal University,
Tianjin 300387, P. R. China}
\email{wangyuchen@mail.nankai.edu.cn}
\author[Y. Zheng]{Youquan Zheng}
\address{\noindent School of Mathematics, Tianjin University, Tianjin 300072, P. R. China}
\email{zhengyq@tju.edu.cn}

\begin{document}
\begin{abstract}
For any prescribed positive integer $n$, we establish the existence of a family of stationary solutions to the two-dimensional capillary-gravity water wave equations in finite depth, whose stream functions exhibit $n$ well-separated vertically aligned vortex spikes with alternating signs. The construction is based on a nonlinear gluing scheme for a singularly perturbed semilinear elliptic equation coupled to the Bernoulli equation on the free surface. Our results extend the single-spike construction to arbitrary finite multi-spike configurations and reveal a mechanism through which exponentially localized vortex spikes can coexist with an almost flat free surface. 
\vspace{0.3 cm}

\end{abstract}

\maketitle

{\hypersetup{linkcolor=magenta} \tableofcontents}

\section{Introduction}\label{sec:intro}

Water waves system constitute a central subject in the mathematical theory of fluid mechanics. For the two-dimensional case, this system can be modeled by the two-dimensional incompressible Euler equations with free boundary:  
\begin{subequations}\label{1.1}
\begin{equation}\label{1.1a}
\partial_t v+(v\cdot\nabla)v+\nabla p=-g\vec{e}_2, \quad 
\nabla\cdot v=0, \;\;  x\in\Omega(t),
\end{equation}
where $v=(v_1,v_2)^t$ is the fluid velocity and $P$ its the pressure. Here $-g \vec{e}_2$ denotes the downward gravity acting on the bulk of water. The evolving fluid domain 
\begin{equation}\label{1.1b}
\Omega(t) = \left\{(x_1,x_2)\in\mathbb R^2: -1<x_2<1+\zeta(t,x_1),\;\; \zeta(t,\cdot)\in H^s(\mathbb R), s>3/2 \right\}.
\end{equation}
is bounded by the flat bottom
\[
\mathbf B=\{(x_1,x_2):x_2=-1\}
\]
from below and \textit{the water-air interface} $S(t)= \left(x_1,1+\zeta(t,x_1)\right)$ from above. The fluid 
satisfies the kinematic boundary condition 
\begin{equation}\label{1.1c}
\partial_t\zeta=v_2-v_1\partial_{x_1}\zeta\qquad\text{on }S(t),
\end{equation}
and the slip boundary condition 
\begin{equation}\label{1.1d}
v_2=0\qquad\text{on }\mathbf B.
\end{equation}
Here $\|\zeta(t,\cdot)\|_{L^\infty(\mathbb R)} \leq \frac12$ guarantees the well-separation of the water-air interface and the bottom. Moreover, the Young--Laplace law gives the pressure 
\begin{equation}\label{1.1e}
p=\alpha^2\kappa,\qquad\kappa=-\frac{\partial_{x_1}^2\zeta}{\bigl(1+(\partial_{x_1}\zeta)^2\bigr)^{3/2}},
\end{equation}
where $g>0$ is the gravitational acceleration and $\alpha^2>0$ is the surface-tension coefficient.
\end{subequations}
Here the pressure in the region above the free surface has been normalized to zero. 

The local well-posedness theory for free-boundary incompressible Euler equations, including flows with
nontrivial vorticity, has been developed in a number of works; see, for example, \cite{ChristodoulouLindblad2000,Lindblad2005,CoutandShkoller2007,ShatahZeng2008Geometry,ShatahZeng2011,ShatahZeng2008interface} and the references therein. Nevertheless, the construction and qualitative analysis of coherent structures in rotational water-wave dynamics remain relatively less understood. One of the principal difficulties is that the free surface is affected not only by the velocity field but also, nonlocally, by the nonlinear interactions of vorticity within the fluid. It is also worth noting that the rotational water-wave equations admit a noncanonical Hamiltonian formulation \cite{Lannes2015} with a generalized Dirichlet--Neumann operator. The corresponding energy is given as:
\begin{equation}\label{E:energy}
\begin{aligned}
E(v,\zeta)=\frac12\int_{\Omega(t)}|v|^2\,dx+\frac g2\int_{\R}\zeta^2\,dx_1+\alpha^2\int_{\R}\left(\sqrt{1+(\partial_{x_1}\zeta)^2}-1\right)\,dx_1.
\end{aligned}
\end{equation}
This structure provides a natural framework for studying finite-energy solutions, for which the kinetic energy is finite and the free-surface perturbation decays at spatial infinity.

Substantial progress has been made in understanding irrotational water waves, including traveling water waves and their stability and long-time dynamics, see, for example, \cite{Lannes2013}. The corresponding questions for rotational flows, however, remain considerably more challenging. The presence of vorticity introduces a transport mechanism in the bulk of water and leads to a nontrivial coupling between the interior vortex motion and the evolving free surface. We refer interested readers to the monograph \cite{Adrian2011} for more details. 

A particularly interesting problem is to understand how localized vorticity structures, which have been extensively studied in the incompressible Euler equations in fixed domains \cite{CaoLiuWei2014,CPY2015,Didier2010}, can coexist with a free surface, while retaining a finite total energy. In a finite-depth fluid, concentrated vorticity interacts with both the rigid bottom and the water-air interface, thereby producing a delicate balance between nonlinear interior dynamics and boundary deformation. It is therefore natural to investigate whether several strongly concentrated vorticity structures can be arranged into a stationary configuration. Such solutions need not be small in velocity or vorticity while the resulting deformation of the free surface remains small.

A fundamental class of coherent structures consists of traveling waves:
\begin{equation}\label{E:TW}
v(t,x_1,x_2)=v(x_1-\mathbf ct,x_2),
\qquad
\zeta(t,x_1)=\zeta(x_1-\mathbf ct),
\end{equation}
where $\mathbf c\in\mathbb R$ is the propagation speed of the wave and the free-surface displacement tends to zero as $|x_1|\to\infty$.

In this paper, we restrict attention to stationary solutions, corresponding to $\mathbf c=0$. We seek localized configurations satisfying
\[v\in L^2(\Omega),
\qquad
\zeta\in H^2(\mathbb R),
\]
where
\[\Omega=\{(x_1,x_2):-1<x_2<1+\zeta(x_1)\}.
\]

For two-dimensional incompressible flows, the scalar vorticity is defined by
\[
\omega=\partial_{x_1}v_2-\partial_{x_2}v_1.
\]
Taking the curl of the Euler equations gives the vorticity transport equation
\begin{equation}\label{1.2}
\partial_t\omega+v\cdot\nabla\omega=0.
\end{equation}
Consequently, a stationary solution satisfies
\begin{equation}\label{1.4}
\begin{cases}
v\cdot\nabla\omega=0,& x\in\Omega,\\
v_2-v_1\partial_{x_1}\zeta=0,& x\in S.
\end{cases}
\end{equation}

Since the fluid is incompressible, we introduce a stream function $\Psi$ through
\[
v=\nabla^\perp\Psi=(-\partial_{x_2}\Psi,\partial_{x_1}\Psi).
\]
Then
\[
\omega=\Delta\Psi,
\]
and the stationary vorticity equation becomes
\begin{equation}\label{E:Steady-Euler}
\nabla^\perp\Psi\cdot\nabla\Delta\Psi=0\qquad\text{in }\Omega.
\end{equation}

The kinematic boundary condition implies that $\Psi$ is constant along the free surface. Likewise, the
impermeability condition implies that $\Psi$ is constant along the rigid bottom. For localized stationary flows with finite kinetic energy, the horizontal volume flux vanishes, and hence these two boundary constants agree. After adding a constant to the stream function, we may therefore impose
\begin{equation}\label{E:Dirichlet}
\Psi=0\qquad\text{on }\partial\Omega.
\end{equation}

To obtain a formulation of elliptic equations, it is sufficient to prescribe the vorticity as a nonlinear function of the stream function. More precisely, we consider
\begin{equation}\label{E:vorticity-strength}
\omega=\frac1{\delta^2}\gamma(\Psi),
\end{equation}
where $0<\delta\ll1$ is a concentration parameter and $\gamma$ is a given smooth nonlinearity.

Since
\[
\nabla^\perp\Psi\cdot\nabla\gamma(\Psi)=0,
\]
the stationary vorticity equation is automatically satisfied whenever $\Psi$ solves
\begin{equation}\label{eq:main}
\begin{cases}
\displaystyle\Delta\Psi=\frac1{\delta^2}\gamma(\Psi),& x\in\Omega,\\[2mm]
\Psi=0,& x\in\partial\Omega.
\end{cases}
\end{equation}

The remaining condition is furnished by the Bernoulli equation on the free surface. Let
\[
\Gamma(t)=\int_0^t\gamma(s)\,ds.
\]
For the stationary Euler flow satisfying \eqref{E:vorticity-strength}, the momentum equation implies
\[
\nabla\left(\frac12|\nabla\Psi|^2+p+gx_2-\frac1{\delta^2}\Gamma(\Psi)\right)=0.
\]
Assuming that the velocity and the surface displacement vanish at spatial infinity, the constant is equal to $g$. Since $\Psi=0$ on $S$ and $\Gamma(0)=0$, the Young--Laplace condition yields
\begin{equation}\label{E:Bernoulli}
\frac12|\nabla\Psi|^2+g\zeta+\alpha^2\kappa=0\qquad\text{on }S.
\end{equation}

Thus, the construction of stationary rotational capillary-gravity waves is reduced to the coupled elliptic free-boundary system
\begin{equation}\label{E:free-boundary-system}
\begin{cases}
\displaystyle\Delta\Psi=\frac1{\delta^2}\gamma(\Psi),& x\in\Omega,\\[2mm]
\Psi=0,& x\in\partial\Omega,\\[2mm]
\displaystyle\frac12|\nabla\Psi|^2+g\zeta+\alpha^2\kappa=0,& x\in S,
\end{cases}
\end{equation}
where the fluid domain $\Omega$ is itself unknown.

The main difficulty is that the geometry of the domain depends on the solution through the nonlinear Bernoulli condition. In particular, the vorticity concentrations produce a velocity field whose trace on the free surface determines the resulting boundary deformation.

The existence and qualitative properties of steady water waves carrying localized vorticity have attracted considerable attention in recent years. An important contribution was made by Shatah, Walsh, and Zeng \cite{ShatahWalshZeng2013}, who established the existence of traveling capillary-gravity water waves with compactly supported vorticity, including configurations involving submerged point vortices and vortex patches. In finite depth, Varholm \cite{Varholm2016} constructed solitary gravity-capillary waves carrying finitely many point vortices arranged along a vertical line. The existence, nonexistence, and far-field
asymptotics of deep-water solitary waves with localized vorticity were subsequently investigated by Chen, Walsh, and Wheeler \cite{ChenWalshWheeler2019}.

Further developments include the construction and instability analysis of solitary waves carrying a finite vortex dipole \cite{Le2019}, the conditional orbital stability of solitary waves with a submerged point vortex \cite{VarholmWahlenWalsh2020}, and the long-time dynamics of water waves interacting with point vortices \cite{Su2020}. More recently, Chen, Varholm, Walsh, and Wheeler \cite{ChenVarholmWalshWheeler2025}
established the existence of large-amplitude solitary gravity waves carrying a submerged point vortex, without requiring surface tension.

A particularly relevant result for the present work was obtained by Ehrnstrom, Walsh, and Zeng
\cite{MR4577959}. They constructed smooth, finite-energy stationary capillary-gravity waves with exponentially localized vorticity. Their approach revealed a connection between stationary rotational water waves and singularly perturbed semilinear elliptic equations. In particular, the construction is based on
a localized ground-state profile, allowing the kinetic energy to remain nonvanishing
even when the free-surface displacement is exponentially small. Motivated by this development, we investigate stationary water waves possessing multiple localized stream-function spikes with alternating signs.

More precisely, we are interested in solutions for which the vorticity is concentrated near finitely many points in the fluid interior. The natural building block is a positive, radially symmetric, spatially localized solution of the entire-space semilinear elliptic equation
\begin{equation}\label{eq:ground state}
\Delta U=\gamma(U)\qquad\text{in }\mathbb R^2.
\end{equation}

The concentration parameter $\delta$ determines the characteristic size of each localized core. Indeed, if $x_0\in\Omega$, then
\[
U\left(\frac{x-x_0}{\delta}\right)
\]
solves the rescaled equation
\[
\Delta U_\delta=\frac1{\delta^2}\gamma(U_\delta)
\]
in the whole plane.

We impose the following assumptions on the nonlinearity and its ground state.

\begin{enumerate}[label=\rm(\Alph*)]
\item \label{assumption:A}
There exists an integer $k_0\geq4$ such that $\gamma\in C^{k_0}(\mathbb R)$ and
\[
\gamma(0)=0,
\qquad
\gamma'(0)=1,
\qquad
\gamma''(t)<0\quad\text{for }t>0.
\]
Moreover, $\gamma$ is odd:
\[
\gamma(-t)=-\gamma(t).
\]
Equation \eqref{eq:ground state} admits a nontrivial positive radial solution $U\in C^{k_0+2}(\mathbb R^2)$ satisfying
\[
U(x),\nabla U(x)\longrightarrow0\qquad\text{as }|x|\to\infty.
\]

\item \label{assumption:B}
The ground state $U$ is nondegenerate modulo translations, in the sense that
\[
\ker\left(-\Delta+\gamma'(U)\right)=\operatorname{span}\left\{\partial_{x_1}U,\partial_{x_2}U\right\}
\]
as an operator from $H^2(\mathbb R^2)$ into $L^2(\mathbb R^2)$.
\end{enumerate}

A representative example is
\[
\gamma(t)=t-t^3,
\]
for which \eqref{eq:ground state} becomes
\[
-\Delta U+U=U^3.
\]
More generally, nonlinearities of the form
\[
\gamma(t)=t-|t|^pt
\]
may be considered whenever the exponent and the corresponding ground state satisfy the regularity and nondegeneracy assumptions above.

Under these assumptions, the ground state and its derivatives of all orders decay exponentially at infinity.
The nondegeneracy property is essential for the perturbative construction, since the kernel of the linearized operator is spanned by the two translational modes.

The oddness of $\gamma$ is particularly important in the present problem. Indeed, if $U$ solves
\eqref{eq:ground state}, then $-U$ is also a solution, and
\[
\gamma'(-U)=\gamma'(U).
\]
Thus positive and negative stream-function spikes have identical linearized operators. This permits the construction of alternating-sign configurations using a common ground-state profile.

The purpose of this paper is to construct a family of finite-energy stationary capillary-gravity water waves
containing an arbitrary prescribed number of localized vorticity concentrations. More precisely, we establish the existence of solutions whose stream functions exhibit multiple vertically aligned spikes with alternating signs. The vorticity becomes increasingly concentrated as a small parameter tends to zero, whereas the free-surface deformation is exponentially small. The construction reveals an interaction mechanism that involves neighboring vorticity concentrations and the two boundaries of the fluid domain.

To this end, we shall construct solutions whose leading-order stream functions take the form
\begin{equation}\label{eq:def psi}
\Psi(x)=\sum_{j=1}^n(-1)^{j+1} U\left(\frac{x-x_j}{\delta}\right)+\text{exponentially small corrections},
\end{equation}
where $x_j$ are the concentration points to be determined.

Correspondingly, the vorticity has the leading structure
\[
\omega(x)\approx\frac1{\delta^2}\sum_{j=1}^n(-1)^{j+1}\Delta U\left(\frac{x-x_j}{\delta}\right).
\]
The localization of the ground state implies that the different cores interact only exponentially weakly when their mutual distances remain bounded away from zero. These exponentially small interactions play a decisive role in determining the admissible
positions of the spikes.

Our main theorem establishes the existence of stationary water waves with an arbitrary finite
number of vertically aligned, alternating-sign stream-function spikes.
\begin{theorem}\label{main euler theorem}
Assume that $\gamma$ satisfies Assumptions \ref{assumption:A} and \ref{assumption:B}.
Let $g>0$, $\alpha>0$, and let $n\geq1$ be an integer. There exists $\delta_0>0$ such that, for every $\delta\in(0,\delta_0)$, the stationary capillary-gravity water-wave problem \eqref{E:free-boundary-system} admits a solution
\[
(\Psi_\delta,\zeta_\delta)\in\bigl(H^{k_0}(\Omega)\cap H_0^1(\Omega)\bigr)\times H^{k_0}(\mathbb R)
\]
possessing exactly $n$ localized stream-function spikes.

More precisely, there exist concentration heights
\[
\boldsymbol{\tau}=(\tau_1,\ldots,\tau_n)\in\mathbb R^n
\]
satisfying
\begin{equation}\label{E:spike-locations}
\left|\tau_j-\frac{n-2j+1}{n}\right|\leq C\delta,\qquad j=1,\ldots,n,
\end{equation}
such that
\begin{equation}\label{E:spike-expansion}
\Psi_\delta(x)=\sum_{j=1}^n\sigma_j\left[U_j\left(\frac{x}{\delta}\right)-(U_j)_{\mathrm{bc}}\left(\frac{x}{\delta}\right)\right]+\widetilde v\left(\frac{x}{\delta}\right),
\end{equation}
where
\[
U_j(y)=U\left(y-\frac{\tau_j}{\delta}\vec e_2\right),\qquad\sigma_j=(-1)^{j+1}.
\]
Here $(U_j)_{\mathrm{bc}}$ denotes the corresponding boundary correction, as defined in Section~\ref{perturbed problem section}.

Furthermore, the correction term and the free-surface displacement satisfy
\begin{equation}\label{E:main-estimates}
\left\|\widetilde v(\cdot/\delta)\right\|_{H^{k_0}(\Omega)}\leq C\delta^{\beta}\exp\left(-\frac{2}{n\delta}\right)
\end{equation}
and
\begin{equation}\label{E:surface-estimate}
\|\zeta_\delta\|_{H^{k_0}(\mathbb R)}\leq C\delta^{\beta}|\log\delta|\exp\left(-\frac{2}{n\delta}\right).
\end{equation}
The constants $C$ and $\beta$ are independent of sufficiently small $\delta$.
\end{theorem}

The theorem extends the single-spike construction of \cite{MR4577959} to configurations containing an arbitrary finite number of alternating-sign spikes. In contrast to a single concentration, the multi-spike construction requires a detailed analysis of the interactions among neighboring cores and between the cores and the fluid boundaries.

An important feature of the result is the asymptotic distribution of the concentration points. By \eqref{E:spike-locations},
\[
\tau_j\longrightarrow\frac{n-2j+1}{n}\qquad\text{as }\delta\to0.
\]
Hence the limiting concentration points are equally spaced along the vertical segment joining the bottom and the undisturbed free surface. In particular,
\[
\tau_j-\tau_{j+1}=\frac2n+O(\delta),\qquad j=1,\ldots,n-1.
\]

The distance between each outermost concentration point and the corresponding boundary tends to
$1/n$. Thus the nearest-neighbor separation is twice the limiting distance from an outermost spike to the adjacent boundary.

This geometric arrangement reflects the balance between the exponentially weak interactions
of neighboring spikes and the boundary effects. Although these interactions disappear in the
leading-order limit, they determine the finite-dimensional conditions governing the positions of the concentrations.

We emphasize several consequences of the main theorem. First, the constructed solutions have finite, nonvanishing kinetic energy in the singular limit $\delta\to0$. Indeed, the spikes are well separated, and their mutual interactions are exponentially small. The scaling invariance of the Dirichlet integral in two dimensions yields
\begin{equation}\label{E:kinetic-asymptotics}
\begin{aligned}
\int_{\Omega_\delta}|\nabla\Psi_\delta|^2\,dx&=\sum_{j=1}^n\int_{\mathbb R^2}|\nabla U|^2\,dy+o(1)\\
&=n\|\nabla U\|_{L^2(\mathbb R^2)}^2+o(1).
\end{aligned}
\end{equation}
Consequently, the kinetic energy follows
\[
\frac12\int_{\Omega_\delta}|v_\delta|^2\,dx\longrightarrow\frac n2\int_{\mathbb R^2}|\nabla U|^2\,dy.
\]

Thus, although the fluid velocity becomes strongly localized, its total kinetic energy remains of order one.

Second, the vorticity becomes increasingly concentrated. Since
\[
\omega_\delta = \frac1{\delta^2}\gamma(\Psi_\delta),
\]
its magnitude satisfies
\[
\|\omega_\delta\|_{L^\infty(\Omega_\delta)}\simeq \delta^{-2}.
\]
Moreover, the localization and separation of the spikes imply
\begin{equation}\label{E:vorticity-L1}
\|\omega_\delta\|_{L^1(\Omega_\delta)}=n\|\Delta U\|_{L^1(\mathbb R^2)}+o(1).
\end{equation}

On the other hand, integrating \eqref{eq:ground state} over the whole plane gives
\[
\int_{\mathbb R^2}\Delta U\,dx=0.
\]
Therefore each isolated ground-state core has zero net circulation, although its absolute vorticity is nonzero.

The total circulation of the constructed configuration is consequently asymptotically small:
\begin{equation}\label{E:total-vorticity}
\int_{\Omega_\delta}\omega_\delta\,dx=\int_{\partial\Omega_\delta}\partial_N\Psi_\delta\,dS=o(1).
\end{equation}
In fact, the exponential localization of the spikes and their separation from the boundary yield an exponentially small boundary flux.

These properties distinguish the solutions from classical point-vortex configurations with prescribed nonzero circulations. The present solutions instead describe localized, sign-changing vorticity structures generated by nonlinear elliptic ground states. Third, despite the strong concentration of the interior vorticity, the resulting displacement of the free surface is exponentially small. Indeed, \eqref{E:surface-estimate} implies
\[
\|\zeta_\delta\|_{H^{k_0}(\mathbb R)}\longrightarrow0
\]
faster than any algebraic power of $\delta$. The underlying mechanism can be seen directly
from the Bernoulli condition. Since the spikes remain at positive distances from the free
surface, their induced velocity on the surface is exponentially small. Accordingly, the boundary forcing
\[
-\frac12|\nabla\Psi_\delta|^2
\]
is exponentially small.

To leading order, the free-surface equation takes the form
\begin{equation}\label{E:surface-linearized}
(g-\alpha^2\partial_{x_1}^2)\zeta_\delta\sim -\left.\frac12|\nabla\Psi_\delta|^2\right|_{S}.
\end{equation}
The positivity of the inverse of $g-\alpha^2\partial_{x_1}^2$ explains why
the leading-order deformation is a depression rather than an elevation. Thus the constructed solutions exhibit a separation of scales: the vorticity amplitude is of order $\delta^{-2}$, the kinetic energy is of order one, and the deformation of the free surface is exponentially small. In contrast to a single-spike configuration, the presence of several concentrations introduces additional interaction effects. The positions of the individual spikes must be determined simultaneously through a balance involving neighboring spikes and the boundaries of the fluid domain.

The main contribution in this paper is to construct, for every prescribed integer $n\geq 1$, a family of finite-energy stationary capillary-gravity waves whose stream functions exhibit exactly $n$ well-separated, alternating-sign spikes. The concentration centers approach an
equally spaced vertical configuration, while the free-surface deformation remains exponentially small. Our construction is based on a nonlinear gluing procedure combining the nondegeneracy of the entire-space ground state with a finite-dimensional reduction for the concentration parameters.

Although the resulting vorticity is strongly concentrated, each individual core may contain both positive and negative vorticity. Thus the alternating-sign structure refers specifically to the stream-function spikes. The resulting family provides a new class of stationary rotational water waves in which the internal vorticity distribution exhibits an increasingly complicated multi-core structure, while the free surface remains close to its undisturbed position.

\subsection*{Outline of the construction}

The proof is based on a nonlinear gluing scheme for the coupled elliptic free-boundary
problem, building on methods developed for localized solutions of semilinear elliptic equations; see, for example, \cite{MR2270164,MR2846486}.

We begin with a superposition of $n$ alternating-sign copies of the entire-space ground state, positioned along a vertical line in the reference strip
\[
\Omega_0=\mathbb R\times(-1,1).
\]
Each spike is modified by a suitable boundary correction so that the resulting approximation
satisfies the Dirichlet condition on the fluid boundary.

The next step is to analyze the linearized elliptic operator around this approximate multi-spike configuration. The nondegeneracy assumption \ref{assumption:B} identifies the translational modes as the only leading-order obstructions to invertibility. Suitable orthogonality conditions are then imposed to control these modes and to solve for the infinite-dimensional correction.

A further difficulty arises from the unknown free boundary. The Bernoulli condition couples the surface profile to the velocity generated by the concentrated vorticity. This coupling is handled perturbatively, taking advantage of the exponential separation between the spike cores and the free surface.

Finally, the solvability conditions associated with the translational modes lead to a finite-dimensional system for the spike locations. Its dominant terms describe the interactions between neighboring spikes and their reflections across the fluid boundaries. Solving this system yields the asymptotic location formula \eqref{E:spike-locations} and closes the gluing construction.

The resulting approach demonstrates how the nonlinear gluing techniques for elliptic spike solutions can be adapted to a nonlinear free-boundary fluid problem. It also suggests possible
extensions to other concentrated-vorticity configurations, including more general multi-core arrangements and rotational traveling waves. We note that this approach has also been used to construct overhanging solitary water waves with weak gravity \cite{MR5018790}.

\medskip
\noindent\textbf{Notation. }Throughout the paper, $C$ and $c$ denote positive constants which may change from line to line but are independent of sufficiently small $\delta$. We write $a\lesssim b$ if $a\leq Cb$, and $a\simeq b$ if both $a\lesssim b$ and $b\lesssim a$. For a function $f$ in an $L^2$ space, $f^\perp$ denotes an appropriate $L^2$-orthogonal complement when the underlying subspace is clear from the context.

\section{Reformulation of the elliptic free-boundary problem}\label{sec:perturbation}

The steady water-wave problem is a typical elliptic free-boundary problem in which the fluid domain itself is part of the main unknowns. To address this issue, we employ the conformal mapping parametrization. The construction given here follows  \cite{MR4577959}.

\subsection{Conformal mapping parametrization}\label{1448}
Let the infinitely long strip $$\slab:=\{z_1+\I z_2\in\C\;|\;|z_2|<1\}$$ serve as a reference domain. As the impact of highly localized vorticity on the free surface is considerably weak, the fluid domain 
$$\Omega = \{(x_1,x_2)\in\R^2\;|\; -1< x_2 < 1 + \zeta(x_1) \}$$
can be parameterized as a perturbation of $\slab$:  
\begin{align*}
\Omega&=(id+\Gamma)\left(\{(x_1,x_2)\in\R^2\;|\;|x_2|< 1\}\right)\\
&=\left\{\left(x_1+\Gamma_1(x_1,x_2),x_2+\Gamma_2(x_1,x_2)\right)\in \R^2\;|\;|x_2|< 1\right\}.
\end{align*}

Let $\Gamma=\Gamma_1+\I\Gamma_2:\slab\to\C$ be a holomorphic mapping satisfying ${\Gamma_2|_{\{z_2=-1\}}=0}$. Since $\Gamma_2$ is harmonic, it is uniquely determined by its trace on $x_2=1$,
\begin{equation}\label{E:HarmExt} 
\Gamma_2=\frac{\sinh{\left(|\partial_{x_1}|(x_2+1)\right)}}{\sinh{\left(2|\partial_{x_1}|\right)} } \Gamma_2(\cdot,1)\qquad\textrm{in } \slab,
\end{equation} 
and
\begin{equation} \label{E:DN}
\partial_{x_2}\Gamma_2=|\partial_{x_1}|\coth{\left(2|\partial_{x_1}|\right)} \Gamma_2 \qquad\text{on }\{x_2 = 1\}.
\end{equation}
By the symmetry of the water-wave problem, $\Gamma_2(\cdot,1)$, and hence $\Gamma_2$ is even in $x_1$. Moreover, since $\Gamma_1$ and $\Gamma_2$ are conjugate harmonic functions, $\Gamma_1$ is odd in $x_1$ and is therefore uniquely determined by $\Gamma_2$. Consequently, $\Gamma$ is uniquely determined by its trace $\Gamma_2(\cdot,1)$.

Pulling back the stream function gives
\begin{equation}\label{relation Phi to Psi}
\Phi=\Psi\circ(id+\Gamma):\{(x_1,x_2)\in\R^2\;|\;|x_2|< 1\}\to\R,
\end{equation}
problems \eqref{E:Steady-Euler}, \eqref{E:Bernoulli}, and \eqref{eq:main} become
\begin{equation}\label{1836}
\left\{\begin{aligned}
\delta^2\Delta\Phi&=|1+\Gamma^\prime|^2\gamma(\Phi) & \quad\text{ in } \slab \\
\Phi&=0 & \quad \text{ on } \p \slab,
\end{aligned}\right.
\end{equation}
and the Bernoulli condition \eqref{E:Bernoulli} becomes
\begin{equation}\label{eq:dynamic_Phi}
\frac{1}{2}\frac{(\partial_{x_2}\Phi)^2}{|1+\Gamma^\prime|^{2}}-\alpha^2\frac{Im(\Gamma^{\prime\prime}(1+\overline{\Gamma^\prime}))}{|1+\Gamma^\prime|^{3}}+g\Gamma_2=0\quad\text{ on }\{x_2= 1\},
\end{equation}
where $\Gamma^\prime=\partial_{x_1}\Gamma$. Define 
\begin{equation}\label{relation phi and Phi}
\varphi=\Phi(\delta\,\cdot),
\end{equation}
then equation \eqref{1836} translates to
\begin{equation}\label{eq:varphi elliptic}
\left\{\begin{aligned}
&\Delta\varphi=|1+\Gamma^\prime(\delta\,\cdot)|^2\gamma(\varphi)\qquad&&\text{in }\slab_\delta,\\
&\varphi=0\qquad&&\textrm{on }\partial\slab_\delta, 
\end{aligned}\right.
\end{equation}
where $\slab_\delta$ is defined as
\begin{equation}\label{def slab}
\slab_\delta=\left\{x\in\mathbb{R}^2\;|\;|x_2|<\tfrac{1}{\delta}\right\}. \end{equation}
The Bernoulli condition is transformed into
\begin{equation}\label{eq:dynamic_varphi}
\frac{1}{2\delta^{2}}\frac{\left(\partial_{x_2}\varphi(\cdot,\tfrac{1}{\delta})\right)^2}{|1+\Gamma^\prime(\delta\,\cdot,1)|^{2}}-\alpha^2\frac{Im (\Gamma^{\prime\prime}(\delta\,\cdot,1)(1+ \overline{\Gamma^\prime(\delta\,\cdot,1)}))}{|1+\Gamma^\prime(\delta\,\cdot,1)|^3}+g\Gamma_2(\delta\,\cdot,1)=0.
\end{equation}

\subsection{Approximate solution} \label{perturbed problem section}We want to solve for $\varphi$ modeled on $U$. Although $U$ has exponential localization, it does not satisfy the homogeneous boundary conditions in \eqref{eq:varphi elliptic}. We therefore perform a boundary correction, that is, we subtract from $U$ a function with the same trace but with very small interior effect.

Since $U$ is exponentially localized and $\gamma'(0)=1$, away from the origin, the linearized operator $\gamma^\prime(U)-\Delta$ behaves like $1-\Delta$. Hence for $U\in H^s(\partial\slab_\delta)$, define the boundary correction $U_{\mathrm{bc}}$ of $U$ to be the unique solution of the following equation in the space $H^s(\slab_\delta)$:
\begin{equation}\label{definition bc}
\left\{\begin{aligned}
&(1-\Delta)U_{\mathrm{bc}}=0\qquad&&\textrm{in }\slab_{\delta},\\
&U_{\mathrm{bc}}=U\qquad&&\textrm{on }\partial\slab_\delta.
\end{aligned}\right.
\end{equation}
By elliptic regularity estimates (see \cite[Theorem 8.13]{MR1814364}), for any $s>0$, 
\begin{equation}\label{2154}
\|U_{\mathrm{bc}}\|_{H^{s}(\slab_\delta)}\leq C_s\|U\|_{H^s (\partial\slab_\delta)}.
\end{equation}
Moreover, the constant $C_s$ in the elliptic estimate can be chosen independently of $\delta$ as $\delta\to0$. Indeed, this follows from the Fourier representation of the solution to $(1-\Delta)U_{\mathrm{bc}}=0$ in the strip $\slab_\delta$. To see this, denote the boundary data by
$$U_\pm(x_1):=U\left(x_1,\pm \frac{1}{\delta}\right).$$
Taking the Fourier transform with respect to \(x_1\), we obtain
$$\left(1+\xi^2-\partial_{x_2}^2\right)\widehat U_{\mathrm{bc}}(\xi,x_2)=0.$$
The solution given by the prescribed boundary conditions is
\begin{align*}
\widehat U_{\mathrm{bc}}(\xi,x_2)=&\frac{\sinh\left(\sqrt{1+\xi^2}\left(\delta^{-1}-x_2\right)\right)}
{\sinh\left(2\delta^{-1}\sqrt{1+\xi^2}\right)}\widehat U_-(\xi)+\frac{\sinh\left(\sqrt{1+\xi^2}\left(\delta^{-1}+x_2\right)\right)}
{\sinh\left(2\delta^{-1}\sqrt{1+\xi^2}\right)}\widehat U_+(\xi).
\end{align*}
Since \(\sqrt{1+\xi^2}\ge 1\) and \(\frac{1}{\delta}\ge C>0\) as $\delta\to0$, we have
\begin{align*}
\left|\partial_{x_2}^{j}\frac{\sinh\left(\sqrt{1+\xi^2}\left(\delta^{-1}-x_2\right)\right)}
{\sinh\left(2\delta^{-1}\sqrt{1+\xi^2}\right)}\right|\leq C_j(1+\xi^2)^{\frac{j}{2}}\e^{-\sqrt{1+\xi^2}(x_2+\delta^{-1})},
\end{align*}
and
\begin{align*}
\left|\partial_{x_2}^{j}\frac{\sinh\left(\sqrt{1+\xi^2}\left(\delta^{-1}+x_2\right)\right)}{\sinh\left(2\delta^{-1}\sqrt{1+\xi^2}\right)}\right|\leq C_j(1+\xi^2)^{\frac{j}{2}}\e^{-\sqrt{1+\xi^2}(\delta^{-1}-x_2)},
\end{align*}
where the constants $C_j$ are independent of $\delta$ as $\delta\to 0$. Consequently, for all nonnegative integer $m$, Plancherel's theorem yields
\begin{align*}
\|U_{\mathrm{bc}}\|_{H^m(\slab_\delta)}^2&\leq C_m\sum_{\pm}\int_{\mathbb R}(1+\xi^2)^{m-\frac{1}{2}}|\widehat U_\pm(\xi)|^2\,d\xi\\
&\leq C_m\sum_{\pm}\int_{\mathbb R}(1+\xi^2)^{m}|\widehat U_\pm(\xi)|^2\,d\xi\leq C_m\sum_{\pm}\|U_\pm\|_{H^m(\mathbb R)}^2.
\end{align*}
By interpolation, the same estimate holds for every real $s\ge0$. Thus
$$\|U_{\mathrm{bc}}\|_{H^s(\slab_\delta)}\le C_s\left(\|U_+\|_{H^s(\mathbb R)}+\|U_-\|_{H^s(\mathbb R)}\right),$$
where $C_s$ is independent of $\delta$ as $\delta\to0$. A more detailed analysis of $U_{\mathrm{bc}}$ can be found in Section \ref{subsec:basic-properties-U}.

For $\tau\in(-1,1)$, let $U(\,\cdot\,;\tau):=U\left(\,\cdot-\tfrac{\tau}{\delta}\vec e_2\right)$ with $U$ as in Assumption A. For a fixed small constant $0<\varepsilon<\frac{1}{n}$, define
$$P_{\varepsilon}:=\{\boldsymbol{\tau}=(\tau_1,\tau_2,\dots,\tau_n)\;|\;\tau_i\in (-1,1), |\tau_i\pm1|>\varepsilon,|\tau_i-\tau_j|>2\varepsilon,i,j=1,\dots,n,i\neq j\}.$$
Given $\boldsymbol{\tau}\in P_{\varepsilon}$, let 
$$\sigma_i:=(-1)^{i+1},\qquad i=1,\dots,n,$$
we construct the approximate solution as
$$\omega_{\boldsymbol{\tau}}=\sum_{i=1}^n\sigma_i\left(U(\,\cdot\,;\tau_i)-U(\,\cdot\,;\tau_i)_{\mathrm{bc}}\right),$$
where $U(\,\cdot\,;\tau_i)_{\mathrm{bc}}$ is the solution of the equation \eqref{definition bc}. For simplicity, we write $U_i:=U(\,\cdot\,;\tau_i)$. Then we make the ansatz
\begin{equation}\label{definition u}
\varphi=\omega_{\boldsymbol{\tau}}+u.
\end{equation}
Substituting it into \eqref{eq:varphi elliptic}, by \(\Delta \sigma_iU_i =\sigma_i \gamma(U_i)\) and \( \Delta \sigma_i(U_i)_\mathrm{bc} = \sigma_i(U_i)_\mathrm{bc}\), $u$ satisfies the following equation:
\begin{equation}\label{eq:main_u}\\
\begin{aligned}
\Delta u=&\left|1+\Gamma^\prime(\delta\,\cdot\,)\right|^2 \gamma(\omega_{\boldsymbol{\tau}}+u)-\sum_{i=1}^n\sigma_i\gamma(U_i)+\sum_{i=1}^n\sigma_i(U_i)_\mathrm{bc}\\
=&\gamma^\prime\left(\sum_{i=1}^n\sigma_iU_i\right)u+\left|1+\Gamma^\prime(\delta \,\cdot\,)\right|^2\gamma(\omega_{\boldsymbol{\tau}}+u)\\
&-\sum_{i=1}^n\sigma_i\gamma(U_i)+\sum_{i=1}^n\sigma_i(U_i)_\mathrm{bc}-\gamma^\prime\left(\sum_{i=1}^n\sigma_iU_i\right)u\qquad\text{in }\slab_\delta.
\end{aligned}
\end{equation}
Since $\omega_{\boldsymbol{\tau}}|_{\partial\slab_\delta}=0$, the boundary value condition in \eqref{eq:varphi elliptic} reduces to
\begin{equation}\label{eq:kinematic_u}
u=0 \qquad\textrm{on }\partial\slab_\delta.
\end{equation}

We next turn to the Bernoulli condition \eqref{eq:dynamic_varphi}. Let $\Gamma_{\mathrm{s}}=\Gamma_2(\,\cdot\,,1)$, using the Cauchy-Riemann equations,
$$\Gamma^\prime=\partial_{x_1}\Gamma_1+\I \partial_{x_1} \Gamma_2=\partial_{x_2}\Gamma_2+\I \partial_{x_1} \Gamma_2.$$
Together with \eqref{definition u} and \eqref{E:DN}, we have
\begin{equation}\label{eq:dynamic_eta}
\frac{1}{2\delta^{2}}\frac{\left(\partial_{x_2}(\omega_{\boldsymbol{\tau}}+u)\left(\frac{\,\cdot\,}{\delta},\frac{1}{\delta} \right)\right)^2}{\left(1+m(\Diff)\Gamma_{\mathrm{s}}\right)^2+{\Gamma_{\mathrm{s}}^\prime}^2}-\alpha^2\frac{(1+m(\Diff)\Gamma_{\mathrm{s}})\Gamma_{\mathrm{s}}^{\prime\prime}-\Gamma_{\mathrm{s}}^\prime m(\Diff)\Gamma_{\mathrm{s}}^\prime}{\left((1+m(\Diff)\Gamma_{\mathrm{s}})^2+{\Gamma_{\mathrm{s}}^\prime}^2\right)^{3/2}}+g\Gamma_{\mathrm{s}}=0, 
\end{equation}
where $m(\Diff)=|\Diff|\coth{(2|\Diff|)}$. If $\|\Gamma_{\mathrm{s}}\|_{H^2}$ is sufficiently small, the leading term of the operator in \eqref{eq:dynamic_eta} acting on $\Gamma_{\mathrm{s}}$ is an invertible operator $g-\alpha^2 \partial_{x_1}^2:H^k(\R)\to H^{k-2}(\R),\,k\geq2$. Define the operator $A(\Gamma_{\mathrm{s}})$ by   
$$A(\Gamma_{\mathrm{s}})=\left(g-
\alpha^2\frac{\left(1+m(\Diff)\Gamma_{\mathrm{s}}\right) \p_{x_1}^2-\Gamma_{\mathrm{s}}^\prime m(\Diff)\Diff}{\left((1+m(\Diff)\Gamma_{\mathrm{s}})^2+{\Gamma_{\mathrm{s}}^\prime}^2 \right)^{\frac{3}{2}}}
\right)(g-\alpha^2\pd_{x_1}^2)^{-1}.$$
For $\|\Gamma_{\mathrm{s}}\|_{H^{k+1}}$ small, since $A(0)=\mathrm{id}$ and
$$\|A(\Gamma_{\mathrm{s}})-\mathrm{id}\|_{\mathcal L(H^{k})}\lesssim \|\Gamma_{\mathrm{s}}\|_{H^{k+1}},$$
then $A(\Gamma_{\mathrm{s}})$ is invertible. Hence we can write \eqref{eq:dynamic_eta} as
\begin{equation}\label{eq:dynamic_inverse}
\begin{aligned}
(g-\alpha^2\p_{x_1}^2)\Gamma_{\mathrm{s}}+\frac{1}{2\delta^{2}}A(\Gamma_{\mathrm{s}})^{-1}\left[\frac{\left(\partial_{x_2}(\omega_{\boldsymbol{\tau}}+u)\left(\frac{\,\cdot\,}{\delta},\frac{1}{\delta}\right)\right)^2}{(1+m(\Diff)\Gamma_{\mathrm{s}})^2+{\Gamma_{\mathrm{s}}^\prime}^2}\right]=0.
\end{aligned} 
\end{equation}

From the above analysis, if $u$, $\Gamma_{\mathrm{s}}$ and $\boldsymbol{\tau}$ satisfy 
\begin{subequations}\label{perturbed problem}
\begin{align}
\left(-\Delta+\gamma^\prime\left(\textstyle\sum_{i=1}^n\sigma_iU_i\right)\right)u + F(\boldsymbol{\tau},u,\Gamma_{\mathrm{s}})&=0 \qquad\textrm{in }\slab_\delta \label{perturbed u eq}\\
\left(g-\alpha^2\p_{x_1}^2  \right)\Gamma_{\mathrm{s}}+G(\boldsymbol{\tau},u,\Gamma_{\mathrm{s}})&=0 \qquad\textrm{on }\mathbb{R}, \label{perturbed Gamma eq} 
\end{align}
\end{subequations}
where 
\begin{equation}\label{def F}
F(\boldsymbol{\tau},u,\Gamma_{\mathrm{s}}):=\left|1+\Gamma^\prime(\delta \,\cdot\,)\right|^2\gamma(\omega_{\boldsymbol{\tau}}+u)
-\sum_{i=1}^n\sigma_i\gamma(U_i)+\sum_{i=1}^n\sigma_i(U_i)_\mathrm{bc}-\gamma^\prime\left(\sum_{i=1}^n\sigma_iU_i\right)u
\end{equation}
and
\begin{equation}\label{def G}
G(\boldsymbol{\tau},u,\Gamma_{\mathrm{s}}):=  
\frac{1}{2\delta^{2}}A(\Gamma_{\mathrm{s}})^{-1}\left[\frac{\left(\partial_{x_2}(\omega_{\boldsymbol{\tau}}+u)\left(\frac{\,\cdot\,}{\delta},\frac{1}{\delta}\right)\right)^2}{(1+|\Diff|\coth{(2|\Diff|)}\Gamma_{\mathrm{s}})^2+{\Gamma_{\mathrm{s}}^\prime}^2}\right].  
\end{equation}
Then the pair $(\Psi,\zeta)$, defined through \eqref{E:HarmExt}, \eqref{relation Phi to Psi}, \eqref{relation phi and Phi}, \eqref{definition u}, solves the stationary water wave problem \eqref{E:Steady-Euler}, \eqref{E:Bernoulli} and \eqref{eq:main}. 

\subsection{Basic properties of U and boundary correction}\label{subsec:basic-properties-U}
In this subsection, we collect the estimates for $U$ and its boundary correction that will be used in the sequel. These estimates are consequences of the exponential decay for $U$. We state them here. 

\begin{lemma}\label{prop:sign of Uy} 
Under Assumptions \ref{assumption:A} and \ref{assumption:B}, there exists $\lambda> 0$ such that 
\begin{equation}\label{eq:rate of decay}
\lim_{r\to\infty}(-1)^{k}r^{\frac{1}{2}}\e^{r}\partial_r^k U(r)=\lambda, \qquad\text{for all } 0\le k\le k_0+2.
\end{equation}
In particular, for radial functions, we have 
\begin{equation}\label{1859}
\partial_{x_2}U\eqsim- x_2 r^{-\frac{3}{2}}\e^{-r},\qquad r\gg 1.
\end{equation}
\end{lemma}
The proof can be found in \cite[Proposition 3.1]{MR4577959}. Since $U_i$ is obtained by translating $U$, the estimates for $U_i$ and its derivatives also follow from the above lemma. Next, we consider the boundary correction $(U_i)_{\mathrm{bc}}$. From \eqref{definition bc}, it satisfies
\begin{equation}\label{2142}
\left\{\begin{aligned}
&(1-\Delta)(U_i)_\mathrm{bc}=0\qquad&&\textrm{in }\slab_{\delta},\\
&(U_i)_\mathrm{bc}=U_i\qquad&&\textrm{on }\partial\slab_\delta.
\end{aligned}\right.
\end{equation}
To obtain $(U_i)_{\mathbf{bc}}$, we need the Green's function corresponding to the Dirichlet problem for the operator $1-\Delta$ in $\slab_{\delta}$.

We first consider the Green's function in $\R^2$. Let $G_0$ be the fundamental solution satisfying
$$-\Delta G_0+G_0=\delta_0\qquad\text{in }\R^2.$$
Taking the Fourier transform, we obtain
$$G_0(x)=\frac{1}{(2\pi)^2}\int_{\R^2}\frac{e^{i\xi\cdot x}}{1+|\xi|^2}\,d\xi.$$
Since $G_0$ is radially symmetric, we write $G_0(x)=G_0(|x|)=G_0(r)$. Evaluating this inverse Fourier transform yields
$$G_0(r)=\frac{1}{2\pi}K_0(r),$$
where $K_0$ denotes the modified Bessel function of the second kind of order zero. Consequently, the asymptotic expansions of the modified Bessel functions, together with their differentiation formulas (see, e.g., \cite[§§10.29--10.32 and 10.40]{NIST:DLMF}, \cite{MR3837350}), yield the following estimates:
\begin{equation}\label{012}
\begin{aligned}
G_0(r)&=O(|\log r|)&&\quad\text{as }r\to0,\\
\partial_r^kG_0(r)&=O(r^{-k})&&\quad\text{as }r\to0,k\geq 1,\\
\partial_r^kG_0(r)&=O(r^{-1/2}e^{-r})&&\quad\text{as }r\to\infty,k\geq 0.
\end{aligned}
\end{equation}

Next, we construct the Green's function in the strip $\slab_\delta$ by the method of images. Recall that the boundary of
$\slab_\delta$ consists of two horizontal lines. In order to enforce the Dirichlet boundary condition, we reflect the singularity with respect to each boundary and alternate the sign of the reflected terms.

More precisely, define
\begin{equation}\label{2114}
\begin{aligned}
\widetilde G(x,y):=\sum_{m\in\mathbb Z}\Bigg[&G_0\left(x_1-y_1,
x_2-y_2+\frac{4m}{\delta}\right)\\
&-G_0\left(x_1-y_1,x_2+y_2-\frac{2}{\delta}+\frac{4m}{\delta}
\right)\Bigg].
\end{aligned}
\end{equation}
The alternating signs guarantee the cancellation of the Green function on the two boundary components of $\partial\slab_\delta$. Indeed, on upper boundary $x_2=1/\delta$, using the symmetry of $G_0(r)$, the two sums cancel. Similarly, on lower boundary $x_2=-1/\delta$, the two sums cancel after using the symmetry of \(G_0\) and changing the index $m\mapsto m+1$ of the latter term. Moreover, only the $m=0$ term in \eqref{2114} has a singularity at $x=y$, while all other image terms are smooth in $\slab_\delta$. Hence, $\widetilde G(x,y)$ satisfies
$$\begin{cases}
-\Delta \widetilde G(x,y)+\widetilde G(x,y)=\delta_y(x),&x\in\slab_\delta,\\
\widetilde G(x,y)=0,&x\in\partial\slab_\delta.
\end{cases}$$

We next define the Poisson kernel associated with the Green's function $\widetilde G(x,y)$ by
$$P(x,y):=-\partial_{\nu_y}\widetilde G(x,y),\qquad y\in\partial\slab_\delta,$$
where $\nu_y$ denotes the unit outward normal to $\partial\slab_\delta$. By the Poisson integral formula, we obtain
\begin{align}
&(U_i)_{\mathrm{bc}}(x)=\int_{\partial\slab_\delta}P(x,y)
U_i(y)\,dS\label{1844}\\
&=-\int_{\R}\partial_{y_2}\widetilde G\left(x,\left(y_1,\frac{1}{\delta}\right)\right)U_i\left(y_1,\frac{1}{\delta}\right)\,dy_1+\int_{\R}\partial_{y_2}\widetilde G\left(x,\left(y_1,-\frac{1}{\delta}\right)\right)U_i\left(y_1,-\frac{1}{\delta}\right)\,dy_1.\notag
\end{align}
We proceed to derive some estimates for $(U_i)_{\mathrm{bc}}$ that will be used later.

\begin{corollary}\label{C:Ubc}
For any $|\tau_i|,\,|\tau_j|< 1-\varepsilon$, $U(\,\cdot\,;\tau_i)_\mathrm{bc}$ satisfy
\begin{align}
\label{2023}\left\|U(\,\cdot\,;\tau_i)_\mathrm{bc}\right\|_{H^{k_0+2} (\slab_\delta)}&\lesssim\e^{-\frac{1-|\tau_i|}\delta},\\
\label{2024}\left|\left(U(\,\cdot\,;\tau_i)_\mathrm{bc},U(\,\cdot\,;\tau_j)\right)_{L^2(\slab_\delta)}\right|&\lesssim\delta^{-\frac{5}{2}}\e^{-\frac{2-\tau_i-\tau_j}{\delta}}+\delta^{-\frac{5}{2}}\e^{-\frac{2+\tau_i+\tau_j}{\delta}},\\
\label{2025}\left|\left(\partial_{x_2}U(\,\cdot\,;\tau_i)_\mathrm{bc},U(\,\cdot\,;\tau_j)\right)_{L^2(\slab_\delta)}\right|&\lesssim\delta^{-\frac{5}{2}}\e^{-\frac{2-\tau_i-\tau_j}{\delta}}+\delta^{-\frac{5}{2}}\e^{-\frac{2+\tau_i+\tau_j}{\delta}}.
\end{align}
Here and in what follows, $\left(f,g\right)_{L^2}$ denotes the $L^2$ inner product of $f$ and $g$.
\end{corollary}

\begin{proof}
By \eqref{2154} and Lemma \ref{prop:sign of Uy}, a direct estimate yields
\begin{align*}
&\left| U(\,\cdot\,;\tau_i)_\mathrm{bc}\right|^2_{H^{k_0+2} (\slab_\delta)}\lesssim\left|U(\,\cdot\,;\tau_i)\right|^2_{H^{k_0+2} (\partial\slab_\delta)}\\
&\lesssim\displaystyle\int_{\mathbb R}\left|U\left(x,\frac{1}{\delta}-\frac{\tau_i}{\delta}\right)\right|^2\,dx+\displaystyle\int_{\mathbb R}\left|U\left(x,-\frac{1}{\delta}-\frac{\tau_i}{\delta}\right)\right|^2\,dx\\
&\lesssim\delta\displaystyle\int_{\left\{|x|\leq\frac{1-\tau_i}{\delta}\right\}}\e^{-2\sqrt{x^2+(\frac{1}{\delta}-\frac{\tau_i}{\delta})^2}}\,dx+\delta\displaystyle\int_{\left\{|x|\leq\frac{1+\tau_i}{\delta}\right\}}\e^{-2\sqrt{x^2+(\frac{1}{\delta}+\frac{\tau_i}{\delta})^2}}\,dx\\
&\quad\,+\delta\displaystyle\int_{\left\{|x|>\frac{1-\tau_i}{\delta}\right\}}\e^{-2\sqrt{x^2+(\frac{1}{\delta}-\frac{\tau_i}{\delta})^2}}\,dx+\delta\displaystyle\int_{\left\{|x|>\frac{1+\tau_i}{\delta}\right\}}\e^{-2\sqrt{x^2+(\frac{1}{\delta}+\frac{\tau_i}{\delta})^2}}\,dx\\
&\lesssim\delta\displaystyle\int_{\left\{|x|\leq\frac{1-\tau_i}{\delta}\right\}}\e^{-\frac{2-2\tau_i}{\delta}}\,dx+\delta\displaystyle\int_{\left\{|x|\leq\frac{1+\tau_i}{\delta}\right\}}\e^{-\frac{2+2\tau_i}{\delta}}\,dx+\delta\displaystyle\int_{\frac{1-\tau_i}{\delta}}^{\infty}\e^{-2x}\,dx+\delta\displaystyle\int_{\frac{1+\tau_i}{\delta}}^{\infty}\e^{-2x}\,dx\\
&\lesssim\e^{-\frac{2-2\tau_i}{\delta}}+\e^{-\frac{2+2\tau_i}{\delta}}\lesssim\e^{-\frac{2-2|\tau_i|}{\delta}}.
\end{align*}
This implies \eqref{2023}. We next consider \eqref{2024}. Using \eqref{1844},
$$\left(U(\,\cdot\,;\tau_i)_\mathrm{bc},U(\,\cdot\,;\tau_j)\right)_{L^2(\slab_\delta)}=\displaystyle\int_{\slab_\delta}\displaystyle\int_{\partial\slab_\delta}P(x,y)U(y\,;\tau_i)U(x\,;\tau_j)\,dy\,dx.$$
By Fubini's theorem, \eqref{012} and Lemma \ref{prop:sign of Uy},
\begin{align*}
&\left|\displaystyle\int_{\slab_\delta}\displaystyle\int_{\partial\slab_\delta}P(x,y)U(y\,;\tau_i)U(x\,;\tau_j)\,dy\,dx\right|\lesssim\displaystyle\int_{\partial\slab_\delta}|U(y\,;\tau_i)|\,dy\displaystyle\int_{\slab_\delta}|P(x,y)||U(x\,;\tau_j)|\,dx\\
&\lesssim\displaystyle\int_{\partial\slab_\delta}|U(y\,;\tau_i)|\,dy\left(\displaystyle\int_{\{|x-y|\leq\mu\}\cap\slab_\delta}+\displaystyle\int_{\{|x-y|>\mu\}\cap\slab_\delta}\right)|P(x,y)||U(x\,;\tau_j)|\,dx\\
&\lesssim\displaystyle\int_{\partial\slab_\delta}|U(y\,;\tau_i)|\,dy\displaystyle\int_{\{|x-y|\leq\mu\}\cap\slab_\delta}\left|x-y\right|^{-1}\left|x-\tfrac{\tau_j}{\delta}\vec e_2\right|^{-\frac{1}{2}}\e^{-\left|x-\frac{\tau_j}{\delta}\vec e_2\right|}\,dx\\
&\quad\,+\displaystyle\int_{\partial\slab_\delta}|U(y\,;\tau_i)|\,dy\displaystyle\int_{\{|x-y|>\mu\}\cap\slab_\delta}\left(1+\left|x-y\right|\right)^{-\frac{1}{2}}\left(1+\left|x-\tfrac{\tau_j}{\delta}\vec e_2\right|\right)^{-\frac{1}{2}}\e^{-|x-y|-\left|x-\frac{\tau_j}{\delta}\vec e_2\right|}\,dx\\
&=:I_1+I_2,
\end{align*}
where $\mu$ is a fixed small constant. 

For $I_1$, by Lemma \ref{prop:sign of Uy}, it follows that 
\begin{align*}
&\displaystyle\int_{\partial\slab_\delta}|U(y\,;\tau_i)|\,dy\displaystyle\int_{\{|x-y|\leq\mu\}\cap\slab_\delta}\left|x-y\right|^{-1}\left|x-\tfrac{\tau_j}{\delta}\vec e_2\right|^{-\frac{1}{2}}\e^{-\left|x-\frac{\tau_j}{\delta}\vec e_2\right|}\,dx\\
&\lesssim\displaystyle\int_{\partial\slab_\delta}\left|y-\tfrac{\tau_i}{\delta}\vec e_2\right|^{-\frac{1}{2}}\e^{-\left|y-\frac{\tau_i}{\delta}\vec e_2\right|}\left(\left|y-\tfrac{\tau_j}{\delta}\vec e_2\right|-\mu\right)^{-\frac{1}{2}}\e^{-\left|y-\frac{\tau_j}{\delta}\vec e_2\right|+\mu}\,dy\\
&\lesssim\delta\displaystyle\int_{\partial\slab_\delta}\e^{-\left|y-\frac{\tau_i}{\delta}\vec e_2\right|}\e^{-\left|y-\frac{\tau_j}{\delta}\vec e_2\right|}\,dy\\
&=\delta\displaystyle\int_{\R}\e^{-\sqrt{y_1^2+(\frac{1}{\delta}-\frac{\tau_i}{\delta})^2}}\e^{-\sqrt{y_1^2+(\frac{1}{\delta}-\frac{\tau_j}{\delta})^2}}\,dy_1+\delta\displaystyle\int_{\R}\e^{-\sqrt{y_1^2+(\frac{1}{\delta}+\frac{\tau_i}{\delta})^2}}\e^{-\sqrt{y_1^2+(\frac{1}{\delta}+\frac{\tau_j}{\delta})^2}}\,dy_1\\
&\lesssim\delta\displaystyle\int_{-\frac{2}{\delta}}^{\frac{2}{\delta}}\e^{-\sqrt{y_1^2+(\frac{1}{\delta}-\frac{\tau_i}{\delta})^2}}\e^{-\sqrt{y_1^2+(\frac{1}{\delta}-\frac{\tau_j}{\delta})^2}}\,dy_1+\delta\displaystyle\int_{-\frac{2}{\delta}}^{\frac{2}{\delta}}\e^{-\sqrt{y_1^2+(\frac{1}{\delta}+\frac{\tau_i}{\delta})^2}}\e^{-\sqrt{y_1^2+(\frac{1}{\delta}+\frac{\tau_j}{\delta})^2}}\,dy_1\\
&\quad\,+\delta\displaystyle\int_{\frac{2}{\delta}}^{\infty}\e^{-2y_1}\,dy_1\\
&\lesssim\e^{-\frac{2-\tau_i-\tau_j}{\delta}}+\e^{-\frac{2+\tau_i+\tau_j}{\delta}}.
\end{align*}
For $I_2$, using $|x-y|+\left|x-\frac{\tau_j}{\delta}\vec e_2\right|\geq|y-\frac{\tau_j}{\delta}\vec e_2|$, we have
\begin{align*}
&\displaystyle\int_{\partial\slab_\delta}|U(y\,;\tau_i)|\,dy\displaystyle\int_{\{|x-y|>\mu\}\cap\slab_\delta}\left(1+\left|x-y\right|\right)^{-\frac{1}{2}}\left(1+\left|x-\tfrac{\tau_j}{\delta}\vec e_2\right|\right)^{-\frac{1}{2}}\e^{-|x-y|-\left|x-\frac{\tau_j}{\delta}\vec e_2\right|}\,dx\\
&\lesssim\displaystyle\int_{\partial\slab_\delta}|U(y\,;\tau_i)|\,dy\displaystyle\int_{-\frac{4}{\delta}}^{\frac{4}{\delta}}\displaystyle\int_{-\frac{1}{\delta}}^{\frac{1}{\delta}}\e^{-\left|y-\frac{\tau_j}{\delta}\vec e_2\right|}\,dx_2\,dx_1\\
&\quad\,+\displaystyle\int_{\partial\slab_\delta}|U(y\,;\tau_i)|\,dy\displaystyle\int_{\left\{|x_1|>\frac{4}{\delta}\right\}\cap\slab_\delta}\e^{-\frac{4}{\delta}-\left|x-y\right|}\,dx\\
&\lesssim\delta^{-\frac{3}{2}}\displaystyle\int_{\R}\e^{-\sqrt{y_1^2+(\frac{1}{\delta}-\frac{\tau_i}{\delta})^2}}\e^{-\sqrt{y_1^2+(\frac{1}{\delta}-\frac{\tau_j}{\delta})^2}}\,dy_1+\delta^{-\frac{3}{2}}\displaystyle\int_{\R}\e^{-\sqrt{y_1^2+(\frac{1}{\delta}+\frac{\tau_i}{\delta})^2}}\e^{-\sqrt{y_1^2+(\frac{1}{\delta}+\frac{\tau_j}{\delta})^2}}\,dy_1\\
&\quad\,+\delta^{\frac{1}{2}}\e^{-\frac{4}{\delta}}\displaystyle\int_{\R}\e^{-\sqrt{y_1^2+(\frac{1}{\delta}-\frac{\tau_i}{\delta})^2}}\,dy_1+\delta^{\frac{1}{2}}\e^{-\frac{4}{\delta}}\displaystyle\int_{\R}\e^{-\sqrt{y_1^2+(\frac{1}{\delta}+\frac{\tau_i}{\delta})^2}}\,dy_1\\
&\lesssim\delta^{-\frac{5}{2}}\e^{-\frac{2-\tau_i-\tau_j}{\delta}}+\delta^{-\frac{5}{2}}\e^{-\frac{2+\tau_i+\tau_j}{\delta}}.
\end{align*}
Summarizing the estimates for $I_1$ and $I_2$ gives \eqref{2024}. 

Finally, consider \eqref{2025}. Integration by parts gives
\begin{align*}
&\left|\left(\partial_{x_2}U(\,\cdot\,;\tau_i)_\mathrm{bc},U(\,\cdot\,;\tau_j)\right)_{L^2(\slab_\delta)}\right|=\left|\displaystyle\int_{\slab_\delta}\partial_{x_2}U(x\,;\tau_i)_{\mathrm{bc}}U(x\,;\tau_j)\,dx\right|\\
&\leq\left|\int_{\R}(U_i)_{\mathrm{bc}}\left(x_1,\frac{1}{\delta}\right)U_j\left(x_1,\frac{1}\delta\right)\,dx_1\right|+\left|\int_{\R}(U_i)_{\mathrm{bc}}\left(x_1,-\frac{1}{\delta}\right)U_j\left(x_1,-\frac{1}\delta\right)\,dx_1\right|\\
&\quad\,+\left|\displaystyle\int_{\slab_\delta}(U_i)_{\mathrm{bc}}(x)\partial_{x_2}U_j(x)\,dx\right|\\
&=\left|\int_{\R}U_i\left(x_1,\frac{1}{\delta}\right)U_j\left(x_1,\frac{1}\delta\right)\,dx_1\right|+\left|\int_{\R}U_i\left(x_1,-\frac{1}{\delta}\right)U_j\left(x_1,-\frac{1}\delta\right)\,dx_1\right|\\
&\quad\,+\left|\displaystyle\int_{\slab_\delta}(U_i)_{\mathrm{bc}}(x)\partial_{x_2}U_j(x)\,dx\right|\\
&=:J_1+J_2+J_3,
\end{align*}
where the boundary condition in \eqref{2142} is used.

For $J_1$, when $|x_1|\leq\frac 4\delta$, the following inequality holds
\begin{align*}
\sqrt{x_1^2+\left(\frac {1-\tau_i}\delta\right)^2}-\left|\frac {1-\tau_i}\delta\right|=\frac{x_1^2}{\sqrt{x_1^2+\left(\frac {1-\tau_i}\delta\right)^2}+\left|\frac {1-\tau_i}\delta\right|}\geq\frac{\delta x_1^2}{2(\sqrt{5}+1)}.
\end{align*}
Therefore, by Lemma \ref{prop:sign of Uy},
\begin{align*}
J_1&\lesssim\displaystyle\int_{-\frac 4\delta}^{\frac{4}{\delta}}\left(x_1^2+\left(\tfrac{1-\tau_i}\delta\right)^2\right)^{-\frac{1}{4}}\left(x_1^2+\left(\tfrac{1-\tau_j}\delta\right)^2\right)^{-\frac{1}{4}}\e^{-\left(\sqrt{x_1^2+\left(\frac {1-\tau_i}\delta\right)^2}+\sqrt{x_1^2+\left(\frac {1-\tau_j}\delta\right)^2}\right)}\,dx_1\\
&\quad\,+\displaystyle\int_{\frac{4}{\delta}}^{\infty}\left(x_1^2+\left(\tfrac {1-\tau_i}\delta\right)^2\right)^{-\frac{1}{4}} \left(x_1^2+\left(\tfrac {1-\tau_j}\delta\right)^2\right)^{-\frac{1}{4}} \e^{-2x_1}\,dx_1\\
&\lesssim\delta\displaystyle\int_{-\frac 4\delta}^{\frac{4}{\delta}}\e^{-\left(\left|\frac {1-\tau_i}\delta\right|+\left|\frac {1-\tau_j}\delta\right|+\frac{\delta x_1^2}{\sqrt{5}+1}\right)}\,dx+\delta\e^{-\frac{8}{\delta}}\lesssim\delta^{\frac{1}{2}}\e^{-\frac{2-\tau_i-\tau_j}{\delta}}.
\end{align*}
Similarly, 
$$J_2\lesssim\delta^{\frac{1}{2}}\e^{-\frac{2+\tau_i+\tau_j}{\delta}}.$$
For $J_3$, by Lemma \ref{prop:sign of Uy}, the decay estimates for $\partial_{x_2}U_j$ and $U_j$ are the same. Hence, we have
$$J_3\lesssim\left|\left(U(\,\cdot\,;\tau_i)_\mathrm{bc},U(\,\cdot\,;\tau_j)\right)_{L^2(\slab_\delta)}\right|\lesssim\delta^{-\frac{5}{2}}\e^{-\frac{2-\tau_i-\tau_j}{\delta}}+\delta^{-\frac{5}{2}}\e^{-\frac{2+\tau_i+\tau_j}{\delta}}.$$
It follows that \eqref{2025} holds.
\end{proof}
We conclude this section by constructing $(U_i)_{\mathrm{bc}}$. Since $U_i$ satisfies $\Delta U_i=\gamma(U_i)$ and $\gamma(0)=0$, $\gamma'(0)=1$, then for $U_i$ small,
$$(1-\Delta)U_i=U_i-\gamma(U_i)=O(U_i^2).$$
Thus, away from the center of the $i$-th spike where $U_i$ is exponentially small, it satisfies the following homogeneous linear equation to leading order:
$$(1-\Delta)U_i=0.$$
This observation suggests that the leading-order behavior of $(U_i)_{\mathrm{bc}}$ can be described by the corresponding linear boundary-value problem. In particular, the boundary data of $U_i$ are propagated into the strip through the Poisson kernel, and the leading-order contribution is given by the reflections of $U_i$ across the upper and lower boundaries.

Denote
\begin{equation}\label{2031}
(U_i)_{\mathrm{bc}}(x)=U_i\left(x_1,\tfrac2\delta-x_2\right)+U_i\left(x_1,-\tfrac2\delta-x_2\right)+\mathcal{R}_i(x),
\end{equation}
where $\mathcal{R}_i(x)$ satisfies
$$\begin{cases} (1-\Delta)\mathcal{R}_i(x)=-(1-\Delta)U_i\left(x_1,\frac2\delta-x_2\right)-(1-\Delta)U_i\left(x_1,-\frac2\delta-x_2\right),\quad x\in\slab_\delta,\\ \mathcal{R}_i\left(x_1,\pm\frac1\delta\right)=U_i\left(x_1,\mp\frac3\delta\right). 
\end{cases}$$
We now consider the estimates for $\mathcal{R}_i$. By the Green's function representation above
\begin{align*}
\mathcal{R}_i(x)=&-\int_{\slab_\delta}\widetilde G(x,y)
\left((1-\Delta)U_i\left(y_1,\tfrac2\delta-y_2\right)+(1-\Delta)U_i\left(y_1,-\tfrac2\delta-y_2\right)\right)\,dy\\
&+\int_{\R}\partial_{y_2}\widetilde G\left(x,\left(y_1,\tfrac{1}{\delta}\right)\right)U_i\left(y_1,-\tfrac{3}{\delta}\right)\,dy_1-\int_{\R}\partial_{y_2}\widetilde G\left(x,\left(y_1,-\tfrac{1}{\delta}\right)\right)U_i\left(y_1,\tfrac{3}{\delta}\right)\,dy_1.
\end{align*}
Applying the elliptic regularity estimate used in \eqref{2154} and the argument in Corollary \ref{C:Ubc} for $\left\|U(\,\cdot\,;\tau_i)\right\|_{H^{k_0+2} (\partial\slab_\delta)}$, we have
\begin{align*}
\left\|{R}_i\right\|_{H^{k_0+2}(\slab_\delta)}&\lesssim\left\|(1-\Delta)U_i\left(\,\cdot\,,\tfrac2\delta-\,\cdot\,\right)\right\|_{H^{k_0} (\slab_\delta)}+\left\|(1-\Delta)U_i\left(\,\cdot\,,-\tfrac2\delta-\,\cdot\,\right)\right\|_{H^{k_0} (\slab_\delta)}\\
&\quad\,+\left\|U(\,\cdot\,,-\tfrac{3+\tau_i}\delta)\right\|_{H^{k_0+2} (\R)}+\left\|U(\,\cdot\,,\tfrac{3-\tau_i}\delta)\right\|_{H^{k_0+2}(\R)}\\
&\lesssim\left\|U_i^2\left(\,\cdot\,,\tfrac2\delta-\,\cdot\,\right)\right\|_{H^{k_0} (\slab_\delta)}+\left\|U_i^2\left(\,\cdot\,,-\tfrac2\delta-\,\cdot\,\right)\right\|_{H^{k_0} (\slab_\delta)}+\e^{-\frac{3-|\tau_i|}{\delta}}.
\end{align*}
For $|x_1|\leq\frac2\delta$, $|x_2|\leq\frac1\delta$ and $\frac{|\tau_i|}\delta<\frac{1-\varepsilon}{\delta}$,
\begin{align*}
\left|x-\frac{2-\tau_i}\delta e_2\right|&=\sqrt{x_1^2+\left(x_2-\frac {2-\tau_i}\delta\right)^2}\geq\sqrt{x_1^2+\left(\frac{1-\tau_i}{\delta}\right)^2}.
\end{align*}
Since
\begin{align*}
\sqrt{x_1^2+\left(\frac{1-\tau_i}{\delta}\right)^2}-\frac{1-\tau_i}{\delta}=\frac{x_1^2}{\sqrt{x_1^2+\left(\frac{1-\tau_i}{\delta}\right)^2}+\frac{1-\tau_i}{\delta}}\geq\frac{\delta x_1^2}{2(\sqrt{2}+1)},
\end{align*}
we have
$$\left|x-\frac{2-\tau_i}\delta e_2\right|\geq\frac{1-\tau_i}{\delta}+\frac{\delta x_1^2}{2(\sqrt{2}+1)}.$$
Then from Lemma \ref{prop:sign of Uy},
\begin{align*}
\left|U_i^2\left(\,\cdot\,,\tfrac2\delta-\,\cdot\,\right)\right|^2_{H^{k_0} (\slab_\delta)}&\lesssim\int_{\slab_\delta}\left|x-\frac{2-\tau_i}\delta e_2\right|^{-2}\e^{-4\left|x-\frac{2-\tau_i}\delta e_2\right|}\,dx\\
&\lesssim\int_{\left\{|x_1|\leq\frac2\delta\right\}}\int_{-\frac1\delta}^{\frac1\delta}\left(\frac{2-\tau_i}\delta-x_2\right)^{-2}\e^{-4\left(\frac{1-\tau_i}{\delta}+\frac{\delta x_1^2}{2(\sqrt{2}+1)}\right)}\,dx_2\,dx_1\\
&\quad\,+2\int_{\frac2\delta}^{\infty}\int_{-\frac1\delta}^{\frac1\delta}\left(\frac{2-\tau_i}\delta-x_2\right)^{-2}\e^{-4x_1}\,dx_2\,dx_1\\
&\lesssim\delta^{\frac{1}{2}}\e^{-\frac{4-4\tau_i}{\delta}}+\delta\e^{-\frac{8}{\delta}}\lesssim\delta^{\frac{1}{2}}\e^{-\frac{4-4|\tau_i|}{\delta}}.
\end{align*}
Analogously,
$$\left|U_i^2\left(\,\cdot\,,-\tfrac2\delta-\,\cdot\,\right)\right|^2_{H^{k_0} (\slab_\delta)}\lesssim\delta^{\frac{1}{2}}\e^{-\frac{4-4|\tau_i|}{\delta}}.$$
Therefore,
\begin{equation}\label{2032}
\left\|{R}_i\right\|_{H^{k_0+2}(\slab_\delta)}\lesssim\delta^{\frac{1}{4}}\e^{-\frac{2-2|\tau_i|}{\delta}}.
\end{equation}

\section{Linear theory}\label{sec:Spec}
\subsection{The inner--outer gluing system}\label{sec:IO}
Motivated by \cite{MR2270164,MR2846486}, we use an inner–outer gluing method to construct the perturbation term $u$ in \eqref{definition u}. For each concentration point, we consider an inner problem corresponding to the spike $U_i$, while the outer problem describes the correction outside all spike cores and captures the interactions among the spikes.

From \eqref{definition u}, we split $u$ into an outer part and inner parts:
\begin{align}\label{1555}
u(x)=\sum_{i=1}^n\eta\left(\frac{\delta}{R}\left|x-\left(0,\frac{\tau_i}{\delta}\right)\right|\right)\phi_i\left(x\right)+\psi(x),
\end{align}
where R is a sufficiently small positive constant and $\eta(x)$ is a smooth cutoff function satisfying 
$$\eta(x) = 
\begin{cases}
    1, ~&|x|\leq1,\\
    0, ~&|x|\geq2.
\end{cases}$$
By \eqref{perturbed u eq}, for $i=1,\dots,n$, $x\in B_{i,2}:=\left\{x\in\R^2\;\left|\;\left|x-\left(0,\frac{\tau_i}{\delta}\right)\right|\leq \frac{2R}{\delta}\right.\right\}$, the inner problem is
\begin{align}\label{in}
\left(-\Delta+\gamma^\prime\left(U_i\right)\right)\phi_i + F(\boldsymbol{\tau},u,\Gamma_{\mathrm{s}})+\left(\gamma^\prime\left(\sum_{i=1}^n\sigma_iU_i\right)-1\right)\psi=0,\quad \text{for }i=1,\dots,n.
\end{align}
The outer problem is given by
\begin{align}\label{out}
\begin{cases}
&\Delta\psi(x)-\psi(x)\\
&=\left(1-\sum_{i=1}^n\eta\left(\frac{\delta}{R}\left|x-\left(0,\frac{\tau_i}{\delta}\right)\right|\right)\right)\left(\gamma^\prime\left(\sum_{i=1}^n\sigma_iU_i\right)-1\right)\psi(x)\\[3mm]
&\quad\,+\left(1-\sum_{i=1}^n\eta\left(\frac{\delta}{R}\left|x-\left(0,\frac{\tau_i}{\delta}\right)\right|\right)\right)F\\[3mm]
&\quad\,-\cfrac{2\delta}{R}\sum_{i=1}^n\nabla\eta\left(\frac{\delta}{R}\left|x-\left(0,\frac{\tau_i}{\delta}\right)\right|\right)\cdot\nabla\phi_i(x)-\cfrac{\delta^2}{R^2}\sum_{i=1}^n\Delta\eta\left(\frac{\delta}{R}\left|x-\left(0,\frac{\tau_i}{\delta}\right)\right|\right)\phi_i(x)\\[3mm]
&\quad\,+\sum_{i=1}^n\eta\left(\frac{\delta}{R}\left|x-\left(0,\frac{\tau_i}{\delta}\right)\right|\right)\left(\gamma^\prime\left(\sum_{i=1}^n\sigma_iU_i\right)-\gamma^\prime\left(U_i\right)\right)\phi_i(x),\quad x\in \slab_\delta,\\[4mm]
&\psi(x)=0,\quad x\in \partial\slab_\delta.
\end{cases}
\end{align}
\subsection{Linear theory for the inner problem}
In this subsection, we consider the linear operator corresponding to problem \eqref{in}. Note that no boundary conditions are imposed, so we only need to find a solution $\phi_i$ defined on $\slab_\delta$ for each $i$.

Define
$$L_{\tau_i} := -\Delta + \gamma^\prime\left(U_i\right).$$
Given $h\in L_{\mathrm{e}}^2(\slab_\delta)$, we consider
$$L_{\tau_i}[\phi_i] = h\quad \text{in }\slab_\delta.$$

By Assumption~\ref{assumption:B}, the kernel of $L_{\tau_i}$ in $H_{\mathrm e}^2(\mathbb R^2)$ is spanned by $\partial_{x_2} U_i$, since $\p_{x_1} U_i$ is odd in $x_1$. Let $P_i$ denote the orthogonal projection of $L_{\textrm{e}}^2(\mathbb{R}^2)$ onto $span{\{\p_{x_2} U_i\}}$ and we use the same symbol for the projection $H_{\textrm{e}}^2(\mathbb{R}^2)$ onto $span{\{\p_{x_2} U_i\}}$. The non-degeneracy result of $L_{\tau_i}$ in $\mathbb{R}^2$ implies that
$$L_{\tau_i}:(I-P_i)H_{\mathrm e}^2(\mathbb R^2)\longrightarrow(I-P_i)L_{\mathrm e}^2(\mathbb R^2)$$
is an isomorphism, with uniformly bounded inverse for $|\tau_i| < 1-\varepsilon$.

Now, let $f\in L_{\mathrm{e}}^2(\slab_\delta)$, and define its zero extension $\mathrm{E}f$ by
$$\mathrm{E}f=\begin{cases}
    f(x),&\quad x\in \slab_\delta,\\
    0,&\quad x\in \R^2\backslash\slab_\delta.
\end{cases}$$
Then $\mathrm{E}f\in L_{\mathrm e}^2(\mathbb R^2)$ and
$$\left\|\mathrm{E}f\right\|_{L^2(\R^2)}=\|f\|_{L^2(\slab_\delta)}.$$
If
$$\left(\p_{x_2} U_i,f\right)_{L^2(\slab_\delta)}=\left(\p_{x_2} U_i,\mathrm{E}f\right)_{L^2(\R^2)}=0,$$
then $\mathrm{E}f\in(I-P_i)L_{\mathrm e}^2(\mathbb R^2)$. By the nondegeneracy of $L_{\tau_i}$ in $\R^2$, there exists a unique $\phi\in(I-P_i)H_{\mathrm e}^2(\mathbb R^2)$ satisfying
$$L_{\tau_i}[\phi]=\mathrm{E}f\quad\text{in }\mathbb R^2.$$
Consequently, the restriction of $\phi$ to $\slab_\delta$ satisfies
$$L_{\tau_i}[\phi]=f\quad\text{in }\slab_\delta,$$
and the nondegeneracy estimate extends as follows.

\begin{lemma}\label{non-degeneracy proposition}
Suppose that $\|h\|_{L^2(\slab_\delta)}<+\infty$ satisfies orthogonality condition
\begin{align}\label{42}
\left(\p_{x_2} U_i\,,\,\mathrm{E}h\right)_{L^2(\R^2)}=\left(\p_{x_2} U_i\,,\,h\right)_{L^2(\slab_\delta)}=0.
\end{align}
Then there exists $\delta_0 >0$ and $c_0 > 0$ 
such that, for all $\delta \in (0,\delta_0)$ and $|\tau_i|<1-\varepsilon$, problem
$$L_{\tau_i}[\phi_i] = \mathrm{E}h$$
has a unique solution $\phi_i\in (I-P_i)H_{\textrm{e}}^2(\mathbb{R}^2)$ that satisfies the estimate
\begin{equation}\label{1512}
\|\phi_i\|_{H^2(\slab_\delta)}\leq c_0\|h\|_{L^2(\slab_\delta)}
\end{equation}
\end{lemma}

\begin{proof}
We follow the approach of \cite{MR1721719}. The existence of a solution has already been discussed above. We now prove the estimate \eqref{1512}. Since $\|h\|_{L^2(\slab_\delta)}=\|\mathrm{E}h\|_{L^2(\R^2)}$, it suffices to prove
\begin{equation}\label{1513}
\|\phi\|_{H^2(\R^2)}\leq c_0\|\mathrm{E}h\|_{L^2(\R^2)}.
\end{equation}
Suppose that \eqref{1513} is false. Then there exist sequences $\{\delta_k\}$, $\{\tau_{i,k}\}$, $\{\phi_k\}$ and $h_k$ ($k=1,2,\dots$) with $|\tau_{i,k}|<1-\varepsilon$, $\phi_k\in(I-P_{i,k})H_{\textrm{e}}^2(\mathbb{R}^2)$, $\mathrm{E}h_k\in(I-P_{i,k})L_{\mathrm e}^2(\mathbb R^2)$ such that
\begin{equation}\label{1945}
L_{\tau_{i,k}}[\phi_k] = \mathrm{E}h_k,\qquad \|h_k\|_{L^2(\slab_{\delta_k})}\to 0,
\end{equation}
\begin{equation}\label{1525}
\|\phi_k\|_{H^2(\R^2)}=1,
\end{equation}
where $P_{i,k}$ denote the orthogonal projection of $L_{\textrm{e}}^2(\mathbb{R}^2)$ onto $span\left\{\p_{x_2} U(\,\cdot-\tfrac{\tau_{i,k}}{\delta_k}e_2)\right\}$ and we use the same symbol for the projection $H_{\textrm{e}}^2(\mathbb{R}^2)$ onto $span\left\{\p_{x_2} U(\,\cdot-\tfrac{\tau_{i,k}}{\delta_k}e_2)\right\}$. 

We consider new sequences $\{\Phi_{i,k}\}$ by 
$$\Phi_{i,k}(x)=\phi_k(x+(\tau_{i,k}/\delta_k)e_2).$$
It follows from \eqref{1525} that
$$\|\Phi_{i,k}\|_{H^2(\R^2)}\leq C$$
for all $k$ sufficiently large. Therefore there exists a subsequence, again 
denoted by $\{\Phi_{i,k}\}$ which converges weakly to a limit $\Phi_{i,\infty}$ as $k\to\infty$ in $H^2(\R^2)$. We are now going to show that $\Phi_{i,\infty}\equiv 0$. As a first step we deduce 
\begin{equation}\label{2026}
\displaystyle\int_{\R^2}\Phi_{i,\infty}\partial_{x_2}U=0,\quad i=1,\dots,n.
\end{equation}
Since $\phi_k\in(I-P_{i,k})H_{\textrm{e}}^2(\mathbb{R}^2)$, this statement is shown as
\begin{align*}
\displaystyle\int_{\R^2}\Phi_{i,k}(x)\partial_{x_2}U(x)\,dx&=\displaystyle\int_{\R^2}\phi_k(x+(\tau_{i,k}/\delta_k)e_2)\partial_{x_2}U(x)\,dx\\
&=\displaystyle\int_{\R^2}\phi_k(x)\partial_{x_2}U(x-(\tau_{i,k}/\delta_k)e_2)\,dx=0.
\end{align*}
This implies \eqref{2026}.

Using $\mathcal{K}_0=span{\{\p_{x_2} U\}}$ to denote the kernel and cokernel of the linear operator $-\Delta + \gamma^\prime(U)$. Equation \eqref{2026} implies that $\Phi_{i,\infty}\in \mathcal{K}_0^\perp$. By the exponential 
decay of $U$ and \eqref{1945}, we have after possibly taking a further subsequence that
$$-\Delta\Phi_{i,\infty} + \gamma^\prime(U)\Phi_{i,\infty}=0,$$
i.e., $\Phi_{i,\infty}\in\mathcal{K}_0$. Therefore $\Phi_{i,\infty} = 0$.

Now we have
$$\Phi_{i,k}\rightharpoonup 0\quad\text{in }H^2(\R^2)\quad \text{and}\quad \Phi_{i,k}\to 0\quad\text{in }L^2(B_\rho(0))$$
as $k\to\infty$ for any $\rho>0$. By the definition of $\Phi_{i,k}$, we have
\begin{align*}
\left\|\left(\gamma^\prime\left(U(\,\cdot\,-(\tau_{i,k}/\delta_k)e_2)\right)-1\right)\phi_k\right\|_{L^2(\R^2)}\leq &\|\gamma^\prime\left(U\right)-1\|_{L^\infty(B_\rho(0))}\|\Phi_{i,k}\|_{L^2(B_\rho(0))}\\
&+\sup_{|x|>\rho}|\gamma^\prime\left(U(x)\right)-1|\|\Phi_{i,k}\|_{L^2(\R^2)}.
\end{align*}
From the estimates of $\gamma^\prime(U) = 1 + O(U)$ for small $U$ and the exponential decay of $U$, let $\rho\to\infty$, we obtain
$$\left\|\left(\gamma^\prime\left(U(\,\cdot\,-(\tau_{i,k}/\delta_k)e_2)\right)-1\right)\phi_k\right\|_{L^2(\R^2)}\to 0\quad\text{as }k\to\infty.$$
Thus,
$$\left\|\left(-\Delta+1\right)\phi_k\right\|_{L^2(\R^2)}\to 0\quad\text{as }k\to\infty.$$
By the Plancherel theorem, we have
$$|\phi_k|_{H^2(\R^2)}\leq C\int_{\mathbb R^2}
(1+|\xi|^2)^2\left|\widehat{\phi_k}(\xi)\right|^2
\,d\xi=C\left\|\left(-\Delta+1\right)\phi_k\right\|_{L^2(\R^2)}\to 0\quad\text{as }k\to\infty.$$
This contradicts the assumption
$$\left\|\phi_k\right\|_{H^2(\R^2)}= 1,$$
and the proof of Lemma \ref{non-degeneracy proposition} is completed. 
\end{proof}
Furthermore, by the interior elliptic regularity estimates (see \cite[Theorem 8.10]{MR1814364} and \cite[Theorem 11.3.2]{MR3012036}), we have higher-order Sobolev estimates and H\"older estimates: \\
If $h\in H^k(\slab_\delta)$, $k\geq 0$, then
\begin{equation}\label{322}
\|\phi_i\|_{H^{k+2}(B_{i,2})}\leq C\|h\|_{H^k(\slab_\delta)};
\end{equation}
if $h\in C^0(\slab_\delta)$, then
$$\left\|\phi_i\right\|_{C^{1,\alpha}(B_{i,2})}\leq C\left(\left\|h\right\|_{C^0(\slab_\delta)}+\left\|\phi_i\right\|_{L^2(\slab_\delta)}\right)\leq C\left(\left\|h\right\|_{C^0(\slab_\delta)}+\left\|h\right\|_{L^2(\slab_\delta)}\right),$$
where the constant $C$ is independent of $\delta$. Consequently, if $h\in H^k(\slab_\delta)$ with $k>1$, then the Sobolev embedding theorem yields
\begin{equation}\label{133}
\left\|\phi_i\right\|_{C^{1,\alpha}(B_{i,2})}\leq C\left\|h\right\|_{H^k(\slab_\delta)},\qquad 0<\alpha<\min\{1,k-1\}, 
\end{equation}
where the constant $C$ is independent of $\delta$. 

\subsection{Linear theory for the outer problem}

To solve the outer problem \eqref{out}, we consider the associated
linear problem
\begin{equation}\label{1845}
\begin{cases}
\Delta\psi-\psi=h,
&\quad x\in\slab_\delta,\\[2mm]
\psi=0,
&\quad x\in\partial\slab_\delta .
\end{cases}
\end{equation}
By the Green's function constructed in Section \ref{subsec:basic-properties-U}, the solution of \eqref{1845} can be represented as
\begin{equation}\label{134}
\psi(x)=-\int_{\slab_\delta}\widetilde G(x,\xi)h(\xi)\,d\xi.
\end{equation}
Moreover, from the exponential decay of $G_0$ in \eqref{012}, the series defining $\widetilde G$ converges absolutely and inherits the same decay properties away from the singularity. Consequently, standard elliptic estimates yield
$$\|\psi\|_{H^k(\slab_\delta)}\lesssim\|h\|_{H^{k-2}(\slab_\delta)},\qquad k\ge2.$$

\section{Solving the inner--outer gluing system} \label{Solve the inner--outer gluing system}
In this section, we use the fixed‑point argument to solve the inner--outer gluing system \eqref{in}, \eqref{out} and \eqref{perturbed Gamma eq}. This yields the functions $(\phi_1,\dots,\phi_n,\psi,\Gamma_{\mathrm{s}})$ and reduces the system to a finite-dimensional problem.

\subsection{Solving the outer problem}
Before considering the outer problem \eqref{out}, we first employ the fixed-point method to determine $\Gamma_{\mathrm{s}}$ from equation \eqref{perturbed Gamma eq}. We therefore establish estimates for $G(\boldsymbol{\tau}, 0, 0)$, $D_u G$ and $D_{\Gamma_{\mathrm{s}}}G$. In what follows, we assume $2\leq k\leq k_0$. The operators $D_u G$ and $D_{\Gamma_{\mathrm{s}}}G$ are defined respectively as
$$D_uG(\boldsymbol{\tau},u,\Gamma_{\mathrm{s}})[v]=\frac{\mathrm{d}}{\mathrm{d}t}G(\boldsymbol{\tau},u+tv,\Gamma_{\mathrm{s}})\bigg|_{t=0},\qquad v\in H^k(\slab_\delta)$$
and
$$D_{\Gamma_{\mathrm{s}}}G(\boldsymbol{\tau},u,\Gamma_{\mathrm{s}})[\Lambda_{\mathrm{s}}]=\frac{\mathrm{d}}{\mathrm{d}t}G(\boldsymbol{\tau},u,\Gamma_{\mathrm{s}}+t\Lambda_{\mathrm{s}})\bigg|_{t=0},\qquad \Lambda_{\mathrm{s}}\in H^k(\R).$$

For $G(\tau, 0, 0)$, one can directly compute 
\begin{equation*}
G(\boldsymbol{\tau}, 0, 0)
= \frac 1{2 \delta^2}\left(\p_{x_2}\omega_{\boldsymbol{\tau}}\left(\tfrac{\,\cdot\,}{\delta}, \tfrac 1\delta \right)\right)^2. 
\end{equation*}
By Corollary \ref{C:Ubc} and the Sobolev trace theorem, we have
\begin{equation*}
\left \|\p_{x_2} (U_i - (U_i)_{\mathrm{bc}}) (\,\cdot\,, \tfrac 1\delta)
-2(\p_{x_2} U_i) (\,\cdot\,, \tfrac 1\delta) \right\|_{H^k (\R)} \lesssim \|(\p_{x_2} U_i) ( \,\cdot\,, -\tfrac 3\delta)\|_{H^k (\R)} + \delta^{\frac 34} e^{-\frac {2(1-|\tau_i|)}\delta}.
\end{equation*}
From Lemma \ref{prop:sign of Uy}, we can estimate $(\p_{x_2} U_i) (\,\cdot\,, \frac 1\delta)$
\begin{equation*}
\begin{split}
\left \|(\p_{x_2} U_i) ( \,\cdot\,, \tfrac 1\delta) \right\|_{H^k (\R)}^2 \lesssim & \left(\int_0^{\delta^{-\frac 12}} + \int_{\delta^{-\frac 12}}^\infty\right) \left(x_1^2 + (\tfrac {1-\tau_i}\delta)^2 \right)^{-\frac 12} e^{-2\sqrt{x_1^2 + (\frac {1-\tau_i} \delta)^2}} \, \diff x_1\\
\lesssim & \delta^{\frac 12} e^{-\frac {2(1-\tau_i)}\delta} + \int_{\frac {((1-\tau_i)^2 +\delta)^{\frac 12}}\delta }^\infty \frac {\diff \rho} {\left(\rho^2 - (\frac {1-\tau_i}\delta)^2\right)^{\frac 12} e^{2\rho}}\lesssim \delta^{\frac 12} e^{-\frac {2(1-\tau_i)}\delta}.  
\end{split}
\end{equation*}
This implies that
\begin{equation} \label{1150}
\left\|(\p_{x_2} (U_i - (U_i)_{\mathrm{bc}})) (\,\cdot\,, \tfrac 1\delta) \right\|_{H^k (\R)} \lesssim \delta^{\frac 14} e^{-\frac {1-\tau_i}\delta}+ \delta^{\frac 34} e^{-\frac {2(1-|\tau_i|)}\delta}\quad\text{for all }i=1,\dots,n.
\end{equation}
From the one-dimensional scaling property,
$$\|f(\tfrac{\cdot}{\delta})\|_{\dot H^s (\R)} = \delta^{-s+\frac 12} \|f\|_{\dot H^s (\R)} \quad\text{for all }s\in \mathbb{R}.$$
By Lemma \ref{prop:sign of Uy}, we obtain the estimate on $G(\boldsymbol{\tau}, 0, 0)$ from the scaling property as 
\begin{equation}\label{344}
\|G(\boldsymbol{\tau},0, 0)\|_{H^{k-2} (\R)} \lesssim \delta^{\frac 12 -(k-2)-2} \left\| \left((\p_{x_2}\omega_{\boldsymbol{\tau}}) (\,\cdot\,, \tfrac 1\delta)\right)^2 \right\|_{H^{k-2} (\R)} \lesssim \delta^{1-k} e^{-\frac {2(1-|\boldsymbol{\tau}|)}\delta}.
\end{equation}

For $D_u G$, from the definition of $G$, for any $v \in H_{\mathrm{e}}^k(\slab_\delta)$
$$D_u G(\boldsymbol{\tau}, u, \Gamma_{\mathrm{s}}) v= \frac{1}{\delta^{2}} A(\Gamma_{\mathrm{s}})^{-1} \left[ \frac{(\p_{x_2} (\omega_{\boldsymbol{\tau}}+u)\p_{x_2} v)(\frac{\,\cdot\,} \delta,\frac{1}{\delta}) }{(1 +|\Diff| \coth{(2|\Diff|)} \Gamma_{\mathrm{s}})^2 + {\Gamma_{\mathrm{s}}^\prime}^2}\right]. $$
By the Sobolev product theorem, the Sobolev trace theorem and \eqref{1150}, we have 
\begin{align*}
& \left\| \left(\p_{x_2} (\omega_{\boldsymbol{\tau}}+u)\p_{x_2} v \right)({\,\cdot\,},\tfrac{1}{\delta})\right\|_{H^{k-2} (\R)} \\
&\lesssim\|\p_{x_2} (\omega_{\boldsymbol{\tau}}+u)({\,\cdot\,},\tfrac{1}{\delta})\|_{H^{k-\frac{3}{2}} (\R)} \|\p_{x_2} v ({\,\cdot\,},\tfrac{1}{\delta})\|_{H^{k-\frac{3}{2}} (\R)} \lesssim \left(\|u\|_{H^k (\slab_\delta)} + \delta^{\frac 14} e^{-\frac {1-|\boldsymbol{\tau}|}\delta} \right) \|v\|_{H^k (\slab_\delta)}.
\end{align*}
The scaling property implies the bound on $D_u G$.
\begin{equation}\label{107}
\|D_u G(\boldsymbol{\tau}, u, \Gamma_{\mathrm{s}}) v\|_{H^{k-2} (\R)}\lesssim \left(\delta^{\frac 12-k}\|u\|_{H^k (\slab_\delta)} + \delta^{\frac 34-k} e^{-\frac {1-|\boldsymbol{\tau}|}\delta} \right) \|v\|_{H^k (\slab_\delta)}.
\end{equation}
To estimate $D_{\Gamma_{\mathrm{s}}} G$, for any $\Lambda_{\mathrm{s}} \in H_{\mathrm{e}}^k(\mathbb{R})$, 
\begin{align*}
2\delta^2 D_{\Gamma_{\mathrm{s}}} G(\boldsymbol{\tau}, u, \Gamma_{\mathrm{s}}) \Lambda_{\mathrm{s}} = & \left( D_{\Gamma_{\mathrm{s}}} (A(\Gamma_{\mathrm{s}})^{-1}) \Lambda_{\mathrm{s}} \right) \left[ \frac{(\p_{x_2} (\omega_{\boldsymbol{\tau}}+u))(\frac{\,\cdot\,} \delta,\frac{1}{\delta})^2 }{(1 +|\Diff| \coth{(2|\Diff|)} \Gamma_{\mathrm{s}})^2 + {\Gamma_{\mathrm{s}}^\prime}^2}\right] \\
&-2A(\Gamma_{\mathrm{s}})^{-1}\left[ \frac{(\p_{x_2} (\omega_{\boldsymbol{\tau}}+u))(\frac{\,\cdot\,} \delta,\frac{1}{\delta})^2 }{\left( (1 +|\Diff| \coth{(2|\Diff|)} \Gamma_{\mathrm{s}})^2 + {\Gamma_{\mathrm{s}}^\prime}^2\right)^2}\right.  \\
&\left. \times \left(  (1 +|\Diff| \coth{(2|\Diff|)}\Gamma_{\mathrm{s}})|\Diff| \coth{(2|\Diff|)}  \Lambda_{\mathrm{s}} + \Gamma_{\mathrm{s}}' \Lambda_{\mathrm{s}} '\right)\right].
\end{align*}
For small $\|u\|_{H^{k}(\slab_\delta)}$ and $\|\Gamma_{\mathrm{s}}\|_{H^k (\R)}$, using the scaling property, \eqref{1150} and the Sobolev trace theorem,
\begin{equation}\label{113}
\begin{aligned}
\|D_{\Gamma_{\mathrm{s}}} G(\boldsymbol{\tau}, u, \Gamma_{\mathrm{s}}) \Lambda_{\mathrm{s}}\|_{H^{k-2} (\R)} \lesssim & \delta^{-2} \|(\p_{x_2} (\omega_{\boldsymbol{\tau}}+u))(\tfrac{\,\cdot\,} \delta,\tfrac{1}{\delta})^2\|_{H^{k-2}(\R)}  \|\Lambda_{\mathrm{s}}\|_{H^k (\R)} \\
\lesssim & \delta^{\frac 12-k} \left(\|u\|_{H^{k}(\slab_\delta)}^2 + \delta^{\frac 12} e^{-\frac {2(1-|\boldsymbol{\tau}|)}\delta} \right)  \|\Lambda_{\mathrm{s}}\|_{H^k (\R)}.
\end{aligned}
\end{equation}
Combining the above estimate, we obtain 
\begin{prop}\label{prop5.1}
For $2\leq k\leq k_0$, assume $\|u\|_{H^{k}(\slab_\delta)}\leq\delta^{k+1}$. Then there exist constants $\widetilde C>0$ and $\delta_0>0$ such that for 
$\boldsymbol{\tau} \in P_{\varepsilon}$ and $\delta\in(0,\delta_0)$, the operator
$$\mathrm{T}_s[\Gamma_{\mathrm{s}}]:=-\left(g-\alpha^2(\Diff)^2\right)^{-1}G(\boldsymbol{\tau},u,\Gamma_{\mathrm{s}})$$
has a fixed point $\widetilde{\Gamma}_{\mathrm{s}}[u]$ in the set
$$B_s:=\{\Gamma_{\mathrm{s}}\in H_{\mathrm{e}}^k(\R)\;|\;\|\Gamma_{\mathrm{s}}\|_{H^k (\R)}\leq \delta^{k}\},$$
which satisfies
$$\left\|\widetilde{\Gamma}_{\mathrm{s}}[u]\right\|_{H^k(\R)}\leq \widetilde C\left(\delta^{1-k} e^{-\frac {2(1-|\boldsymbol{\tau}|)}\delta}+\delta^{\frac 12-k}\|u\|_{H^k (\slab_\delta)}^2\right).$$
\end{prop}

\begin{proof}
We first prove that $\mathrm{T}_s[\Gamma_{\mathrm{s}}]\in B_s$ for any $\Gamma_{\mathrm{s}}\in B_s$. Applying \eqref{344}, \eqref{107} and \eqref{113}, we have
\begin{align*}
\left\|\left(g-\alpha^2(\Diff)^2\right)^{-1}G(\boldsymbol{\tau},u,\Gamma_{\mathrm{s}})\right\|_{H^k (\R)}\leq C\left(\delta^{1-k} e^{-\frac {2(1-|\boldsymbol{\tau}|)}\delta}+\delta^{k+\frac{5}{2}}\right).
\end{align*}
By choosing $\delta$ sufficiently small so that $C\left(\delta^{1-k} e^{-\frac {2(1-|\boldsymbol{\tau}|)}\delta}+\delta^{k+\frac{5}{2}}\right)\leq \delta^{k}$, we obtain $\mathrm{T}_s[\Gamma_{\mathrm{s}}]\in B_s$.

Next we prove $\mathrm{T}_s$ is a contraction on $B_s$. Applying \eqref{113} yields
\begin{align*}
\|\mathrm{T}_s[\Gamma_{\mathrm{s},1}]-\mathrm{T}_s[\Gamma_{\mathrm{s},2}]\|_{H^k (\R)}&\leq C\|D_{\Gamma_{\mathrm{s}}} G(\boldsymbol{\tau}, u, \xi) [\Gamma_{\mathrm{s},1}-\Gamma_{\mathrm{s},2}]\|_{H^{k-2} (\R)}\\
&\leq C\delta^{\frac 12-k} \left(\delta^{2k+2} + \delta^{\frac 12} e^{-\frac {2(1-|\boldsymbol{\tau}|)}\delta} \right)\|\Gamma_{\mathrm{s},1}-\Gamma_{\mathrm{s},2}\|_{H^k (\R)}.
\end{align*}
Finally, by choosing $\delta$ sufficiently small so that
$$C\delta^{\frac 12-k} \left(\delta^{2k+2} + \delta^{\frac 12} e^{-\frac {2(1-|\boldsymbol{\tau}|)}\delta} \right)<1,$$
we ensure that  $\mathrm{T}_s$ is a contraction on $B_s$. By the contraction mapping principle, $\mathrm{T}_s$ possesses a fixed point $\widetilde{\Gamma}_{\mathrm{s}}[u]$ in $B_s$.
Moreover, we have
\begin{align*}
\left\|\widetilde{\Gamma}_{\mathrm{s}}[u]\right\|_{H^k(\R)}\lesssim \|G(\boldsymbol{\tau},u, 0)\|_{H^{k-2} (\R)}&\lesssim\|G(\boldsymbol{\tau},0, 0)\|_{H^{k-2} (\R)}+\|D_u G(\boldsymbol{\tau},\xi, 0)u\|_{H^{k-2} (\R)}\\
&\leq \widetilde C\left(\delta^{1-k} e^{-\frac {2(1-|\boldsymbol{\tau}|)}\delta}+\delta^{\frac 12-k}\|u\|_{H^k (\slab_\delta)}^2\right).
\end{align*}
\end{proof}
\begin{remark}\label{rmk1}
For $v\in H^k (\slab_\delta)$, define the operator $D_u\widetilde\Gamma_{\mathrm{s}}[u]:H^k (\slab_\delta)\to H^k (\R)$ as
$$D_u\widetilde{\Gamma}_{\mathrm{s}}[u]v=\frac{\mathrm{d}}{\mathrm{d}t}\widetilde{\Gamma}_{\mathrm{s}}[u+tv]\bigg|_{t=0}.$$
Differentiating both sides of equation \eqref{perturbed Gamma eq} with respect to $u$, we obtain
$$\left(g-\alpha^2(\Diff)^2\right)D_u\widetilde{\Gamma}_{\mathrm{s}}[u]v+D_uG\left(\boldsymbol{\tau},u,\widetilde{\Gamma}_{\mathrm{s}}[u]\right)[v]+D_{\Gamma_{\mathrm{s}}}G\left(\boldsymbol{\tau},u,\widetilde{\Gamma}_{\mathrm{s}}[u]\right)\left[\widetilde{\Gamma}_{\mathrm{s}}[u]v\right]=0.$$
Using the same argument as  in Proposition \ref{prop5.1} and \eqref{107}, we obtain
\begin{equation}\label{1908}
\left\|D_u \widetilde{\Gamma}_{\mathrm{s}}[u] v\right\|_{H^{k} (\R)}\lesssim \left(\delta^{\frac 12-k}\|u\|_{H^k (\slab_\delta)} + \delta^{\frac 34-k} e^{-\frac {1-|\boldsymbol{\tau}|}\delta} \right) \|v\|_{H^k (\slab_\delta)}.
\end{equation}
\end{remark}
We now turn to the outer problem \eqref{out}. Let 
$$L_{\mathrm{out}}[\psi]:=-\Delta\psi(x)+\psi(x),$$
then the outer problem can be written as
\begin{equation}\label{121}
\begin{aligned}
\begin{cases}
L_{\mathrm{out}}[\psi]=f_{\mathrm{out}}[\phi_1,\dots,\phi_n,\psi,\Gamma_{\mathrm{s}}],\;&\text{in }\slab_\delta,\\
\psi(x)=0,\;&\text{on }\partial\slab_\delta,
\end{cases}
\end{aligned}
\end{equation}
where
\begin{align*}
&f_{\mathrm{out}}[\phi_1,\dots,\phi_n,\psi,\Gamma_{\mathrm{s}}]\\
&=\cfrac{2\delta}{R}\sum_{i=1}^n\nabla\eta\left(\frac{\delta}{R}\left|x-\left(0,\frac{\tau_i}{\delta}\right)\right|\right)\cdot\nabla\phi_i(x)+\cfrac{\delta^2}{R^2}\sum_{i=1}^n\Delta\eta\left(\frac{\delta}{R}\left|x-\left(0,\frac{\tau_i}{\delta}\right)\right|\right)\phi_i(x)\\
&\quad\;\,-\sum_{i=1}^n\eta\left(\frac{\delta}{R}\left|x-\left(0,\frac{\tau_i}{\delta}\right)\right|\right)\left(\gamma^\prime\left(\sum_{i=1}^n\sigma_iU_i\right)-\gamma^\prime\left(U_i\right)\right)\phi_i(x)\\
&\quad\;\,-\left(1-\sum_{i=1}^n\eta\left(\frac{\delta}{R}\left|x-\left(0,\frac{\tau_i}{\delta}\right)\right|\right)\right)F(\boldsymbol{\tau}, u, \Gamma_{\mathrm{s}})(x)\\
&\quad\;\,-\left(1-\sum_{i=1}^n\eta\left(\frac{\delta}{R}\left|x-\left(0,\frac{\tau_i}{\delta}\right)\right|\right)\right)\left(\gamma^\prime\left(\sum_{i=1}^n\sigma_iU_i\right)-1\right)\psi(x).
\end{align*}
then
$$\psi=(L_{\mathrm{out}})^{-1}\left(f_{\mathrm{out}}[\phi_1,\dots,\phi_n,\psi,\Gamma_{\mathrm{s}}]\right).$$

First we estimate $F(\boldsymbol{\tau}, u, \Gamma_{\mathrm{s}})$. Using $\gamma(0)=0$ and $\gamma'(0)=1$, we have $\gamma(s)=s+O(s^2)$ for $|s|$ sufficiently small. From the definition of $F$, the linear terms cancel out, then
\begin{equation*}
\begin{aligned}
F(\boldsymbol{\tau}, 0, 0)&=\gamma(\omega_{\boldsymbol{\tau}})-\sum_{i=1}^n\sigma_i\gamma(U_i)+\sum_{i=1}^n\sigma_i(U_i)_\mathrm{bc}\\
&\lesssim\sum_{i\neq j}\left|U_iU_j\right|+\sum_{i=1}^n\sum_{j=1}^n\left|U_i(U_j)_\mathrm{bc}\right|+\sum_{i=1}^n\sum_{j=1}^n\left|(U_i)_\mathrm{bc}(U_j)_\mathrm{bc}\right|.
\end{aligned}
\end{equation*}
Let $|\boldsymbol{\tau}|:=\max\{|\tau_1|,\dots,|\tau_n|\}$, $d_m:=\displaystyle\min_{1\leq i< j\leq n}|\tau_i-\tau_j|$ and $\mathcal{S}_i:=\left\{x\in\slab_\delta\;\left|\;\left|x_2-\frac {\tau_i}\delta\right|\leq\frac{d_m}{2\delta}\right. \right\}$. According to Assumption \ref{assumption:A} and Lemma \ref{prop:sign of Uy}, for any $j\ne i$, $0\le m\le k-2$ and $x \in \mathcal{S}_i$, 
$$\left|\partial^m\left(U_iU_j\right)\right|\lesssim\left(1+\left|x-\frac{\tau_i}\delta e_2\right|\right)^{-\frac 12}\left|x-\frac{\tau_j}\delta e_2\right|^{-\frac 12}\e^{-\left(\left|x-\frac {\tau_i}\delta e_2\right|+\left|x-\frac{\tau_j}\delta e_2\right|\right)}.$$
For $|x_1|\leq\frac 4\delta$ and $\frac{|\tau_i-\tau_j|}\delta\leq\frac{2}{\delta}$, from the convexity of the function $f(x)=\sqrt{a^2+x^2}$
\begin{align*}
\left|x-\frac{\tau_i}\delta e_2\right|+\left|x-\frac{\tau_j}\delta e_2\right|&=\sqrt{x_1^2+\left(x_2-\frac {\tau_i}\delta\right)^2}+\sqrt{x_1^2+\left(x_2-\frac {\tau_j}\delta\right)^2}\\
&\geq\sqrt{4x_1^2+\left(\frac {\tau_i-\tau_j}{\delta}\right)^2}.
\end{align*}
Since
\begin{align*}
\sqrt{4x_1^2+\left(\frac{\tau_i-\tau_j}{\delta}\right)^2}-\left|\frac{\tau_i-\tau_j}{\delta}\right|=\frac{4x_1^2}{\sqrt{4x_1^2+\left(\frac{\tau_i-\tau_j}{\delta}\right)^2}+\left|\frac{\tau_i-\tau_j}{\delta}\right|}\geq\frac{2\delta x_1^2}{\sqrt{17}+1},
\end{align*}
we have
$$\left|x-\frac{\tau_i}\delta e_2\right|+\left|x-\frac{\tau_j}\delta e_2\right|\geq\left|\frac{\tau_i-\tau_j}{\delta}\right|+\frac{2\delta x_1^2}{\sqrt{17}+1}.$$
Thus,
\begin{align*}
&\int_{\mathcal{S}_i}\left(1+\left|x-\tfrac{\tau_i}\delta e_2\right|\right)^{-1} \left|x-\tfrac{\tau_j}\delta e_2\right|^{-1} \e^{-2\left(\left|x-\frac {\tau_i}\delta e_2\right|+\left|x-\frac{\tau_j}\delta e_2\right|\right)}\,dx\\
&\lesssim\int_{\left\{|x_1|\leq\frac 4\delta\right\}}\int_{\frac 1\delta(\tau_i-\frac{d_m}{2})}^{\frac 1\delta(\tau_i+\frac{d_m}{2})}\left(1+\left|x-\tfrac{\tau_i}\delta e_2\right|\right)^{-1} \left|x-\tfrac{\tau_j}\delta e_2\right|^{-1}\e^{-2\left(\frac{|\tau_i-\tau_j|}\delta+\frac{2}{\sqrt{17}+1} \delta x_1^2\right)}\,dx_2\,dx_1\\
&\quad\,+\int_{\left\{|x_1|\geq\frac 4\delta\right\}}\int_{\frac 1\delta(\tau_i-\frac{d_m}{2})}^{\frac 1\delta(\tau_i+\frac{d_m}{2})}\left(1+\left|x-\tfrac{\tau_i}\delta e_2\right|\right)^{-1} \left|x-\tfrac{\tau_j}\delta e_2\right|^{-1}\e^{-2\left(\frac{4}\delta+\left|x-\frac{\tau_j}\delta e_2\right|\right)}\,dx_2\,dx_1\\
&\lesssim\int_{\left\{|x_1|\leq\frac 4\delta\right\}}\int_0^{\frac{d_m}{2\delta}} \delta(1+x_2)^{-1}\e^{-2\left(\frac{|\tau_i-\tau_j|}\delta+\frac{2}{\sqrt{17}+1} \delta x_1^2\right)}\,dx_2\,dx_1+\delta\e^{-\frac{8}\delta}\\
&\lesssim\delta^{\frac 12}\left|\log\delta\right|\e^{-\frac{2|\tau_i-\tau_j|}{\delta}}\lesssim\delta^{\frac 12}\left|\log\delta\right|\e^{-\frac{2d_m}{\delta}}.
\end{align*}
This implies
\begin{equation}\label{E:F0}
\|U_iU_j\|_{H^{k-2}(\mathcal{S}_i)}^2\lesssim\delta^{\frac 12}|\log\delta|e^{-\frac{2d_m}\delta}.
\end{equation}
By \eqref{2031}, \eqref{2032} and the derivation of \eqref{E:F0}, we have
\begin{align*}
&\|U_i(U_j)_\mathrm{bc}\|_{H^{k-2}(\mathcal{S}_i)}\\
&\lesssim\|U_iU_j\left(\,\cdot\,,\tfrac2\delta-\,\cdot\,\right)\|_{H^{k-2}(\mathcal{S}_i)}+\|U_iU_j\left(\,\cdot\,,-\tfrac2\delta-\,\cdot\,\right)\|_{H^{k-2}(\mathcal{S}_i)}+\|U_i\mathcal{R}_j\|_{H^{k-2}(\mathcal{S}_i)}\\
&\lesssim\delta^{\frac 14}\left|\log\delta\right|^{\frac{1}{2}}\e^{-\frac{2-\tau_i-\tau_j}\delta}+\delta^{\frac 14}\left|\log\delta\right|^{\frac{1}{2}}\e^{-\frac{2+\tau_i+\tau_j}\delta}+\delta^{\frac{1}{4}}\e^{-\frac{2-2|\tau_j|}{\delta}}\lesssim\delta^{\frac 14}\left|\log\delta\right|^{\frac{1}{2}}\e^{-\frac{2-2|\boldsymbol{\tau}|}{\delta}}.
\end{align*}
The remaining terms can be estimated similarly. We therefore obtain 
\begin{equation}\label{1551}
\|F(\boldsymbol{\tau},0,0)\|_{H^{k-2}(\slab_\delta)}\lesssim \delta^{\frac 14}\left|\log\delta\right|^{\frac{1}{2}}\e^{-\frac{d_m}\delta}+\delta^{\frac 14}\left|\log\delta\right|^{\frac{1}{2}}\e^{-\frac {2(1-|\boldsymbol{\tau}|)}\delta}.
\end{equation}

Next we consider $D_uF$ and $D_{\Gamma_{\mathrm{s}}}F$. Here the operators $D_uF$ and $D_{\Gamma_{\mathrm{s}}}F$ are defined respectively as
$$D_uF(\boldsymbol{\tau},u,\Gamma_{\mathrm{s}})[v]=\frac{\mathrm{d}}{\mathrm{d}t}F(\boldsymbol{\tau},u+tv,\Gamma_{\mathrm{s}})\bigg|_{t=0},\qquad v\in H^k(\slab_\delta)$$
and
$$D_{\Gamma_{\mathrm{s}}}F(\boldsymbol{\tau},u,\Gamma_{\mathrm{s}})[\Lambda_{\mathrm{s}}]=\frac{\mathrm{d}}{\mathrm{d}t}F(\boldsymbol{\tau},u,\Gamma_{\mathrm{s}}+t\Lambda_{\mathrm{s}})\bigg|_{t=0},\qquad \Lambda_{\mathrm{s}}\in H^k(\R).$$
For any $v \in H_{\mathrm{e}}^k(\slab_\delta)$,
\begin{equation*}
\begin{split}
&D_u F(\boldsymbol{\tau}, u, \Gamma_{\mathrm{s}}) v= \left(\left|1+\Gamma^\prime(\delta \,\cdot\, )\right|^2 \gamma'(\omega_{\boldsymbol{\tau}}+u)- \gamma^\prime\left(\sum_{i=1}^n\sigma_iU_i\right)\right) v \\
&=\left(2 \Gamma_1' (\delta \,\cdot\,)+ |\Gamma'(\delta \,\cdot\,)|^2\right)\gamma'(\omega_{\boldsymbol{\tau}}+u)v\\
&\quad\,\,+ \left(\left(u - \sum_{i=1}^n\sigma_i(U_i)_\mathrm{bc}\right)\displaystyle\int_0^1 \gamma''\left(\sum_{i=1}^n\sigma_iU_i + s\left(u - \sum_{i=1}^n\sigma_i(U_i)_\mathrm{bc}\right) \right) \, \diff s\right) v. 
\end{split}
\end{equation*}
By the scaling identity $\| f(\delta \,\cdot\,) \|_{\dot H^s(\slab_\delta)} = \delta^{s-1} \| f \|_{\dot H^s(\slab_1)}$ and the Sobolev trace theorem, we have
$$\left \| 2 \Gamma_1^\prime(\delta \,\cdot\,) + |\Gamma^\prime(\delta \,\cdot\,) |^2 \right\|_{H^{k-\frac{1}{2}}(\slab_\delta)} \lesssim \delta^{-1} \|\Gamma_{\mathrm{s}}\|_{H^k(\mathbb{R})}.$$
Since $H^l(\slab_\delta)$ is an algebra for $l>1$ and $k\geq2$,
\begin{equation}\label{141}
\begin{aligned}
\|D_u F(\boldsymbol{\tau}, u, \Gamma_{\mathrm{s}}) v\|_{H^{k-2} (\slab_\delta)}&\lesssim\|D_u F(\boldsymbol{\tau}, u, \Gamma_{\mathrm{s}}) v\|_{H^{k-\frac{1}{2}} (\slab_\delta)}\\
&\lesssim\left( \delta^{-1} \|\Gamma_{\mathrm{s}}\|_{H^{k}(\R)} + \|u\|_{H^{k}  (\slab_\delta)} + \sum_{i=1}^n\|(U_i)_\mathrm{bc}\|_{H^{k}  (\slab_\delta)}\right) \|v\|_{H^k (\slab_\delta)}.
\end{aligned}
\end{equation}
Similarly, given any $\Lambda_{\mathrm{s}} \in H_{\mathrm{e}}^k(\mathbb{R})$, we have
$$D_{\Gamma_{\mathrm{s}}} F (\boldsymbol{\tau}, u, \Gamma_{\mathrm{s}})\Lambda_{\mathrm{s}} = 2 \left( \Lambda_1' (\delta \,\cdot\,) + \Gamma' (\delta \,\cdot\,) \cdot \Lambda' (\delta \,\cdot\,) \right) \gamma (\omega_{\boldsymbol{\tau}}+u),$$
where $\Lambda = \Lambda_1 + \I \Lambda_2$ is the complex holomorphic function obtained by $\Lambda_{\mathrm{s}}$. By the Sobolev trace theorem and the scaling identity above, 
\begin{equation}\label{142}
\|D_{\Gamma_{\mathrm{s}}} F (\boldsymbol{\tau}, u, \Gamma_{\mathrm{s}}) \Lambda_{\mathrm{s}}\|_{H^{k-2} (\slab_\delta)}\lesssim\|D_{\Gamma_{\mathrm{s}}} F (\boldsymbol{\tau}, u, \Gamma_{\mathrm{s}}) \Lambda_{\mathrm{s}}\|_{H^{k-\frac{1}{2}} (\slab_\delta)}\lesssim\delta^{-1} \|\Lambda_{\mathrm{s}}\|_{H^{k}(\R)}.
\end{equation}
Since $k_0\geq4$, Lemma 3.8 in \cite{MR4577959} implies that, for any integer $2\leq l\leq k_0$, 
$$F(\boldsymbol{\tau},u,\Gamma_{\mathrm{s}})\in H_{\mathrm{e}}^{l-1}(\slab_\delta),\qquad G(\boldsymbol{\tau},u,\Gamma_{\mathrm{s}})\in H_{\mathrm{e}}^{l-\frac{3}{2}}(\R).$$
Taking $l=3$ and applying the Sobolev embedding, we obtain
\begin{equation}\label{1536}
\|F(\boldsymbol{\tau},u,\Gamma_{\mathrm{s}})\|_{L^\infty(\slab_\delta)}\lesssim\|F(\boldsymbol{\tau},u,\Gamma_{\mathrm{s}})\|_{H^2(\slab_\delta)},\quad \|G(\boldsymbol{\tau},u,\Gamma_{\mathrm{s}})\|_{L^\infty(\R)}\lesssim\|G(\boldsymbol{\tau},u,\Gamma_{\mathrm{s}})\|_{H^{\frac{3}{2}}(\R)}.
\end{equation}
With these estimates, given solutions $\phi_i$, $i=1,\dots,n$ of the inner problems, we can find $\psi$ solving \eqref{121}.
\begin{prop}\label{prop5.2}
For $2\leq k\leq k_0$, assume $\|\phi_i\|_{H^{k}(B_{i,2})}\leq\delta^{k+1}$ for all $i=1,\dots,n$. Then there exist constants $\widetilde C>0$ and $\delta_0>0$ such that for 
$\boldsymbol{\tau} \in P_{\varepsilon}$ and $\delta\in(0,\delta_0)$, the operator
$$\mathrm{T}_{\mathrm{out}}[\psi]:=(L_{\mathrm{out}})^{-1}f_{\mathrm{out}}[\phi_1,\dots,\phi_n,\psi,\widetilde{\Gamma}_{\mathrm{s}}[u]]$$
has a fixed point $\widetilde{\psi}[\phi_1,\dots,\phi_n]$ in the set $$B_{\mathrm{out}}:=\{\psi \in H_{\mathrm{e}}^k(\slab_\delta) \;|\;\|\psi\|_{H^k(\slab_\delta)}\leq \delta^{k+1}\},$$
which satisfies
\begin{align*}
\left\|\widetilde{\psi}[\phi_1,\dots,\phi_n]\right\|_{H^{k}(\slab_\delta)}\lesssim\delta^{\frac{1}{2}}\sum_{i=1}^n\left\|\phi_i\right\|_{H^{k}(B_{i,2})}+\delta^{\frac 14} | \log \delta|^{\frac{1}{2}}e^{-\frac {d_m}\delta}+\delta^{-k}e^{-\frac {2(1-|\boldsymbol{\tau}|)}\delta}.
\end{align*}
Moreover, if in addition $\left\|\phi_i\right\|_{C^{1,\alpha}(B_{i,2})}\leq\delta^{k+1}$ for all $i=1,\dots,n$, then the following pointwise estimate holds:
{\small\begin{align*}
&|\widetilde{\psi}[\phi_1,\dots,\phi_n](x)|\\
&\leq \widetilde C \begin{cases}
\delta^{-3/2}\e^{-\frac{R}{2\delta}}\left(e^{-\frac {d_m}\delta}+\delta^{-k}e^{-\frac {2(1-|\boldsymbol{\tau}|)}\delta}+\delta^{-\frac{1}{2}-k}\sum_{i=1}^n\|\phi_i\|_{H^k(B_{i,2})}^2\right)\\
+\delta^{-3/2}\e^{-\frac{R}{2\delta}}\delta\sum_{i=1}^n\left\|\phi_i\right\|_{C^{1,\alpha}(B_{i,2})}+\e^{-\frac{d-2R}{\delta}}\sum_{i=1}^n\left\|\phi_i\right\|_{C^{1,\alpha}(B_{i,2})},&x\in\cup_{i=1}^nB_{i,\frac{1}{2}},\\[4mm]
e^{-\frac {d_m}\delta}+\delta^{-k}e^{-\frac {2(1-|\boldsymbol{\tau}|)}\delta}+\delta^{-\frac{1}{2}-k}\sum_{i=1}^n\|\phi_i\|_{H^k (B_{i,2})}^2\\
+\delta\sum_{i=1}^n\left\|\phi_i\right\|_{C^{1,\alpha}(B_{i,2})}+\e^{-\frac{d-2R}{\delta}}\sum_{i=1}^n\left\|\phi_i\right\|_{C^{1,\alpha}(B_{i,2})},&x\in \slab_\delta\backslash \cup_{i=1}^nB_{i,\frac{1}{2}},
\end{cases}
\end{align*}}

\noindent
where $d:=\min\{2-2|\boldsymbol{\tau}|,d_m\}$, and $B_{i,\frac{1}{2}}:=\left\{x\in\R^2\;\left|\;\left|x-\left(0,\frac{\tau_i}{\delta}\right)\right|\leq \frac{R}{2\delta}\right.\right\}$.
\end{prop}

\begin{proof}
We first prove that $\mathrm{T}_{\mathrm{out}}[\psi]\in B_{\mathrm{out}}$ for any $\psi\in B_{\mathrm{out}}$. From the invertibility of $L_{\mathrm{out}}$,
\begin{align*}
&\|\mathrm{T}_{\mathrm{out}}[\psi]\|_{H^k (\slab_\delta)}\lesssim\left\|f_{\mathrm{out}}[\phi_1,\dots,\phi_n,\psi,\widetilde{\Gamma}_{\mathrm{s}}[u]]\right\|_{H^{k-2} (\slab_\delta)}\\
&\lesssim\delta\sum_{i=1}^n\left\|\nabla\eta\left(\frac{\delta}{R}\left|x-\left(0,\frac{\tau_i}{\delta}\right)\right|\right)\cdot\nabla\phi_i(x)\right\|_{H^{k-2} (\slab_\delta)}\\
&\quad\;+\delta^2\sum_{i=1}^n\left\|\Delta\eta\left(\frac{\delta}{R}\left|x-\left(0,\frac{\tau_i}{\delta}\right)\right|\right)\phi_i(x)\right\|_{H^{k-2} (\slab_\delta)}\\
&\quad\;+\sum_{i=1}^n\left\|\eta\left(\frac{\delta}{R}\left|x-\left(0,\frac{\tau_i}{\delta}\right)\right|\right)\left(\gamma^\prime\left(\sum_{j=1}^n\sigma_jU_j\right)-\gamma^\prime\left(U_i\right)\right)\phi_i(x)\right\|_{H^{k-2} (\slab_\delta)}\\
&\quad\;+\left\|\left(1-\sum_{i=1}^n\eta\left(\frac{\delta}{R}\left|x-\left(0,\frac{\tau_i}{\delta}\right)\right|\right)\right)F(\boldsymbol{\tau}, u, \Gamma_{\mathrm{s}})(x)\right\|_{H^{k-2} (\slab_\delta)}\\
&\quad\;+\left\|\left(1-\sum_{i=1}^n\eta\left(\frac{\delta}{R}\left|x-\left(0,\frac{\tau_i}{\delta}\right)\right|\right)\right)\left(\gamma^\prime\left(\sum_{j=1}^n\sigma_jU_j\right)-1\right)\psi(x)\right\|_{H^{k-2} (\slab_\delta)}.
\end{align*}
Since the support of $\eta\left(\frac{\delta}{R}\left|x-\left(0,\frac{\tau_i}{\delta}\right)\right|\right)$ is contained in the ball
$$B_{i,2}=\left\{x\in\R^2\;\left|\;\left|x-\left(0,\frac{\tau_i}{\delta}\right)\right|\leq \frac{2R}{\delta}\right.\right\},$$
using the evenness of $\gamma'$ and the exponential decay of
$U$ and $U_{\mathrm{bc}}$, we have
$$\gamma^\prime\left(\sum_{j=1}^n\sigma_jU_j\right)-\gamma^\prime\left(U_i\right)=\gamma^\prime\left(\sum_{j=1}^n\sigma_jU_j\right)-\gamma^\prime\left(\sigma_iU_i\right)\lesssim
\sum_{j\neq i}|U_j|\quad \text{in }B_{i,2}.$$
Then by Lemma \ref{prop:sign of Uy},
\begin{equation}\label{338}
\begin{aligned}
&\left\|\eta\left(\frac{\delta}{R}\left|x-\left(0,\frac{\tau_i}{\delta}\right)\right|\right)\left(\gamma^\prime\left(\sum_{j=1}^n\sigma_jU_j\right)-\gamma^\prime\left(U_i\right)\right)\phi_i\right\|_{H^{k-2} (\slab_\delta)}\\
&\lesssim\left(\sum_{j\neq i}\left\|U_j\right\|_{H^{k} (B_{i,2})}\right)\left\|\phi_i\right\|_{H^{k}(B_{i,2})}\lesssim\e^{-\frac{d-2R}{\delta}}\left\|\phi_i\right\|_{H^{k}(B_{i,2})},
\end{aligned}
\end{equation}
where $d:=\min\{2-2|\boldsymbol{\tau}|,d_m\}$ and take $R$ small such that $d-2R>0$.

Similarly, since $1-\sum_{i=1}^n\eta\left(\frac{\delta}{R}\left|\xi-\left(0,\frac{\tau_i}{\delta}\right)\right|\right)=0$ in all ball
$$B_{i,1}:=\left\{x\in\R^2\;\left|\;\left|x-\left(0,\frac{\tau_i}{\delta}\right)\right|\leq \frac{R}{\delta}\right.\right\},\quad i=1,\dots,n,$$
using
$$\gamma^\prime\left(\sum_{j=1}^n\sigma_jU_j\right)-1\lesssim\sum_{j=1}^n|U_j|\quad\text{in }\slab_\delta\backslash\cup_{i=1}^nB_{i,1},$$
we have
\begin{equation}\label{335}
\left\|\left(1-\sum_{i=1}^n\eta\left(\frac{\delta}{R}\left|x-\left(0,\frac{\tau_i}{\delta}\right)\right|\right)\right)\left(\gamma^\prime\left(\sum_{j=1}^n\sigma_jU_j\right)-1\right)\psi(x)\right\|_{H^{k-2} (\slab_\delta)}\lesssim\e^{-\frac {R}\delta}\left\|\psi\right\|_{H^{k}(\slab_\delta)}.
\end{equation}

Thus, by \eqref{1551}--\eqref{142} and Proposition \ref{prop5.1}
$$\|\mathrm{T}_{\mathrm{out}}[\psi]\|_{H^k (\slab_\delta)}\lesssim\delta\sum_{i=1}^n\left\|\phi_i\right\|_{H^{k}(B_{i,2})}+\delta^{\frac 14} | \log \delta|^{\frac{1}{2}}e^{-\frac {d_m}\delta}+\delta^{-k}e^{-\frac {2(1-|\boldsymbol{\tau}|)}\delta}+\delta^{-\frac{1}{2}-k}\|u\|_{H^k (\slab_\delta)}^2.$$
Combining the above estimates and $\|\phi_i\|_{H^{k}(B_{i,2})}\leq\delta^{k+1}$ for all $i=1,\dots,n$, by choosing $\delta$ sufficiently small, we obtain $\mathrm{T}_{\mathrm{out}}[\psi]\in B_{\mathrm{out}}$.

Next we prove $\mathrm{T}_{\mathrm{out}}$ is a contraction on $B_{\mathrm{out}}$. By \eqref{1908}, \eqref{141}, \eqref{142} and \eqref{335}
\begin{align*}
&\|\mathrm{T}_{\mathrm{out}}[\psi_1]-\mathrm{T}_{\mathrm{out}}[\psi_2]\|_{H^{k}(\slab_\delta)}\\
&\lesssim\left\|f_{\mathrm{out}}[\phi_1,\dots,\phi_n,\psi_1,\widetilde{\Gamma}_{\mathrm{s}}[u]]-f_{\mathrm{out}}[\phi_1,\dots,\phi_n,\psi_2,\widetilde{\Gamma}_{\mathrm{s}}[u]]\right\|_{H^{k-2}(\slab_\delta)}\\
&\lesssim\left\|D_u F(\boldsymbol{\tau}, \zeta, \widetilde{\Gamma}_{\mathrm{s}})(\psi_1-\psi_2)\right\|_{H^{k-2}(\slab_\delta)}+\left\|D_{\Gamma_{\mathrm{s}}} F (\boldsymbol{\tau}, u, \zeta)D_u \widetilde{\Gamma}_{\mathrm{s}}[\zeta](\psi_1-\psi_2)\right\|_{H^{k-2}(\slab_\delta)}\\
&\quad\;+\left\|\left(1-\sum_{i=1}^n\eta\left(\frac{\delta}{R}\left|\xi-\left(0,\frac{\tau_i}{\delta}\right)\right|\right)\right)\left(\gamma^\prime\left(\sum_{j=1}^n\sigma_jU_j\right)-1\right)(\psi_1-\psi_2)\right\|_{H^{k-2}(\slab_\delta)}\\
&\lesssim\left(\delta^{-1} \left\|\widetilde{\Gamma}_{\mathrm{s}}\right\|_{H^{k}(\R)} + \|u\|_{H^{k}  (\slab_\delta)} + \sum_{i=1}^n\|(U_i)_\mathrm{bc}\|_{H^{k}  (\slab_\delta)}+e^{-\frac {R}\delta}\right)\|\psi_1-\psi_2\|_{H^{k}(\slab_\delta)}\\
&\quad\;+\delta^{-1} \left\|\widetilde{\Gamma}_{\mathrm{s}}\right\|_{H^{k}(\R)}\left(\delta^{\frac 12-k}\|u\|_{H^k (\slab_\delta)} + \delta^{\frac 34-k} e^{-\frac {1-|\boldsymbol{\tau}|}\delta} \right)\|\psi_1-\psi_2\|_{H^{k}  (\slab_\delta)}.
\end{align*}
From Corollary \ref{C:Ubc} and Proposition \ref{prop5.1}, by choosing $\delta$ sufficiently small, we find that $\mathrm{T}_{\mathrm{out}}$ is a contraction mapping on $B_{\mathrm{out}}$. Hence $\mathrm{T}_{\mathrm{out}}$ possesses a fixed point $\widetilde{\psi}[\phi_1,\dots,\phi_n]$ in $B_{\mathrm{out}}$. Moreover, combining the above estimates we have
\begin{align*}
\left\|\widetilde{\psi}[\phi_1,\dots,\phi_n]\right\|_{H^{k}(\slab_\delta)}&\lesssim \left\|f_{\mathrm{out}}[\phi_1,\dots,\phi_n,0,\widetilde{\Gamma}_{\mathrm{s}}[\phi_1,\dots,\phi_n,0]]\right\|_{H^{k-2}(\slab_\delta)}\\
&\lesssim\delta^{\frac{1}{2}}\sum_{i=1}^n\left\|\phi_i\right\|_{H^{k}(B_{i,2})}+\delta^{\frac 14} | \log \delta|^{\frac{1}{2}}e^{-\frac {d_m}\delta}+\delta^{-k}e^{-\frac {2(1-|\boldsymbol{\tau}|)}\delta}.
\end{align*}

Finally, we consider a pointwise estimate for $\widetilde{\psi}[\phi_1,\dots,\phi_n]$. Applying \eqref{134} we have
\begin{align}
&\left|\widetilde{\psi}[\phi_1,\dots,\phi_n](x)\right|=\left|(L_{\mathrm{out}})^{-1}f_{\mathrm{out}}[\phi_1,\dots,\phi_n,\widetilde{\psi},\widetilde{\Gamma}_{\mathrm{s}}[u]](x)\right|\notag\\
&\lesssim\delta\sum_{i=1}^n\left|\displaystyle\int_{\slab_\delta}\widetilde G(x,\xi)\nabla\eta\left(\frac{\delta}{R}\left|\xi-\left(0,\frac{\tau_i}{\delta}\right)\right|\right)\cdot\nabla\phi_i(\xi)\;d\xi\right|\label{5.11}\\
&\quad\;+\delta^2\sum_{i=1}^n\left|\displaystyle\int_{\slab_\delta}\widetilde G(x,\xi)\Delta\eta\left(\frac{\delta}{R}\left|\xi-\left(0,\frac{\tau_i}{\delta}\right)\right|\right)\phi_i(\xi)\;d\xi\right|\label{5.12}\\
&\quad\;+\sum_{i=1}^n\left|\displaystyle\int_{\slab_\delta}\widetilde G(x,\xi)\eta\left(\frac{\delta}{R}\left|\xi-\left(0,\frac{\tau_i}{\delta}\right)\right|\right)\left(\gamma^\prime\left(\sum_{i=1}^n\sigma_iU_i\right)-\gamma^\prime\left(U_i\right)\right)\phi_i(\xi)\;d\xi\right|\label{5.13}\\
&\quad\;+\left|\displaystyle\int_{\slab_\delta}\widetilde G(x,\xi)\left(1-\sum_{i=1}^n\eta\left(\frac{\delta}{R}\left|\xi-\left(0,\frac{\tau_i}{\delta}\right)\right|\right)\right)F(\boldsymbol{\tau}, u, \Gamma_{\mathrm{s}})(\xi)\;d\xi\right|\label{5.14}.
\end{align}

Consider \eqref{5.11}, since the support of $\nabla\eta\left(\frac{\delta}{R}\left|x-\left(0,\frac{\tau_i}{\delta}\right)\right|\right)$ is contained in the annulus
$$A_i:=\left\{x\in\R^2\;\left|\;\frac{R}{\delta}\leq\left|x-\left(0,\frac{\tau_i}{\delta}\right)\right|\leq \frac{2R}{\delta}\right.\right\},$$
by \eqref{2114} and \eqref{133},
\begin{align*}
&\left|\displaystyle\int_{\slab_\delta}\widetilde G(x,\xi)\nabla\eta\left(\frac{\delta}{R}\left|\xi-\left(0,\frac{\tau_i}{\delta}\right)\right|\right)\cdot\nabla\phi_i(\xi)\;d\xi\right|\lesssim\displaystyle\int_{A_i}\left|\widetilde G(x,\xi)\right|\left|\nabla\phi_i(\xi)\right|\;d\xi\\
&\lesssim\left\|\phi_i\right\|_{C^{1,\alpha}(B_{i,2})}\displaystyle\int_{A_i}\sum_{m\in\mathbb Z}\left|G_0\left(x_1-y_1,x_2-y_2+\frac{4m}{\delta}\right)\right|\;d\xi\\
&\quad\,+\left\|\phi_i\right\|_{C^{1,\alpha}(B_{i,2})}\displaystyle\int_{A_i}\sum_{m\in\mathbb Z}\left|G_0\left(x_1-y_1,x_2+y_2-\frac{2}{\delta}+\frac{4m}{\delta}
\right)\right|\;d\xi.
\end{align*}
For $x\in B_{i,\frac{1}{2}}$, $\xi\in A_i$, we have $\frac{R}{2\delta}\leq|x-\xi|\leq\frac{5R}{2\delta}$, then 
$$\left(x_2-\xi_2+\frac{4m}{\delta}\right)^2-(x_2-\xi_2)^2\geq\frac{16m^2}{\delta^2}-\frac{20|m|R}{\delta^2},\qquad m\in\mathbb{Z}.$$
Since $R>0$ small, it follows that
\begin{align*}
&\sqrt{(x_1-\xi_1)^2+\left(x_2-\xi_2+\frac{4m}{\delta}\right)^2}-\sqrt{(x_1-\xi_1)^2+\left(x_2-\xi_2\right)^2}\\
&\geq\sqrt{(x_1-\xi_1)^2+\left(x_2-\xi_2\right)^2+\frac{16m^2}{\delta^2}-\frac{20|m|R}{\delta^2}}-\sqrt{(x_1-\xi_1)^2+\left(x_2-\xi_2\right)^2}\\
&=\frac{16m^2-20|m|R}{\delta^2\left(\sqrt{(x_1-\xi_1)^2+\left(x_2-\xi_2\right)^2+\frac{16m^2}{\delta^2}-\frac{20|m|R}{\delta^2}}+\sqrt{(x_1-\xi_1)^2+\left(x_2-\xi_2\right)^2}\right)}\geq\frac{16m^2}{5\delta R}.
\end{align*}
Using the above inequality and \eqref{012}, if $x\in B_{i,\frac{1}{2}}$,
\begin{align*}
\displaystyle\int_{A_i}\sum_{m\in\mathbb Z}\left|G_0\left(x_1-y_1,x_2-y_2+\frac{4m}{\delta}\right)\right|\;d\xi&\lesssim\delta^{\frac{1}{2}}\displaystyle\int_{A_i}\sum_{m\in\mathbb Z}\e^{-\sqrt{(x_1-\xi_1)^2+\left(x_2-\xi_2+\frac{4m}{\delta}\right)^2}}\;d\xi\\
&\lesssim\delta^{\frac{1}{2}}\displaystyle\int_{A_i}\e^{-|x-\xi|}\sum_{m\in\mathbb Z}\e^{-\frac{16m^2}{5\delta R}}\;d\xi\\
&\lesssim\delta^{-\frac{3}{2}}\e^{-\frac{R}{2\delta}}.
\end{align*}
Similarly,
$$\displaystyle\int_{A_i}\sum_{m\in\mathbb Z}\left|G_0\left(x_1-y_1,x_2+y_2-\frac{2}{\delta}+\frac{4m}{\delta}
\right)\right|\;d\xi\lesssim\delta^{-\frac{3}{2}}\e^{-\frac{R}{2\delta}}.$$
For $x\in \slab_\delta\backslash B_{i,\frac{1}{2}}$, by absolute integrability of $\widetilde G$, we finally obtain
{\small$$\left|\displaystyle\int_{\slab_\delta}\widetilde G(x,\xi)\nabla\eta\left(\frac{\delta}{R}\left|\xi-\left(0,\frac{\tau_i}{\delta}\right)\right|\right)\cdot\nabla\phi_i(\xi)\;d\xi\right|\lesssim
\begin{cases}
\delta^{-3/2}\e^{-\frac{R}{2\delta}}\left\|\phi_i\right\|_{C^{1,\alpha}(B_{i,2})},&x\in B_{i,\frac{1}{2}},\\
\left\|\phi_i\right\|_{C^{1,\alpha}(B_{i,2})},&x\in \slab_\delta\backslash B_{i,\frac{1}{2}}.
\end{cases}$$}

\noindent
Using the same method as in \eqref{5.11}, for \eqref{5.12}
\begin{align*}
\left|\displaystyle\int_{\slab_\delta}\widetilde G(x,\xi)\Delta\eta\left(\frac{\delta}{R}\left|\xi-\left(0,\frac{\tau_i}{\delta}\right)\right|\right)\phi_i(\xi)\;d\xi\right|\lesssim
\begin{cases}
\delta^{-3/2}\e^{-\frac{R}{2\delta}}\left\|\phi_i\right\|_{C^{1,\alpha}(B_{i,2})},&x\in B_{i,\frac{1}{2}},\\
\left\|\phi_i\right\|_{C^{1,\alpha}(B_{i,2})},&x\in \slab_\delta\backslash B_{i,\frac{1}{2}}.
\end{cases}
\end{align*}

For \eqref{5.13}, an estimate similar to \eqref{338} holds
\begin{align*}
&\left|\displaystyle\int_{\slab_\delta}\widetilde G(x,\xi)\eta\left(\frac{\delta}{R}\left|\xi-\left(0,\frac{\tau_i}{\delta}\right)\right|\right)\left(\gamma^\prime\left(\sum_{j=1}^n\sigma_jU_j\right)-\gamma^\prime\left(U_i\right)\right)\phi_i(\xi)\;d\xi\right|\\
&\lesssim\displaystyle\int_{B_{i,2}}|\widetilde G(x,\xi)|\left|\sum_{j\neq i}U_j(\xi)\right||\phi_i(\xi)|\;d\xi\lesssim\e^{-\frac{d-2R}{\delta}}\left\|\phi_i\right\|_{C^{1,\alpha}(B_{i,2})}.
\end{align*}

For \eqref{5.14}, since $1-\sum_{i=1}^n\eta\left(\frac{\delta}{R}\left|\xi-\left(0,\frac{\tau_i}{\delta}\right)\right|\right)=0$ in all $B_{i,1},\; i=1,\dots,n,$
by \eqref{1551}, \eqref{141}, \eqref{142}, Proposition \ref{prop5.1} and \eqref{1536},
\begin{align*}
&\left|\displaystyle\int_{\slab_\delta}\widetilde G(x,\xi)\left(1-\sum_{i=1}^n\eta\left(\frac{\delta}{R}\left|\xi-\left(0,\frac{\tau_i}{\delta}\right)\right|\right)\right)F(\boldsymbol{\tau}, u, \Gamma_{\mathrm{s}})(\xi)\;d\xi\right|\\
&\lesssim\displaystyle\int_{\slab_\delta\backslash\cup_{i=1}^nB_{i,1}}|\widetilde G(x,\xi)||F(\boldsymbol{\tau}, u, \Gamma_{\mathrm{s}})(\xi)|\;d\xi\\
&\lesssim\begin{cases}
\delta^{-3/2}\e^{-\frac{R}{2\delta}}\left(e^{-\frac {d_m}\delta}+\delta^{-k}e^{-\frac {2(1-|\boldsymbol{\tau}|)}\delta}+\delta^{-\frac{1}{2}-k}\|u\|_{H^k (\slab_\delta)}^2\right),&x\in B_{i,\frac{1}{2}},\\
e^{-\frac {d_m}\delta}+\delta^{-k}e^{-\frac {2(1-|\boldsymbol{\tau}|)}\delta}+\delta^{-\frac{1}{2}-k}\|u\|_{H^k (\slab_\delta)}^2,&x\in \slab_\delta\backslash \cup_{i=1}^nB_{i,\frac{1}{2}}.
\end{cases}
\end{align*}

Combining the above estimate, we have
{\small\begin{align*}
&|\widetilde{\psi}[\phi_1,\dots,\phi_n](x)|\\
&\leq \widetilde C \begin{cases}
\delta^{-3/2}\e^{-\frac{R}{2\delta}}\left(e^{-\frac {d_m}\delta}+\delta^{-k}e^{-\frac {2(1-|\boldsymbol{\tau}|)}\delta}+\delta^{-\frac{1}{2}-k}\sum_{i=1}^n\|\phi_i\|_{H^k (B_{i,2})}^2\right)\\
+\delta^{-3/2}\e^{-\frac{R}{2\delta}}\delta\sum_{i=1}^n\left\|\phi_i\right\|_{C^{1,\alpha}(B_{i,2})}+\e^{-\frac{d-2R}{\delta}}\sum_{i=1}^n\left\|\phi_i\right\|_{C^{1,\alpha}(B_{i,2})},&x\in\cup_{i=1}^nB_{i,\frac{1}{2}},\\[4mm]
e^{-\frac {d_m}\delta}+\delta^{-k}e^{-\frac {2(1-|\boldsymbol{\tau}|)}\delta}+\delta^{-\frac{1}{2}-k}\sum_{i=1}^n\|\phi_i\|_{H^k (B_{i,2})}^2\\
+\delta\sum_{i=1}^n\left\|\phi_i\right\|_{C^{1,\alpha}(B_{i,2})}+\e^{-\frac{d-2R}{\delta}}\sum_{i=1}^n\left\|\phi_i\right\|_{C^{1,\alpha}(B_{i,2})},&x\in \slab_\delta\backslash \cup_{i=1}^nB_{i,\frac{1}{2}}.
\end{cases}
\end{align*}}
\end{proof}

\begin{remark}\label{rmk2}
For $v\in H^k(B_{i,2})$, define the operator $D_{\phi_i}\widetilde\psi[\phi_1,\dots,\phi_n]:H^k (B_{i,2})\to H^k(\slab_\delta)$ as
$$D_{\phi_i}\widetilde\psi[\phi_1,\dots,\phi_n]v=\frac{\mathrm{d}}{\mathrm{d}t}\widetilde\psi[\phi_1,\dots,\phi_i+tv,\dots,\phi_n]\bigg|_{t=0},$$
$D_{\phi_i}f_{\mathrm{out}}[\phi_1,\dots,\phi_n,\widetilde\psi,\Gamma_{\mathrm{s}}]$ is defined similarly. Then differentiating both sides of equation \eqref{121} with respect to $\phi_i$ gives
\begin{equation*}
\begin{aligned}
\begin{cases}
L_{\mathrm{out}}\left[D_{\phi_i}\widetilde\psi[\phi_1,\dots,\phi_n]v\right]=D_{\phi_i}f_{\mathrm{out}}[\phi_1,\dots,\phi_n,\widetilde\psi,\Gamma_{\mathrm{s}}]v,\;&\text{in }\slab_\delta,\\
D_{\phi_i}\widetilde\psi[\phi_1,\dots,\phi_n]v=0,\;&\text{on }\partial\slab_\delta,
\end{cases}
\end{aligned}
\end{equation*}
Using the same argument as in Proposition \ref{prop5.2} and \eqref{141}, for $v\in H^k (B_{i,2})$, we obtain
\begin{equation}\label{128}
\left\|D_{\phi_i}\widetilde\psi[\phi_1,\dots,\phi_n]v\right\|_{H^{k}(\slab_\delta)}\lesssim\delta^{\frac{1}{2}}\|v\|_{H^{k}(B_{i,2})}\quad\text{for all }i=1,\dots,n.
\end{equation}
\end{remark}

\begin{remark}\label{1631}
By the exponential decay of $\p_{x_2}U_i$, and a same argument as the pointwise estimate for $\widetilde{\psi}$, we have
\begin{align*}
&\left|\left(f_{\mathrm{out}}[\phi_1,\dots,\phi_n,\widetilde\psi,\widetilde\Gamma_{\mathrm{s}}],\p_{x_2}U_i\right)_{L^2(\slab_\delta)}\right|\\
&\lesssim\delta\sum_{i=1}^n\left|\displaystyle\int_{\slab_\delta}\p_{x_2}U\left(\xi-\frac{\tau_i}{\delta}\vec e_2\right)\nabla\eta\left(\frac{\delta}{R}\left|\xi-\frac{\tau_i}{\delta}\vec e_2\right|\right)\cdot\nabla\phi_i(\xi)\;d\xi\right|\\
&\quad\;+\delta^2\sum_{i=1}^n\left|\displaystyle\int_{\slab_\delta}\p_{x_2}U\left(\xi-\frac{\tau_i}{\delta}\vec e_2\right)\Delta\eta\left(\frac{\delta}{R}\left|\xi-\frac{\tau_i}{\delta}\vec e_2\right|\right)\phi_i(\xi)\;d\xi\right|\\
&\quad\;+\sum_{i=1}^n\left|\displaystyle\int_{\slab_\delta}\p_{x_2}U\left(\xi-\frac{\tau_i}{\delta}\vec e_2\right)\eta\left(\frac{\delta}{R}\left|\xi-\frac{\tau_i}{\delta}\vec e_2\right|\right)\left(\gamma^\prime\left(\sum_{i=1}^n\sigma_iU_i\right)-\gamma^\prime\left(U_i\right)\right)\phi_i(\xi)\;d\xi\right|\\
&\quad\;+\left|\displaystyle\int_{\slab_\delta}\p_{x_2}U\left(\xi-\frac{\tau_i}{\delta}\vec e_2\right)\left(1-\sum_{i=1}^n\eta\left(\frac{\delta}{R}\left|\xi-\frac{\tau_i}{\delta}\vec e_2\right|\right)\right)F(\boldsymbol{\tau}, u, \Gamma_{\mathrm{s}})(\xi)\;d\xi\right|\\
&\quad\,+\left|\displaystyle\int_{\slab_\delta}\p_{x_2}U\left(\xi-\frac{\tau_i}{\delta}\vec e_2\right)\left(1-\sum_{i=1}^n\eta\left(\frac{\delta}{R}\left|\xi-\frac{\tau_i}{\delta}\vec e_2\right|\right)\right)\left(\gamma^\prime\left(\sum_{i=1}^n\sigma_iU_i\right)-1\right)\widetilde\psi(\xi)\;d\xi\right|\\
&\lesssim\delta^{-3/2}\e^{-\frac{R}{2\delta}}\left(e^{-\frac {d_m}\delta}+\delta^{-k}e^{-\frac {2(1-|\boldsymbol{\tau}|)}\delta}+\delta^{-\frac{1}{2}-k}\sum_{i=1}^n\|\phi_i\|_{H^k (B_{i,2})}^2+\delta\sum_{i=1}^n\left\|\phi_i\right\|_{C^{1,\alpha}(B_{i,2})}\right)\\
&\quad\,+\e^{-\frac{d-2R}{\delta}}\sum_{i=1}^n\left\|\phi_i\right\|_{C^{1,\alpha}(B_{i,2})}
\end{align*}
\end{remark}

\subsection{Solving the inner problem}
Given $\phi_i$, $i=1,\dots,n$ with $\|\phi_i\|_{H^{k}(B_{i,2})}\leq\delta^k$, the outer problem admits a solution $\widetilde{\psi}[\phi_1,\dots,\phi_n]$. After substituting it into the inner problem, we apply the contraction mapping principle to show that, for sufficiently small $\delta$ and $\boldsymbol{\tau}\in P_{\varepsilon}$, the inner problem admits a solution.

The inner problem \eqref{in} is
\begin{equation}\label{417}
\begin{aligned}
L_{\mathrm{in}}^i[\phi_i]=f_{\mathrm{in}}[\phi_1,\dots,\phi_n,\widetilde{\psi},\widetilde\Gamma_{\mathrm{s}}],\quad i=1,\dots,n,
\end{aligned}
\end{equation}
where $$L_{\mathrm{in}}^i:=\Delta-\gamma^\prime\left(U_i\right),$$
$$f_{\mathrm{in}}[\phi_1,\dots,\phi_n,\widetilde{\psi},\widetilde\Gamma_{\mathrm{s}}]:=F(\boldsymbol{\tau},\phi_1,\dots,\phi_n,\widetilde{\psi},\widetilde\Gamma_{\mathrm{s}})+\left(\gamma^\prime\left(\sum_{i=1}^n\sigma_iU_i\right)-1\right)\widetilde\psi.$$
\begin{prop}\label{prop5.3}
For each \(i=1,\dots,n\), assume that $\|\phi_j\|_{H^{k_0}(B_{j,2})}\leq\delta^{k_0+1}$ for all $j\neq i$. Then there exist constants $\widetilde C >0$ and $\delta_0 > 0$  such that for $\boldsymbol{\tau} \in P_{\varepsilon}$ and $\delta\in(0,\delta_0)$, the operator
$$\mathrm{T}^i_{\mathrm{in}}[\phi_i]:=(L_{\mathrm{in}}^i)^{-1}\left[f_{\mathrm{in}}[\phi_1,\dots,\phi_n,\widetilde{\psi},\widetilde\Gamma_{\mathrm{s}}]+c_i[\phi_1,\dots,\phi_n,\widetilde{\psi},\widetilde\Gamma_{\mathrm{s}}]\p_{x_2} U_i\right]$$
has a fixed point $\widetilde\phi_i$ in the set
$$B_{\mathrm{in}}^i:=\left\{\phi_i\in H_{\mathrm{e}}^{2}(\slab_\delta)\;\left|\; \|\phi_i\|_{H^{2}(\slab_\delta)}\leq\delta^{k_0+1}\right.\right\}.$$
Here $c_i[\phi_1,\dots,\phi_n,\widetilde{\psi},\widetilde\Gamma_{\mathrm{s}}]$ is the orthogonal projection coefficient such that the following orthogonality condition holds:
$$\left(f_{\mathrm{in}}[\phi_1,\dots,\phi_n,\widetilde{\psi},\widetilde\Gamma_{\mathrm{s}}]+c_i[\phi_1,\dots,\phi_n,\widetilde{\psi},\widetilde\Gamma_{\mathrm{s}}]\p_{x_2} U_i,\p_{x_2} U_i\right)_{L^2(\slab_\delta)}=0.$$
And $\widetilde\phi_i$ satisfies
$$\left\|\widetilde{\phi_i}\right\|_{H^{k_0}(B_{i,2})}\leq\widetilde C \left(\delta^{\frac{1}{2}}\sum_{j\neq i}\left\|\phi_j\right\|_{H^{k_0}(B_{j,2})}+\delta^{\frac 14} | \log \delta|^{\frac{1}{2}}e^{-\frac {d_m}\delta}+\delta^{-k_0}e^{-\frac {2(1-|\boldsymbol{\tau}|)}\delta}\right),$$
$$\left\|\widetilde\phi_i\right\|_{C^{1,\alpha}(B_{i,2})}\leq\widetilde C \left(\delta^{\frac{1}{2}}\sum_{j\neq i}\left\|\phi_j\right\|_{H^{k_0}(B_{j,2})}+\delta^{\frac 14} | \log \delta|^{\frac{1}{2}}e^{-\frac {d_m}\delta}+\delta^{-k_0}e^{-\frac {2(1-|\boldsymbol{\tau}|)}\delta}\right).$$
\end{prop}

\begin{proof}
Throughout the proof below, we assume $2\leq k\leq k_0$. Firstly we prove that $\mathrm{T}^i_{\mathrm{in}}[\phi_i]\in B^i_{\mathrm{in}}$ for any $\phi_i\in B^i_{\mathrm{in}}$. By the linear theory Lemma \ref{non-degeneracy proposition} for the inner problem and Lemma \ref{prop:sign of Uy}, 
\begin{align*}
&\left\|\mathrm{T}^i_{\mathrm{in}}[\phi_i]\right\|_{H^{2}(\slab_\delta)}\lesssim\left\|f_{\mathrm{in}}[\phi_1,\dots,\phi_n,\widetilde{\psi},\widetilde\Gamma_{\mathrm{s}}]+c_i[\phi_1,\dots,\phi_n,\widetilde{\psi},\widetilde\Gamma_{\mathrm{s}}]\p_{x_2} U_i\right\|_{H^{k-2}(\slab_\delta)}\\
&\lesssim\left\|F(\boldsymbol{\tau},\phi_1,\dots,\phi_n,\widetilde{\psi},\widetilde\Gamma_{\mathrm{s}})\right\|_{H^{k-2}(\slab_\delta)}+\left\|\widetilde\psi\right\|_{H^{k-2}(\slab_\delta)}+\left|c_i[\phi_1,\dots,\phi_n,\widetilde{\psi},\widetilde\Gamma_{\mathrm{s}}]\right|.
\end{align*}
Since $c_i[\phi_1,\dots,\phi_n,\widetilde{\psi},\widetilde\Gamma_{\mathrm{s}}]$ is the orthogonal projection coefficient, we have
\begin{align*}
&|c_i[\phi_1,\dots,\phi_n,\widetilde{\psi},\widetilde\Gamma_{\mathrm{s}}]|=\left|\frac{\left(f_{\mathrm{in}}[\phi_1,\dots,\phi_n,\widetilde\psi,\widetilde\Gamma_{\mathrm{s}}],\partial_{x_2}U_i\right)_{L^2(\slab_\delta)}}{\|\partial_{x_2}U_i\|_{L^2(\slab_\delta)}^2}\right|\\
&\lesssim\left\|f_{\mathrm{in}}[\phi_1,\dots,\phi_n,\widetilde{\psi},\widetilde\Gamma_{\mathrm{s}}]\right\|_{H^{k-2}(\slab_\delta)}\left\|\p_{x_2} U_i\right\|_{H^{k-2}(\slab_\delta)}\lesssim\left\|f_{\mathrm{in}}[\phi_1,\dots,\phi_n,\widetilde{\psi},\widetilde\Gamma_{\mathrm{s}}]\right\|_{H^{k-2}(\slab_\delta)}.
\end{align*}
Thus
$$\left\|\mathrm{T}^i_{\mathrm{in}}[\phi_i]\right\|_{H^{2}(\slab_\delta)}\lesssim\left\|F(\boldsymbol{\tau},\phi_1,\dots,\phi_n,\widetilde{\psi},\widetilde\Gamma_{\mathrm{s}})\right\|_{H^{k-2}(\slab_\delta)}+\left\|\widetilde\psi\right\|_{H^{k}(\slab_\delta)}.$$
Applying \eqref{1551}, \eqref{141}, \eqref{142}, Proposition \ref{prop5.1} and Proposition \ref{prop5.2},
$$\left\|\mathrm{T}^i_{\mathrm{in}}[\phi_i]\right\|_{H^{2}(\slab_\delta)}\lesssim\delta^{\frac{1}{2}}\sum_{j=1}^n\left\|\phi_j\right\|_{H^{k}(B_{j,2})}+\delta^{\frac 14} | \log \delta|^{\frac{1}{2}}e^{-\frac {d_m}\delta}+\delta^{-k}e^{-\frac {2(1-|\boldsymbol{\tau}|)}\delta}.$$
Combining the above estimate, for sufficiently small $\delta$, $\mathrm{T}^i_{\mathrm{in}}[\phi_i]\in B^i_{\mathrm{in}}$.

Next we prove $\mathrm{T}^i_{\mathrm{in}}$ is a contraction on $B^i_{\mathrm{in}}$. By \eqref{141}, \eqref{142}, \eqref{1908} and \eqref{128}, 
{\small\begin{align*}
&\left\|\mathrm{T}^i_{\mathrm{in}}[\phi_{i,1}]-\mathrm{T}^i_{\mathrm{in}}[\phi_{i,2}]\right\|_{H^{2}(\slab_\delta)}\\
&\lesssim\left\|f_{\mathrm{in}}[\phi_1,\dots,\phi_{i,1},\dots,\phi_n,\widetilde{\psi},\widetilde\Gamma_{\mathrm{s}}]-f_{\mathrm{in}}[\phi_1,\dots,\phi_{i,2},\dots,\phi_n,\widetilde{\psi},\widetilde\Gamma_{\mathrm{s}}]\right\|_{H^{k-2}(\slab_\delta)}\\
&\quad\;+\left|c_i[\phi_1,\dots,\phi_{i,1},\dots,\phi_n,\widetilde{\psi},\widetilde\Gamma_{\mathrm{s}}]-c_i[\phi_1,\dots,\phi_{i,2},\dots,\phi_n,\widetilde{\psi},\widetilde\Gamma_{\mathrm{s}}]\right|\left\|\p_{x_2} U_i\right\|_{H^{k-2}(\slab_\delta)}\\
&\lesssim\left\|D_u F(\boldsymbol{\tau}, \zeta, \widetilde{\Gamma}_{\mathrm{s}})(\phi_{i,1}-\phi_{i,2})\right\|_{H^{k-2}(\slab_\delta)}+\left\|D_{\Gamma_{\mathrm{s}}} F (\boldsymbol{\tau}, u, \zeta)D_u \widetilde{\Gamma}_{\mathrm{s}}[\zeta](\phi_{i,1}-\phi_{i,2})\right\|_{H^{k-2}(\slab_\delta)}\\
&\quad\;+\left\|D_{\phi_i}\widetilde\psi(\phi_{i,1}-\phi_{i,2})\right\|_{H^{k-2}(\slab_\delta)}+\left|c_i[\phi_1,\dots,\phi_{i,1},\dots,\phi_n,\widetilde{\psi},\widetilde\Gamma_{\mathrm{s}}]-c_i[\phi_1,\dots,\phi_{i,2},\dots,\phi_n,\widetilde{\psi},\widetilde\Gamma_{\mathrm{s}}]\right|\\
&\lesssim\left(\delta^{-1} \left\|\widetilde{\Gamma}_{\mathrm{s}}\right\|_{H^{k}(\R)} + \|u\|_{H^{k}  (\slab_\delta)} + \sum_{j=1}^n\|(U_j)_\mathrm{bc}\|_{H^{k}  (\slab_\delta)}+\delta\right)\|\phi_{i,1}-\phi_{i,2}\|_{H^{2}(\slab_\delta)}\\
&\quad\;+\delta^{-1} \left\|\widetilde{\Gamma}_{\mathrm{s}}\right\|_{H^{k}(\R)}\left(\delta^{\frac 12-k}\|u\|_{H^k (\slab_\delta)} + \delta^{\frac 34-k} e^{-\frac {1-|\boldsymbol{\tau}|}\delta} \right)\|\phi_{i,1}-\phi_{i,2}\|_{H^{2}(\slab_\delta)}\\
&\quad\;+\left|c_i[\phi_1,\dots,\phi_{i,1},\dots,\phi_n,\widetilde{\psi},\widetilde\Gamma_{\mathrm{s}}]-c_i[\phi_1,\dots,\phi_{i,2},\dots,\phi_n,\widetilde{\psi},\widetilde\Gamma_{\mathrm{s}}]\right|.
\end{align*}}

\noindent
For $\left|c_i[\phi_1,\dots,\phi_{i,1},\dots,\phi_n,\widetilde{\psi},\widetilde\Gamma_{\mathrm{s}}]-c_i[\phi_1,\dots,\phi_{i,2},\dots,\phi_n,\widetilde{\psi},\widetilde\Gamma_{\mathrm{s}}]\right|$, we have
\begin{align*}
&\left|c_i[\phi_1,\dots,\phi_{i,1},\dots,\phi_n,\widetilde{\psi},\widetilde\Gamma_{\mathrm{s}}]-c_i[\phi_1,\dots,\phi_{i,2},\dots,\phi_n,\widetilde{\psi},\widetilde\Gamma_{\mathrm{s}}]\right|\\
&\lesssim\left\|f_{\mathrm{in}}[\phi_1,\dots,\phi_{i,1},\dots,\phi_n,\widetilde{\psi},\widetilde\Gamma_{\mathrm{s}}]-f_{\mathrm{in}}[\phi_1,\dots,\phi_{i,2},\dots,\phi_n,\widetilde{\psi},\widetilde\Gamma_{\mathrm{s}}]\right\|_{H^{k-2}(\slab_\delta)}.
\end{align*}
Thus
\begin{align*}
&\left\|\mathrm{T}^i_{\mathrm{in}}[\phi_{i,1}]-\mathrm{T}^i_{\mathrm{in}}[\phi_{i,2}]\right\|_{H^{2}(\slab_\delta)}\\
&\lesssim\left(\delta^{-1} \left\|\widetilde{\Gamma}_{\mathrm{s}}\right\|_{H^{k}(\R)} + \|u\|_{H^{k}  (\slab_\delta)} + \sum_{j=1}^n\|(U_j)_\mathrm{bc}\|_{H^{k}  (\slab_\delta)}+\delta\right)\|\phi_{i,1}-\phi_{i,2}\|_{H^{2}(\slab_\delta)}\\
&\quad\;+\delta^{-1} \left\|\widetilde{\Gamma}_{\mathrm{s}}\right\|_{H^{k}(\R)}\left(\delta^{\frac 12-k}\|u\|_{H^k (\slab_\delta)} + \delta^{\frac 34-k} e^{-\frac {1-|\boldsymbol{\tau}|}\delta} \right)\|\phi_{i,1}-\phi_{i,2}\|_{H^{2}(\slab_\delta)}.
\end{align*}

Finally, from Corollary \ref{C:Ubc} and Proposition \ref{prop5.1}, for sufficiently small $\delta$, $\mathrm{T}^i_{\mathrm{in}}$ is a contraction mapping on $B_{\mathrm{in}}^i$. Then by the contraction mapping principle, $\mathrm{T}^i_{\mathrm{in}}$ possesses a fixed point $\widetilde{\phi_i}$ in $B^i_{\mathrm{in}}$. 

Moreover, by \eqref{1536}, \eqref{322}, \eqref{133} and the above argument, since $k_0\geq 4$, we have
\begin{align*}
\left\|\widetilde{\phi_i}\right\|_{H^{k_0}(B_{i,2})}&\lesssim \left\|f_{\mathrm{in}}[\phi_1,\dots,\phi_{i-1},0,\phi_{i+1},\dots,\phi_n,\widetilde{\psi},\widetilde\Gamma_{\mathrm{s}}]\right\|_{H^{k_0-2}(\slab_\delta)}\\
&\lesssim\delta^{\frac{1}{2}}\sum_{j\neq i}\left\|\phi_j\right\|_{H^{k_0}(B_{j,2})}+\delta^{\frac 14} | \log \delta|^{\frac{1}{2}}e^{-\frac {d_m}\delta}+\delta^{-k_0}e^{-\frac {2(1-|\boldsymbol{\tau}|)}\delta},
\end{align*}
and
\begin{align*}
\left\|\widetilde\phi_i\right\|_{C^{1,\alpha}(B_{i,2})}&\lesssim\left\|f_{\mathrm{in}}[\phi_1,\dots,\widetilde\phi_i,\dots,\phi_n,\widetilde{\psi},\widetilde\Gamma_{\mathrm{s}}]\right\|_{H^{k_0-2}(\slab_\delta)}\\
&\lesssim\delta^{\frac{1}{2}}\sum_{j\neq i}\left\|\phi_j\right\|_{H^{k_0}(B_{j,2})}+\delta^{\frac 14} | \log \delta|^{\frac{1}{2}}e^{-\frac {d_m}\delta}+\delta^{-k_0}e^{-\frac {2(1-|\boldsymbol{\tau}|)}\delta}.
\end{align*}
\end{proof}
By a method similar to that in Remark \ref{rmk2}, for each $i$, we also have
\begin{equation}\label{440}
\left\|D_{\phi_j}\widetilde{\phi_i}[\phi_1,\dots,\phi_{i-1},\phi_{i+1},\dots,\phi_n]v\right\|_{H^{2}(\slab_\delta)}\lesssim\delta^{\frac{1}{2}}\|v\|_{H^{2}(\slab_\delta)}\quad\text{for all }j\neq i, v\in H^{2}(\slab_\delta),
\end{equation}
where 
$$D_{\phi_j}\widetilde{\phi_i}[\phi_1,\dots,\phi_{i-1},\phi_{i+1},\dots,\phi_n]v:=\frac{\mathrm{d}}{\mathrm{d}t}\widetilde{\phi_i}[\phi_1,\dots,\phi_j+tv,\dots,\phi_{i-1},\phi_{i+1},\dots,\phi_n]\bigg|_{t=0}$$
We employ the idea in \cite[section 9.2]{MR5018790}, repeatedly use Proposition \ref{prop5.3} and \eqref{440} to construct the functions $\phi_1,\dots,\phi_n$. The main idea is to solve the inner equations one by one, regarding the functions already constructed as fixed functions in the subsequent equations.

More precisely, we first fix $(\phi_2,\dots,\phi_n)$ and apply Proposition \ref{prop5.3} with $i=1$. This gives a unique fixed point
$$\widetilde\phi_1[\phi_2,\dots,\phi_n]\in B_{\mathrm{in}}^1.$$
We then substitute this expression for $\phi_1$ into the equation
for $\phi_2$ and apply Proposition \ref{prop5.3} again. In this way,
$\phi_2$ is determined as a function of
$(\phi_3,\dots,\phi_n)$. Continuing this procedure, we successively
determine $\phi_i$ as a function of the remaining unknowns $(\phi_{i+1},\dots,\phi_n)$.

We next proceed to construct $\phi_2$. For fixed
$(\phi_2,\dots,\phi_n)$, the first equation has already been solved
and gives $\widetilde\phi_1[\phi_2,\dots,\phi_n]\in B_{\mathrm{in}}^1$. We substitute this expression into the second equation and regard
$\phi_3,\dots,\phi_n$ as fixed. Thus, the fixed-point problem for $\phi_2$
takes the form
{\small
\begin{align*}
&\phi_2=\mathrm{T}_{\mathrm{in}}^2\left[\phi_2\right]\\
&:=(L_{\mathrm{in}}^2)^{-1}\left[f_{\mathrm{in}}[\widetilde\phi_1[\phi_2,\dots,\phi_n],\phi_2,\dots,\phi_n,\widetilde{\psi},\widetilde\Gamma_{\mathrm{s}}]+c_i[\widetilde\phi_1[\phi_2,\dots,\phi_n],\phi_2,\dots,\phi_n,\widetilde{\psi},\widetilde\Gamma_{\mathrm{s}}]\p_{x_2} U_2\right].
\end{align*}}

\noindent
By Proposition \ref{prop5.3} and \eqref{440}, for $\delta>0$ sufficiently small, $\mathrm{T}_{\mathrm{in}}^2$ is a contraction on $B_{\mathrm{in}}^2$. Therefore, there exists a unique
$$\widetilde\phi_2[\phi_3,\dots,\phi_n]\in B_{\mathrm{in}}^2,$$
and $\widetilde\phi_2$ satisfies the estimate
$$\left\|\widetilde\phi_2[\phi_3,\dots,\phi_n]\right\|_{H^{k_0}(B_{2,2})}\lesssim
\delta^{\frac12}
\sum_{j=3}^n\|\phi_j\|_{H^{k_0}(B_{j,2})}+\delta^{\frac14}|\log\delta|^{\frac12}e^{-\frac{d_m}{\delta}}+\delta^{-k_0}e^{-\frac{2(1-|\boldsymbol{\tau}|)}{\delta}},$$
$$\left\|\widetilde\phi_2[\phi_3,\dots,\phi_n]\right\|_{C^{1,\alpha}(B_{2,2})}\lesssim\delta^{\frac{1}{2}}\sum_{j=3}^n\left\|\phi_j\right\|_{H^{k}(B_{j,2})}+\delta^{\frac 14} | \log \delta|^{\frac{1}{2}}e^{-\frac {d_m}\delta}+\delta^{-k_0}e^{-\frac {2(1-|\boldsymbol{\tau}|)}\delta}.$$
The construction for $\widetilde\phi_3,\dots,\widetilde\phi_n$ is analogous. 

Finally, solving the equation for $\phi_n$ and successively substituting
the resulting estimate back into those for $\phi_{n-1},\dots,\phi_1$,
we obtain
\begin{equation}\label{1457}
\left\|\widetilde\phi_i\right\|_{H^{k_0}(B_{i,2})}\lesssim\delta^{\frac14}|\log\delta|^{\frac12}e^{-\frac{d_m}{\delta}}+\delta^{-k_0}e^{-\frac{2(1-|\boldsymbol{\tau}|)}{\delta}},\qquad i=1,\dots,n,
\end{equation}
and by Proposition \ref{prop5.1}, Proposition \ref{prop5.2}, we obtain
\begin{equation}\label{1458}
\left\|\widetilde{\psi}\right\|_{H^{k_0}(\slab_\delta)}\lesssim\delta^{\frac 14} |\log \delta|^{\frac{1}{2}}e^{-\frac {d_m}\delta}+\delta^{-k_0}e^{-\frac {2(1-|\boldsymbol{\tau}|)}\delta},
\end{equation}
\begin{equation}\label{1434}
\left\|\widetilde{\Gamma}_{\mathrm{s}}\right\|_{H^{k_0}(\R)}\lesssim\delta^{1-k_0} e^{-\frac {2(1-|\boldsymbol{\tau}|)}\delta}+\delta^{1-k_0}|\log \delta|e^{-\frac {2d_m}\delta}.
\end{equation}

\section{Proof of the main result} \label{proof main theorem sec}

This section is devoted to the proof of the main result of this paper. In Section \ref{Solve the inner--outer gluing system}, we construct $(\widetilde\phi_1,\dots,\widetilde\phi_n,\widetilde\psi,\widetilde{\Gamma}_{\mathrm{s}})$ to the following system
\begin{equation}\label{2209}
\begin{cases}
-\Delta\widetilde\phi_i+\gamma^\prime\left(U_i\right)\widetilde\phi_i+f_{\mathrm{in}}[\widetilde\phi_1,\dots,\widetilde\phi_n,\widetilde{\psi},\widetilde\Gamma_{\mathrm{s}}]+c_i\p_{x_2} U_i=0,\;&\text{in }\slab_\delta, i=1,\dots,n,\\
-\Delta\widetilde\psi+\widetilde\psi=f_{\mathrm{out}}[\widetilde\phi_1,\dots,\widetilde\phi_n,\widetilde\psi,\widetilde\Gamma_{\mathrm{s}}],\;&\text{in }\slab_\delta,\\
\widetilde\psi=0,\;&\text{on }\partial\slab_\delta,\\
\left(g-\alpha^2 \p_{x_1}^2\right)\widetilde\Gamma_{\mathrm{s}}+G(\boldsymbol{\tau},\widetilde\phi_1,\dots,\widetilde\phi_n,\widetilde\psi,\widetilde\Gamma_{\mathrm{s}})=0,\;&\textrm{on }\mathbb{R}.
\end{cases}
\end{equation}
Therefore, it remains to find a suitable $\boldsymbol{\tau}\in P_\varepsilon$ such that 
$$c_i=c_i[\widetilde\phi_1,\dots,\widetilde\phi_n,\widetilde{\psi},\widetilde\Gamma_{\mathrm{s}}]=0,\qquad i=1,\dots,n.$$ 
Once such a $\boldsymbol{\tau}$ is found, let
$$\tilde u(x):=\sum_{i=1}^n\eta_i\widetilde\phi_i(x)+\widetilde\psi(x),$$
where $\eta_i(x):=\eta\left(\frac{\delta}{R}
\left|x-\left(0,\frac{\tau_i}{\delta}\right)\right|
\right)$, we obtain that $(\tilde u,\widetilde\Gamma_{\mathrm{s}})$ is a solution of the original equation \eqref{perturbed problem}. For this purpose, we need the following lemma.
\begin{lemma}\label{lem6.1}
Let
$$\varphi_i:=\sigma_iU_i-\sigma_i(U_i)_\mathrm{bc}+\eta_i\widetilde\phi_i,$$
for all $i=1,\dots,n$, we have
$$\left|\left(\frac{\I}
{1+\widetilde{\Gamma}'(\delta\,\cdot)} \cdot \nabla\varphi_i,\eta_i\p_{x_2} U_i\right)_{L^2(\slab_\delta)}\right|\eqsim1,$$
$$\left||\left(\frac{\I}
{1+\widetilde{\Gamma}'(\delta\,\cdot)} \cdot \nabla\varphi_i , \eta_j\p_{x_2} U_j\right)_{L^2(\slab_\delta)}\right|\to 0,\quad\text{as }\delta\to0\text{ and }j\neq i,$$
where $\widetilde{\Gamma}=(\widetilde{\Gamma}_1,\widetilde{\Gamma}_2)$ (see Section \ref{1448})
\end{lemma}

\begin{proof}
Recall that
$$\varphi_i=\sigma_iU_i-\sigma_i(U_i)_{\mathrm{bc}}+\eta_i\widetilde\phi_i,\qquad i=1,\dots,n.$$
By the Sobolev trace theorem, Sobolev embedding and \eqref{1434}, we have
$$\left\|\widetilde{\Gamma}'(\delta\,\cdot)\right\|_{L^\infty(\mathbb R)}\lesssim\left\|\widetilde{\Gamma}_{\mathrm{s}}\right\|_{H^k(\R)}\lesssim\delta^{1-k} e^{-\frac {2(1-|\boldsymbol{\tau}|)}\delta}+\delta^{1-k}|\log \delta|e^{-\frac {2d_m}\delta}=o(1).$$
Moreover, by Corollary \ref{C:Ubc}, \eqref{1457} and \eqref{1458},
$$\left\|\partial_{x_2}(U_i)_{\mathrm{bc}}\right\|_{L^2(\slab_\delta)}\lesssim \delta^{\frac 14} e^{-\frac {1 - |\tau|}\delta}=o(1),$$
$$\|\tilde u\|_{H^k(\slab_\delta)}\lesssim\delta^{\frac14}|\log\delta|^{\frac12}e^{-\frac{d_m}{\delta}}+\delta^{-k}e^{-\frac{2(1-|\boldsymbol{\tau}|)}{\delta}}=o(1),$$
as $\delta\to0$, uniformly for $\boldsymbol{\tau}\in P_\varepsilon$. Consequently, 
$$\frac{\I}
{1+\widetilde{\Gamma}'(\delta\,\cdot)} \cdot \nabla\varphi_i=\sigma_i\partial_{x_2}U_i+r_i,$$
where $\|r_i\|_{L^2(\slab_\delta)}=o(1)$.

Taking $j=i$, we obtain
\begin{align*}
\left(\frac{\I}
{1+\widetilde{\Gamma}'(\delta\,\cdot)} \cdot \nabla\varphi_i , \eta_i\partial_{x_2}U_i\right)_{L^2(\slab_\delta)}=\left(\sigma_i\partial_{x_2}U_i,\eta_i\partial_{x_2}U_i\right)_{L^2(\slab_\delta)}+\left(r_i,\eta_i\partial_{x_2}U_i\right)_{L^2(\slab_\delta)}.
\end{align*}
Since $\eta_i\equiv1$ in a neighborhood of the $i$-th concentration point and $\partial_{x_2}U_i$ is exponentially localized, we have
$$\left(\sigma_i\partial_{x_2}U_i,\eta_i\partial_{x_2}U_i\right)_{L^2(\slab_\delta)}=\sigma_i\|\partial_{x_2}U_i\|_{L^2(\slab_\delta)}^2+o(1)=O(1).$$
On the other hand, the Cauchy-Schwarz inequality yields
$$\left|\left(r_i,\eta_i\partial_{x_2}U_i\right)_{L^2(\slab_\delta)}\right|\leq\|r_i\|_{L^2(\slab_\delta)}\|\partial_{x_2}U_i\|_{L^2(\slab_\delta)}=o(1).$$
Therefore,
$$ \left| \left(\frac{\I}{1+\widetilde{\Gamma}'(\delta\,\cdot)} \cdot \nabla \varphi_i, \eta_i\partial_{x_2}U_i\right)_{L^2(\slab_\delta)}\right|\eqsim 1.$$

It remains to consider $j\neq i$. In this case,
$$\left(\frac{\I}
{1+\widetilde{\Gamma}'(\delta\,\cdot)} \cdot \nabla\varphi_i, \eta_j\partial_{x_2}U_j\right)_{L^2(\slab_\delta)}=\left(\sigma_i\partial_{x_2}U_i,\eta_j\partial_{x_2}U_j\right)_{L^2(\slab_\delta)}+\left(r_i,\eta_j\partial_{x_2}U_j\right)_{L^2(\slab_\delta)}.$$
The second term tends to zero by the estimate for $r_i$ and the 
Cauchy-Schwarz inequality. For the first term, Lemma \ref{prop:sign of Uy} and \eqref{E:F0} imply
\begin{equation}\label{1611}
\left|\left(\sigma_i\partial_{x_2}U_i , \eta_j\partial_{x_2}U_j\right)_{L^2(\slab_\delta)}\right|\lesssim\delta^{-\frac 12}e^{-\frac{d_m}\delta}\to0,\quad j\neq i.
\end{equation}
Therefore,
$$\left|\left(\frac{\I}
{1+\widetilde{\Gamma}'(\delta\,\cdot)} \cdot \nabla\varphi_i,\eta_j\partial_{x_2}U_j\right)_{L^2(\slab_\delta)}\right|\to0,\quad j\neq i.$$
\end{proof}

By \eqref{2209}, we have
\begin{align*}
-\Delta\tilde u+\gamma^\prime\left(\textstyle\sum_{i=1}^n\sigma_iU_i\right)\tilde u+F(\boldsymbol{\tau},\tilde u,\widetilde\Gamma_{\mathrm{s}})+\sum_{i=1}^nc_i\eta_i\p_{x_2} U_i=0.
\end{align*}
Define
\begin{equation}\label{1635}
\mathcal{N}_{\boldsymbol{\tau}}(\tilde u,\widetilde\Gamma_{\mathrm{s}})=-\Delta\tilde u+\gamma^\prime\left(\textstyle\sum_{i=1}^n\sigma_iU_i\right)\tilde u+F(\boldsymbol{\tau},\tilde u,\widetilde\Gamma_{\mathrm{s}}),
\end{equation}
$$a_{ij}=\left(\frac{\I}
{1+\widetilde{\Gamma}'(\delta\,\cdot)} \cdot \nabla\varphi_i,\eta_j\p_{x_2} U_j\right)_{L^2(\slab_\delta)},$$
then
\begin{equation}\label{1625}
\begin{pmatrix}
\left(\mathcal{N}_{\boldsymbol{\tau}}(\tilde u,\widetilde\Gamma_{\mathrm{s}}),\frac{\I}
{1+\widetilde{\Gamma}'(\delta\,\cdot)} \cdot \nabla\varphi_1\right)_{L^2(\slab_\delta)}\\
\left(\mathcal{N}_{\boldsymbol{\tau}}(\tilde u,\widetilde\Gamma_{\mathrm{s}})\frac{\I}
{1+\widetilde{\Gamma}'(\delta\,\cdot)} \cdot \nabla\varphi_2\right)_{L^2(\slab_\delta)}\\
\vdots\\
\left(\mathcal{N}_{\boldsymbol{\tau}}(\tilde u,\widetilde\Gamma_{\mathrm{s}}),\frac{\I}
{1+\widetilde{\Gamma}'(\delta\,\cdot)} \cdot \nabla\varphi_n\right)_{L^2(\slab_\delta)}
\end{pmatrix}+\begin{pmatrix}
a_{11} & a_{12} & \cdots & a_{1n}\\
a_{21} & a_{22} & \cdots & a_{2n}\\
\vdots & \vdots & \ddots & \vdots\\
a_{n1} & a_{n2} & \cdots & a_{nn}
\end{pmatrix}\begin{pmatrix}
c_1\\
c_2\\
\vdots\\
c_n
\end{pmatrix}=0.
\end{equation}
By Lemma~\ref{lem6.1}, for sufficiently small $\delta$, the matrix $(a_{ij})_{1\leq i,j\leq n}$ is strictly diagonally dominant, i.e.,
$$|a_{ii}|>\sum_{j\neq i}|a_{ij}|,
\quad i=1,\dots,n.$$
Hence $(a_{ij})_{1\leq i,j\leq n}$ is invertible. Then, from \eqref{1625} we obtain
$$\begin{pmatrix}
c_1\\
c_2\\
\vdots\\
c_n
\end{pmatrix}=0\quad\Longleftrightarrow\quad\begin{pmatrix}
\left(\mathcal{N}_{\boldsymbol{\tau}}(\tilde u,\widetilde\Gamma_{\mathrm{s}}),\frac{\I}
{1+\widetilde{\Gamma}'(\delta\,\cdot)} \cdot \nabla\varphi_1\right)_{L^2(\slab_\delta)}\\
\left(\mathcal{N}_{\boldsymbol{\tau}}(\tilde u,\widetilde\Gamma_{\mathrm{s}}),\frac{\I}
{1+\widetilde{\Gamma}'(\delta\,\cdot)} \cdot \nabla\varphi_2\right)_{L^2(\slab_\delta)}\\
\vdots\\
\left(\mathcal{N}_{\boldsymbol{\tau}}(\tilde u,\widetilde\Gamma_{\mathrm{s}}),\frac{\I}
{1+\widetilde{\Gamma}'(\delta\,\cdot)} \cdot \nabla\varphi_n\right)_{L^2(\slab_\delta)}
\end{pmatrix}=0.$$
Thus, the required conditions
$$c_i=0,\quad i=1,\dots,n,$$
are equivalent to the reduced system
$$b_i(\boldsymbol{\tau}):=\left(\mathcal{N}_{\boldsymbol{\tau}}(\tilde u,\widetilde\Gamma_{\mathrm{s}}),\frac{\I}
{1+\widetilde{\Gamma}'(\delta\,\cdot)} \cdot \nabla\varphi_i\right)_{L^2(\slab_\delta)}=0,
\quad i=1,\dots,n.$$

\begin{proof}[Proof of Theorem~\ref{main euler theorem}]
Since the case $n=1$ has already been proved in \cite{MR4577959}, we only consider the case $n\geq 2$ here. The preceding argument shows that it remains to find $\boldsymbol{\tau} \in P_\varepsilon$ such that $b_i(\boldsymbol{\tau}) = 0$ for all $i=1,\dots,n$. 

Substituting the expression \eqref{1635} of $\mathcal{N}_{\boldsymbol{\tau}}$ into $b_i(\boldsymbol{\tau})$, we have
{\footnotesize
\begin{align*}
&b_i(\boldsymbol{\tau})=\left(-\Delta(\omega_{\boldsymbol{\tau}}+\tilde u)+\left|1+\widetilde\Gamma^\prime(\delta \,\cdot\,)\right|^2\gamma(\omega_{\boldsymbol{\tau}}+\tilde u),\frac{\I}
{1+\widetilde{\Gamma}'(\delta\,\cdot)} \cdot \nabla\varphi_i\right)_{L^2(\slab_\delta)}\\
&=\left(-\Delta(\sigma_iU_i-\sigma_i(U_i)_\mathrm{bc}+\eta_i\widetilde\phi_i),\frac{\I}
{1+\widetilde{\Gamma}'(\delta\,\cdot)} \cdot \nabla\varphi_i\right)_{L^2(\slab_\delta)}\\
&\quad\,+\left(\left|1+\widetilde{\Gamma}^\prime(\delta \,\cdot\,)\right|^2\gamma(\sigma_iU_i-\sigma_i(U_i)_\mathrm{bc}+\eta_i\widetilde\phi_i) ,\frac{\I}
{1+\widetilde{\Gamma}'(\delta\,\cdot)} \cdot \nabla\varphi_i\right)_{L^2(\slab_\delta)}\\
&\quad\,+\left(\left|1+\widetilde{\Gamma}^\prime(\delta \,\cdot\,)\right|^2\left(\gamma(\omega_{\boldsymbol{\tau}}+\tilde u)-\gamma(\sigma_iU_i-\sigma_i(U_i)_\mathrm{bc}+\eta_i\widetilde\phi_i)\right)-\sum_{j\neq i}\sigma_j\gamma(U_j) ,\frac{\I}
{1+\widetilde{\Gamma}'(\delta\,\cdot)} \cdot \nabla\varphi_i\right)_{L^2(\slab_\delta)}\\
&\quad\,+\left(\sum_{j\neq i}\sigma_j(U_j)_\mathrm{bc}-\sum_{j\neq i}\Delta(\eta_j\widetilde\phi_j)-\Delta\widetilde\psi, \frac{\I}
{1+\widetilde{\Gamma}'(\delta\,\cdot)} \cdot \nabla\varphi_i\right)_{L^2(\slab_\delta)}\\
&=:I_1+I_2+I_3+I_4.
\end{align*}}

For $I_1$ and $I_2$, we use the method in Section 5 of \cite{MR4577959}. By the definition of $\varphi_i$ and the boundary correction, we have
$$I_1=\left(-\Delta\varphi_i,\frac{\I}
{1+\widetilde{\Gamma}'(\delta\,\cdot)} \cdot \nabla\varphi_i\right)_{L^2(\slab_\delta)}.$$
Let
$$y=x+\delta^{-1}\widetilde{\Gamma}(\delta x)$$
and denote the corresponding image domain by
$\Omega(\boldsymbol{\tau})=(\mathrm{id}+\delta^{-1}\widetilde{\Gamma}(\delta\,\cdot))(\slab_\delta)$. Let
$$\Psi_i(y):=\varphi_i\circ
\left(\operatorname{id}
+\delta^{-1}\widetilde{\Gamma}(\delta\,\cdot)\right)^{-1},$$
then
$$\partial_{y_2}\Psi_i=\left(\frac{\I}
{1+\widetilde{\Gamma}'(\delta\,\cdot)} \cdot \nabla\varphi_i\right)\circ
\left(\operatorname{id}
+\delta^{-1}\widetilde{\Gamma}(\delta\,\cdot)\right)^{-1}.$$
Hence, by the conformal invariance of the Laplacian,
\begin{equation}\label{1506}
I_1=\int_{\Omega(\boldsymbol{\tau})}(-\Delta_y\Psi_i)\partial_{y_2}\Psi_i\,dy.
\end{equation}
Since $\Psi_i=0$ on $\partial\Omega(\boldsymbol{\tau})$, an integration by parts gives
\begin{align*}
I_1=-\int_{\Omega(\boldsymbol{\tau})}\Delta_y\Psi_i\,\partial_{y_2}\Psi_i\,dy&=\int_{\Omega(\boldsymbol{\tau})}\nabla\partial_{y_2}\Psi_i\cdot\nabla\Psi_i\,dy-\int_{\partial\Omega(\boldsymbol{\tau})}\partial_{y_2}\Psi_i\,\nabla\Psi_i\cdot N\,dS\\
&=\frac{1}{2}\int_{\partial\Omega(\boldsymbol{\tau})}|\nabla\Psi_i|^2N_2\,dS-\int_{\partial\Omega(\boldsymbol{\tau})}\partial_{y_2}\Psi_i\,\nabla\Psi_i\cdot N\,dS\\
&=-\frac{1}{2}\int_{\partial\Omega(\boldsymbol{\tau})}|\nabla\Psi_i|^2N_2\,dS,
\end{align*}
where 
$$N=(N_1,N_2)=\pm\frac{1}{\left|1+\widetilde{\Gamma}'(\delta\,\cdot)\right|}\left(
-\partial_{x_1}\widetilde{\Gamma}_2(\delta\,\cdot),1+\partial_{x_2}\widetilde{\Gamma}_2(\delta\,\cdot)
\right)\circ\left(\operatorname{id}
+\delta^{-1}\widetilde{\Gamma}(\delta\,\cdot)\right)^{-1}$$
denotes the outward unit normal vector and the sign is chosen according to whether the point lies on the upper or lower boundary component. Moreover, the boundary arc-length element is given by
$$dS=\left|1+\widetilde{\Gamma}'(\delta\,\cdot)\right|\circ\left(\operatorname{id}
+\delta^{-1}\widetilde{\Gamma}(\delta\,\cdot)\right)^{-1}\,dx_1.$$
Using the expression for the outward unit normal vector and the
arc-length element, we can rewrite $I_1$ as an integral over the
boundary of $\slab_\delta$. Since $\varphi_i|_{\partial\slab_\delta}=0$, tangential derivatives do not appear in the boundary integrals. Reversing the coordinate change, we obtain
\begin{equation}\label{eq:I1-boundary}
\begin{aligned}
I_1=&-\frac{1}{2}\displaystyle\int_{\mathbb R}\frac{1+\partial_{x_2}\widetilde{\Gamma}_2(\delta x_1,1)}{\left|1+\widetilde{\Gamma}'(\delta x_1,1)\right|^2}\left|\partial_{x_2}\varphi_i\left(x_1,\frac{1}{\delta}\right)\right|^2\,dx_1\\
&+\frac{1}{2}\displaystyle\int_{\mathbb R}\frac{1+\partial_{x_2}\widetilde{\Gamma}_2(\delta x_1,-1)}{\left|1+\widetilde{\Gamma}'(\delta x_1,-1)\right|^2}\left|\partial_{x_2}\varphi_i\left(x_1,-\frac{1}{\delta}\right)\right|^2\,dx_1.
\end{aligned}
\end{equation}

For the upper boundary, using Lemma \ref{prop:sign of Uy}, Corollary \ref{C:Ubc} \eqref{2031} and \eqref{2032}, we have
\begin{align*}
\frac{1}{2}\int_{\mathbb R}\left|\partial_{x_2}\varphi_i\left(x_1,\frac{1}{\delta}\right)\right|^2\,dx_1&=2\int_{\mathbb R}\left|\partial_{x_2}U_i\left(x_1,\frac{1}{\delta}\right)\right|^2\,dx_1+O\left(\delta^{\frac{1}{2}}e^{-\frac{3(1-|\tau_i|)}\delta}\right)\\
&=C_{\mathrm{up}}^i(\delta)\delta^{\frac{1}{2}}e^{-\frac {2(1-\tau_i)}\delta}+O\left(\delta^{\frac{1}{2}}e^{-\frac{3(1-|\tau_i|)}\delta}\right),
\end{align*}
where
$$C_{\mathrm{up}}^i(\delta):=2\delta^{-\frac{1}{2}}e^{\frac {2(1-\tau_i)}\delta}\int_{\mathbb R}\left|\partial_{x_2}U_i\left(x_1,\frac{1}{\delta}\right)\right|^2\,dx_1.$$
Indeed, $C_{\mathrm{up}}^i(\delta)$ can be written as
$$C_{\mathrm{up}}^i(\delta)=C_{\mathrm{up}}^i+o(1),\qquad\text{ as }\delta\to0,$$
and by Lebesgue's Dominated Convergence Theorem,
\begin{align*}
C_{\mathrm{up}}^i&=\lim_{\delta\to0}2\delta^{-\frac{1}{2}}e^{\frac {2(1-\tau_i)}\delta}\int_{\mathbb R}\left|\partial_{x_2}U_i\left(x_1,\frac{1}{\delta}\right)\right|^2\,dx_1\\
&=2\lambda^2(1-\tau_i)^2\lim_{\delta\to0}\delta^{-\frac{1}{2}}\int_{\mathbb R}\left|x^2+(1-\tau_i)^2\right|^{-\frac{3}{2}}\e^{-\frac{2}{\delta}\left(\sqrt{x^2+(1-\tau_i)^2}-(1-\tau_i)\right)}\,dx\\
&=2\lambda^2(1-\tau_i)^2\lim_{\delta\to0}\int_{\mathbb R}\left|\delta x^2+(1-\tau_i)^2\right|^{-\frac{3}{2}}\e^{-\frac{2x^2}{\sqrt{\delta x^2+(1-\tau_i)^2}+(1-\tau_i)}}\,dx=\frac{2\lambda^2\pi^{\frac{1}{2}}}{(1-\tau_i)^{\frac{1}{2}}}.
\end{align*}
Similarly, for the lower boundary,
$$\frac{1}{2}\int_{\mathbb R}\left|\partial_{x_2}\varphi_i\left(x_1,-\frac{1}{\delta}\right)\right|^2\,dx_1=C_{\mathrm{low}}^i(\delta)\delta^{\frac{1}{2}}e^{-\frac {2(1+\tau_i)}\delta}+O\left(\delta^{\frac{1}{2}}e^{-\frac{3(1-|\tau_i|)}\delta}\right),$$
where
$$C_{\mathrm{low}}^i(\delta)=C_{\mathrm{low}}^i+o(1)=\frac{2\lambda^2\pi^{\frac{1}{2}}}{(1+\tau_i)^{\frac{1}{2}}}+o(1),\qquad\text{ as }\delta\to0.$$
From the proof of Lemma \ref{lem6.1},
$$\left|1+\widetilde{\Gamma}'\right|=1+O(\delta^{-1} e^{-\frac {2(1-|\boldsymbol{\tau}|)}\delta}+\delta^{-1}|\log \delta|e^{-\frac {2d_m}\delta}),$$
the factors involving $\widetilde{\Gamma}$ in
\eqref{eq:I1-boundary} contribute only to the lower-order terms. Hence
\begin{equation}\label{I1}
I_1=C_{\mathrm{low}}^i(\delta)\delta^{\frac 12}e^{-\frac {2(1 +\tau_i)}\delta}-C_{\mathrm{up}}^i(\delta)\delta^{\frac 12}e^{-\frac{2(1-\tau_i)}\delta}+O\left(\delta^{\frac 12}e^{-\frac {3(1-|\boldsymbol{\tau}|)}\delta}+\delta^{-\frac 12}|\log \delta|e^{-\frac {2d_m+2(1-|\boldsymbol{\tau}|)}\delta}\right).
\end{equation}

We next consider $I_2$, which can be handled directly by writing $\gamma$ as a derivative. Analogously to \eqref{1506}, we have
$$I_2=\int_{\Omega(\boldsymbol{\tau})}\gamma(\Psi_i)\partial_{y_2}\Psi_i\,dy.$$
Let $\mathcal{G}$ be an antiderivative of $\gamma$ satisfying
$\mathcal{G}(0)=0$, that is,
$$\mathcal{G}'(s)=\gamma(s).$$
Then
\begin{align*}
I_2&=\int_{\Omega(\boldsymbol{\tau})}\gamma(\Psi_i)\partial_{y_2}\Psi_i\,dy=\int_{\Omega(\boldsymbol{\tau})}\partial_{y_2}\mathcal{G}(\Psi_i)\,dy=\int_{\partial\Omega(\boldsymbol{\tau})}\mathcal{G}(\Psi_i)N_2\,dS.
\end{align*}
Since $\Psi_i=0$ on $\partial\Omega(\boldsymbol{\tau})$ and $\mathcal{G}(0)=0$, it follows that
\begin{equation}\label{1517}
I_2=\int_{\partial\Omega(\boldsymbol{\tau})}\mathcal{G}(\Psi_i)N_2\,dS=0.
\end{equation}

For $I_3$, by \eqref{1434},
$$I_3\eqsim\left(\gamma(\omega_{\boldsymbol{\tau}}+\tilde u)-\gamma(\sigma_iU_i-\sigma_i(U_i)_\mathrm{bc}+\eta_i\widetilde\phi_i)-\sum_{j\neq i}\sigma_j\gamma(U_j) ,\partial_{x_2}\varphi_i\right)_{L^2(\slab_\delta)}.$$
Since
\begin{align*}
&\gamma(\omega_{\boldsymbol{\tau}}+\tilde u)-\gamma(\sigma_iU_i-\sigma_i(U_i)_{\mathrm{bc}}+\eta_i\widetilde\phi_i)-\sum_{j\neq i}\sigma_j\gamma(U_j)\\
&\eqsim\sum_{j\neq i}\sigma_j(\gamma'(U_i)-1)U_j-\sum_{j\neq i}\sigma_j(U_j)_{\mathrm{bc}}+\sum_{j\neq i}\sigma_j\eta_j\widetilde\phi_j+\widetilde\psi,\qquad\text{in }B_{i,2},
\end{align*}
by Lemma \ref{prop:sign of Uy}, \eqref{1611}, \eqref{1859}, Assumption \ref{assumption:A} and integration by parts, we have
\begin{equation*}
\begin{aligned}
&\int_{\slab_\delta}\sigma_i\partial_{x_2}U_i(\gamma'(U_i)-1)\sigma_jU_j\,dx=\int_{\slab_\delta}\sigma_i\partial_{x_2}(\gamma(U_i)-U_i)\sigma_jU_j\,dx\\
&=\sigma_i\sigma_j\left[-\int_{\slab_\delta}(\gamma(U_i)-U_i)\partial_{x_2}U_j\,dx+\int_{\p\slab_\delta}(\gamma(U_i)-U_i)U_jN_2\,dS\right]\\
&=-\sigma_i\sigma_j\int_{B_r(0)}(\gamma(U(y))-U(y))\partial_{x_2}U\left(\frac{\tau_i-\tau_j}{\delta}\vec e_2+y\right)\,dy\\
&\quad\,+O\left(\int_{\p\slab_\delta}U_i^2U_jN_2\,dS+\int_{\slab_\delta\backslash B_r\left(\frac{\tau_i}{\delta}\vec e_2\right)}U_i^2\partial_{x_2}U_j\,dx\right)\\
&=\lambda\sigma_i\sigma_j\operatorname{sgn}(\tau_i-\tau_j)\delta^{\frac{1}{2}}\int_{B_r(0)}(\gamma(U(y))-U(y))|(\tau_i-\tau_j)\vec e_2+\delta y|^{-\frac{1}{2}}\e^{-\left|\frac{\tau_i-\tau_j}{\delta}\vec e_2+y\right|}\,dy\\
&\quad\,+O\left(\e^{-\frac{3(1-|\boldsymbol{\tau}|)}{\delta}}+(\delta r)^{-\frac{1}{2}}\e^{-r}\e^{-\frac{|\tau_i-\tau_j|}{\delta}}\right).
\end{aligned}
\end{equation*}
Here, $\lambda>0$ denotes the constant appearing in Lemma \ref{prop:sign of Uy}, and for consistency of presentation, we choose $r=\frac{R}{4\delta}$ with $R$ introduced in \eqref{1555} to guarantee $B_r\left(\frac{\tau_i}{\delta}\vec e_2\right)\subset\slab_\delta$. Using Assumption \ref{assumption:A}, we have
$$\gamma(t)-t<0\quad\text{for }t>0.$$
Let
$$C_{\mathrm{int}}^{ij}(\delta):=-\lambda\int_{B_{\frac{R}{4\delta}}(0)}(\gamma(U(y))-U(y))|(\tau_i-\tau_j)\vec e_2+\delta y|^{-\frac{1}{2}}\e^{-\left|\frac{\tau_i-\tau_j}{\delta}\vec e_2+y\right|+\frac{|\tau_i-\tau_j|}{\delta}}\,dy.$$
by Lemma \ref{prop:sign of Uy}, 
\begin{align*}
|C_{\mathrm{int}}^{ij}(\delta)|\leq C\lambda\int_{B_{\frac{R}{4\delta}}(0)}(1+|y|)^{-1}\e^{-2|y|}\e^{|y|}\,dy<\infty.
\end{align*}
Then 
$$C_{\mathrm{int}}^{ij}(\delta)=C_{\mathrm{int}}^{ij}+o(1),\text{ as }\delta\to0,$$
and using the radial symmetry of $U$,
\begin{align*}
C_{\mathrm{int}}^{ij}=\lim_{\delta\to0}C_{\mathrm{int}}^{ij}(\delta)&=-\lambda\int_{\R^2}(\gamma(U(y))-U(y))|\tau_i-\tau_j|^{-\frac{1}{2}}\e^{-\frac{\tau_i-\tau_j}{|\tau_i-\tau_j|}y_2}\,dy\\
&=-\frac{\lambda}{|\tau_i-\tau_j|^{\frac{1}{2}}}\int_{\R^2}(\gamma(U(y))-U(y))\e^{y_2}\,dy>0.
\end{align*}
Hence $C_{\mathrm{int}}^{ij}(\delta)>0$ for $\delta$ small and 
\begin{align*}
\int_{\slab_\delta}\sigma_i\partial_{x_2}U_i(\gamma'(U_i)-1)\sigma_jU_j\,dx&=-C_{\mathrm{int}}^{ij}(\delta)\sigma_i\sigma_j\operatorname{sgn}(\tau_i-\tau_j)\delta^{\frac12}e^{-\frac{|\tau_i-\tau_j|}{\delta}}\\
&\quad\,+O\left(\e^{-\frac{3(1-|\boldsymbol{\tau}|)}{\delta}}+\e^{-\frac{R}{4\delta}}\e^{-\frac{|\tau_i-\tau_j|}{\delta}}\right).
\end{align*}
The remaining terms are estimated using the argument in Corollary \ref{C:Ubc}, together with \eqref{1611} and
\eqref{1434}, we have
\begin{align*}
&\int_{\slab_\delta}(\gamma^\prime(U_i)-1)U_j\p_{x_2}(U_i)_\mathrm{bc}\,dx+\int_{\slab_\delta}(\gamma^\prime(U_i)-1)U_j\p_{x_2}(\eta_i\widetilde\phi_i)\,dx-\int_{\slab_\delta}(U_j)_\mathrm{bc}\partial_{x_2}\varphi_i \,dx\\
&\lesssim\left|\left(\partial_{x_2}(U_i)_\mathrm{bc},(\gamma^\prime(U_i)-1)U_j\right)_{L^2(\slab_\delta)}\right|+\int_{B_{i,2}}|U_j||\p_{x_2}(\eta_i\widetilde\phi_i)|\,dx+\left|\left((U_j)_\mathrm{bc},\partial_{x_2}U_i\right)_{L^2(\slab_\delta)}\right|\\
&\lesssim \delta^{\frac34}|\log\delta|^{\frac12}e^{-\frac{d_m+|\tau_i-\tau_j|-2R}{\delta}}+\delta^{\frac{1}{2}-k_0}e^{-\frac{2(1-|\boldsymbol{\tau}|)+|\tau_i-\tau_j|-2R}{\delta}}+\delta^{-\frac{5}{2}}\e^{-\frac{2-\tau_i-\tau_j}{\delta}}+\delta^{-\frac{5}{2}}\e^{-\frac{2+\tau_i+\tau_j}{\delta}}.
\end{align*}
On the other hand, since the supports of $\eta_i$ and $\eta_j$ are disjoint whenever $i\neq j$, by \eqref{1457},
{\small\begin{align*}
\left(\eta_j\widetilde\phi_j , \frac{\I}
{1+\widetilde{\Gamma}'(\delta\,\cdot)} \cdot \nabla\varphi_i\right)_{L^2(\slab_\delta)}=O\left(\delta^{\frac34}|\log\delta|^{\frac12}e^{-\frac{d_m+|\tau_i-\tau_j|-2R}{\delta}}+\delta^{\frac{1}{2}-k_0}e^{-\frac{2(1-|\boldsymbol{\tau}|)+|\tau_i-\tau_j|-2R}{\delta}}\right).
\end{align*}}

\noindent
Finally, using the pointwise estimate for $\widetilde\psi$ in Proposition \ref{prop5.2},
\begin{align*}
\left(\widetilde\psi,\frac{\I}
{1+\widetilde{\Gamma}'(\delta\,\cdot)} \cdot \nabla\varphi_i\right)_{L^2(\slab_\delta)}=O\left(\delta^{-\frac{3}{2}}\e^{-\frac{R}{2\delta}}\e^{-\frac {d_m}\delta}+\delta^{-\frac{3}{2}-k_0}\e^{-\frac{R}{2\delta}}\e^{-\frac {2(1-|\boldsymbol{\tau}|)}\delta}\right).
\end{align*}
Therefore,
\begin{equation}\label{1745}
\begin{aligned}
I_3&=-\sum_{j\neq i}C_{\mathrm{int}}^{ij}(\delta)\sigma_i\sigma_jsgn(\tau_i-\tau_j)\delta^{\frac{1}{2}}\e^{-\frac{|\tau_i-\tau_j|}{\delta}}+O\left(\e^{-\frac{R}{4\delta}}\e^{-\frac {d_m}\delta}+\e^{-\frac{R}{4\delta}}\e^{-\frac {2(1-|\boldsymbol{\tau}|)}\delta}\right)\\
&\quad\,+O\left(\sum_{j\neq i}\left(\e^{-\frac{d_m+|\tau_i-\tau_j|-2R}{\delta}}+\e^{-\frac{2(1-|\boldsymbol{\tau}|)+|\tau_i-\tau_j|-R}{\delta}}+\delta^{-\frac{5}{2}}\left(\e^{-\frac{2-\tau_i-\tau_j}{\delta}}+\e^{-\frac{2+\tau_i+\tau_j}{\delta}}\right)\right)\right).
\end{aligned}
\end{equation}

For $I_4$, analogously to $I_3$, using the argument in Corollary \ref{C:Ubc}, 
\begin{equation}\label{1929}
\begin{aligned}
I_4&=\left(\sum_{j\neq i}\sigma_j(U_j)_\mathrm{bc}-\sum_{j\neq i}\Delta(\eta_j\widetilde\phi_j)-\Delta\widetilde\psi,\frac{\I}
{1+\widetilde{\Gamma}'(\delta\,\cdot)} \cdot \nabla\varphi_i\right)_{L^2(\slab_\delta)}\\
&\lesssim\sum_{j\neq i}\left|\left((U_j)_\mathrm{bc},\partial_{x_2}U_i\right)_{L^2(\slab_\delta)}\right|-\left(\widetilde\psi-f_{\mathrm{out}}[\widetilde\phi_1,\dots,\widetilde\phi_n,\widetilde\psi,\widetilde\Gamma_{\mathrm{s}}],\frac{\I}
{1+\widetilde{\Gamma}'(\delta\,\cdot)} \cdot \nabla\varphi_i\right)_{L^2(\slab_\delta)}\\
&\quad\,+\sum_{j\neq i}\left(\delta^{\frac34}|\log\delta|^{\frac12}e^{-\frac{d_m+|\tau_i-\tau_j|-2R}{\delta}}+\delta^{\frac{1}{2}-k_0}e^{-\frac{2(1-|\boldsymbol{\tau}|)+|\tau_i-\tau_j|-2R}{\delta}}\right)\\
&\lesssim\delta^{-\frac{5}{2}}\sum_{j\neq i}\left(\e^{-\frac{2-\tau_i-\tau_j}{\delta}}+\e^{-\frac{2+\tau_i+\tau_j}{\delta}}\right)+\delta^{-\frac{3}{2}}\e^{-\frac{R}{2\delta}}\e^{-\frac {d_m}\delta}+\delta^{-\frac{3}{2}-k_0}\e^{-\frac{R}{2\delta}}\e^{-\frac {2(1-|\boldsymbol{\tau}|)}\delta}\\
&\quad\,+\sum_{j\neq i}\left(\delta^{\frac34}|\log\delta|^{\frac12}e^{-\frac{d_m+|\tau_i-\tau_j|-2R}{\delta}}+\delta^{\frac{1}{2}-k_0}e^{-\frac{2(1-|\boldsymbol{\tau}|)+|\tau_i-\tau_j|-2R}{\delta}}\right),
\end{aligned}
\end{equation}
where we used Remark \ref{1631} and the pointwise estimate for $\widetilde\psi$ in Proposition \ref{prop5.2}.

Combining \eqref{I1}, \eqref{1517}, \eqref{1745} and \eqref{1929}, we obtain
\begin{equation*}
\begin{aligned}
b_i(\boldsymbol{\tau})=&C_{\mathrm{low}}^i(\delta)\delta^{\frac 12}e^{-\frac {2(1 +\tau_i)}\delta}-C_{\mathrm{up}}^i(\delta)\delta^{\frac 12}e^{-\frac{2(1-\tau_i)}\delta}-\sum_{j\neq i}C_{\mathrm{int}}^{ij}(\delta)\sigma_i\sigma_jsgn(\tau_i-\tau_j)\delta^{\frac{1}{2}}\e^{-\frac{|\tau_i-\tau_j|}{\delta}}\\
&{\scalebox{0.98}{$\displaystyle +O\left(\sum_{j\neq i}\left(\e^{-\frac{d_m+|\tau_i-\tau_j|-2R}{\delta}}+\e^{-\frac{2(1-|\boldsymbol{\tau}|)+|\tau_i-\tau_j|-3R}{\delta}}+\delta^{-\frac{5}{2}}\e^{-\frac{2-\tau_i-\tau_j}{\delta}}+\delta^{-\frac{5}{2}}\e^{-\frac{2+\tau_i+\tau_j}{\delta}}\right)\right)$}}\\
&+O\left(\e^{-\frac{3(1-|\boldsymbol{\tau}|)}\delta}+\e^{-\frac{d_m+2(1-|\boldsymbol{\tau}|)}\delta}+\e^{-\frac{R}{4\delta}}\e^{-\frac {d_m}\delta}+\e^{-\frac{R}{4\delta}}\e^{-\frac {2(1-|\boldsymbol{\tau}|)}\delta}\right).
\end{aligned}
\end{equation*}
For simplicity, we restrict our attention to $\boldsymbol{\tau}\in P_{\varepsilon}$ (here $\varepsilon<\frac{1}{n}$) satisfying
$$-1<\tau_n<\tau_{n-1}<\cdots<\tau_{2}<\tau_{1}<1,$$
$$2(\tau_n+1)\approx2(1-\tau_1)\approx\tau_{m}-\tau_{m+1}\approx\frac{2}{n},\qquad m=1,\dots,n-1.$$
Thus, let
$$\boldsymbol{\tau}_0=\left(\frac{n-1}{n},\frac{n-3}{n},\dots,\frac{3-n}{n},\frac{1-n}{n}\right)^T,$$
we consider the $n$-dimensional closed ball
\begin{equation}\label{1942}
B_r(\boldsymbol{\tau}_0):=\{\boldsymbol{\tau}\in\R^n\,|\,\|\boldsymbol{\tau}-\boldsymbol{\tau}_0\|_{l^\infty}\leq r\},
\end{equation}
where $r>0$ is a sufficiently small fixed constant, $\|\boldsymbol{\tau}\|_{l^\infty}:=\displaystyle\max_{1\leq i\leq n}|\tau_i|$.

For $\boldsymbol{\tau}\in B_r(\boldsymbol{\tau}_0)$ and $R>r>0$ sufficiently small, we have
\begin{align*}
b_1(\boldsymbol{\tau})=0&\;\iff\widetilde C_1(\delta)e^{\frac{2(1-\tau_1)-(\tau_1-\tau_2)}{\delta}}-1+O\left(\delta^{-\frac{1}{2}}\e^{-\frac{1-5nr}{n\delta}}+\delta^{-\frac{1}{2}}\e^{-\frac{R-8r}{4\delta}}\right)=0,\\
b_2(\boldsymbol{\tau})=0&\;\iff\widetilde C_2(\delta)e^{\frac{(\tau_1-\tau_2)-(\tau_2-\tau_3)}{\delta}}-1+O\left(\delta^{-\frac{1}{2}}\e^{-\frac{1-5nr}{n\delta}}+\delta^{-\frac{1}{2}}\e^{-\frac{R-8r}{4\delta}}\right)=0,\\
&\;\;\,\,\,\,\vdots\\
b_j(\boldsymbol{\tau})=0&\;\iff\widetilde C_j(\delta)e^{\frac{(\tau_{j-1}-\tau_{j})-(\tau_{j}-\tau_{j+1})}{\delta}}-1+O\left(\delta^{-\frac{1}{2}}\e^{-\frac{1-5nr}{n\delta}}+\delta^{-\frac{1}{2}}\e^{-\frac{R-8r}{4\delta}}\right)=0,\\
&\;\;\,\,\,\,\vdots\\
b_{n-1}(\boldsymbol{\tau})=0&\;\iff\widetilde C_{n-1}(\delta)e^{\frac{(\tau_{n-2}-\tau_{n-1})-(\tau_{n-1}-\tau_{n})}{\delta}}-1+O\left(\delta^{-\frac{1}{2}}\e^{-\frac{1-5nr}{n\delta}}+\delta^{-\frac{1}{2}}\e^{-\frac{R-8r}{4\delta}}\right)=0,\\
b_n(\boldsymbol{\tau})=0&\;\iff\widetilde C_n(\delta)e^{\frac{2(1+\tau_n)-(\tau_{n-1}-\tau_{n})}{\delta}}-1+O\left(\delta^{-\frac{1}{2}}\e^{-\frac{1-5nr}{n\delta}}+\delta^{-\frac{1}{2}}\e^{-\frac{R-8r}{4\delta}}\right)=0,
\end{align*}
where $\widetilde C_i(\delta)>0$ are uniformly bounded as $\delta\to0$ for all $i=1,\dots,n$, and we have used the fact that $\sigma_i\sigma_j=-1$ for adjacent
indices $i$ and $j$. 

We now prove that for small $\delta$, there exists $\boldsymbol{\tau}\in P_\varepsilon$ ($\varepsilon<\frac{1}{n}$) that solves the following system
\begin{equation}\label{129}
\begin{cases}
\tau_2-3\tau_1&=-2+N_1(\boldsymbol{\tau})\\
\tau_1-2\tau_2+\tau_3&=0+N_2(\boldsymbol{\tau})\\
&\cdots\\
\tau_{j-1}-2\tau_{j}+\tau_{j+1}&=0+N_j(\boldsymbol{\tau})\\
&\cdots\\
\tau_{n-2}-2\tau_{n-1}+\tau_{n}&=0+N_{n-1}(\boldsymbol{\tau})\\
3\tau_n-\tau_{n-1}&=-2+N_n(\boldsymbol{\tau}).
\end{cases}
\end{equation}
Here \(N(\boldsymbol{\tau}):=(N_1(\boldsymbol{\tau}),\ldots,N_n(\boldsymbol{\tau}))^T\) denotes the perturbation vector, whose components are defined by the exact equations obtained from $b_i(\boldsymbol{\tau})=0$ after normalization and taking logarithms. Moreover, the asymptotic estimates derived above yield
$$N_j(\boldsymbol{\tau})=-\delta\log\widetilde C_j(\delta)+O\left(\delta^{-\frac{1}{2}}\e^{-\frac{1-5nr}{n\delta}}+\delta^{-\frac{1}{2}}\e^{-\frac{R-8r}{4\delta}}\right).$$
To establish the solvability of the above system,  we introduce the matrix $A_n$ defined as follows.
$$A_n:=\begin{pmatrix}
-3&1\\
-1&3
\end{pmatrix}\quad\text{for }n=2,\qquad A_n:=\begin{pmatrix}
-3 & 1 & 0 & \cdots & 0\\
1 & -2 & 1 & \ddots & 0\\
\vdots & \ddots & \ddots & \ddots & 0\\
0 & \cdots & 1 & -2 & 1\\
0 & \cdots & 0 & -1 & 3
\end{pmatrix}\quad\text{for }n\geq3.$$
Since $A_n$ is an irreducible, weakly diagonally dominant matrix and is strictly diagonally dominant in at least one row, it follows that $A_n$ is invertible. In view of the fact that the perturbation term also depends on $\boldsymbol{\tau}$, We write system \eqref{129} as
\begin{equation}\label{1812}
\boldsymbol{\tau}=A_n^{-1}\begin{pmatrix}
    -2\\
    0\\
    \vdots\\
    0\\
    -2
\end{pmatrix}+A_n^{-1}N(\boldsymbol{\tau})=\begin{pmatrix}
    \frac{n-1}{n}\\
    \frac{n-3}{n}\\
    \vdots\\
    \frac{3-n}{n}\\
    \frac{1-n}{n}\\
\end{pmatrix}+A_n^{-1}N(\boldsymbol{\tau})=\boldsymbol{\tau}_0+A_n^{-1}N(\boldsymbol{\tau}).
\end{equation}
Before proceeding with the proof, we first clarify how the parameters $R$ in \eqref{1555} and $r$ in \eqref{1942} are chosen. For $\boldsymbol{\tau}\in P_\varepsilon$, we choose a fixed $R<\frac{\varepsilon}{4}$ such that the regions corresponding to different inner problems are disjoint. After that, we choose a fixed parameter $r<\min\left\{\frac{R}{4},\frac{1}{5n}\right\}$.

Next we use the Brouwer fixed-point theorem to prove that \eqref{1812} has a solution. Define the map $T(\boldsymbol{\tau})$ by
$$T(\boldsymbol{\tau})=\boldsymbol{\tau}_0+A_n^{-1}N(\boldsymbol{\tau}).$$
From the continuity of $\omega_{\boldsymbol{\tau}}$ in $\boldsymbol{\tau}$, it follows that $\widetilde\phi_i(\boldsymbol{\tau})$, $\widetilde{\psi}(\boldsymbol{\tau})$ and $\widetilde{\Gamma}_{\mathrm{s}}(\boldsymbol{\tau})$ are all continuous with respect to $\boldsymbol{\tau}$. Hence, $b_i(\boldsymbol{\tau})$ are continuous with respect to $\boldsymbol{\tau}$, $i=1,\dots,n$. From the preceding derivation, we obtain that $N(\boldsymbol{\tau})$ is continuous with respect to $\boldsymbol{\tau}$ in $B_r(\boldsymbol{\tau}_0)$. By the Brouwer fixed-point theorem, it suffices to verify that
$$T(B_r(\boldsymbol{\tau}_0))\subset B_r(\boldsymbol{\tau}_0).$$
Using the estimate in \eqref{129} and the fact that $\widetilde C_i(\delta)>0$ are uniformly bounded above and below as $\delta\to0$ for all $i=1,\dots,n$, we have
\begin{align*}
\|T(B_r(\boldsymbol{\tau}_0))-\boldsymbol{\tau}_0\|_{l^\infty}=\|A_n^{-1}N(B_r(\boldsymbol{\tau}_0))\|_{l^\infty}\leq C\|N(B_r(\boldsymbol{\tau}_0))\|_{l^\infty}\leq C\delta.
\end{align*}
Therefore, for $\delta$ sufficiently small such that $C\delta<r$, system \eqref{129} admits a solution $\tilde{\boldsymbol{\tau}}=(\tilde\tau_1,\dots,\tilde\tau_n)$ that satisfies
\begin{equation}\label{1526}
\begin{cases}
\tilde\tau_1=\frac{n-1}{n}+O\left(\delta\right)\\
\tilde\tau_2=\frac{n-3}{n}+O\left(\delta\right)\\
\quad\cdots\\
\tilde\tau_{j}=\frac{n-2j+1}{n}+O\left(\delta\right)\\
\quad\cdots\\
\tilde\tau_{n-1}=\frac{3-n}{n}+O\left(\delta\right)\\
\tilde\tau_n=\frac{1-n}{n}+O\left(\delta\right)
\end{cases}\quad\text{as }\delta\to0.
\end{equation}
Consequently, $(\tilde{\boldsymbol{\tau}}, \tilde u(\tilde{\boldsymbol{\tau}}), \widetilde{\Gamma}_{\mathrm{s}}(\tilde{\boldsymbol{\tau}}))$ solves \eqref{perturbed problem}, and thus corresponds to a solution of the stationary capillary-gravity wave problem.

The stream function is given by
\begin{equation*}
\Psi =  \left(\omega_{\tilde{\boldsymbol{\tau}}}+ \tilde u (\tilde{\boldsymbol{\tau}})\right) \circ \left( \tfrac 1\delta \left( id + \widetilde \Gamma (\tilde{\boldsymbol{\tau}}) \right)^{-1}\right), 
\end{equation*}
defined on 
$$\Omega = \left(id + \widetilde \Gamma (\tilde{\boldsymbol{\tau}}) \right) \left(\{(x_1,x_2)\in\R^2\;|\;|x_2|< 1\}\right).$$
For simplicity of presentation, let
$$\widetilde v(\cdot/\delta):=\Psi-\omega_{\tilde{\boldsymbol{\tau}}}(\cdot/\delta).$$
Furthermore, the construction in Section \ref{1448} gives
$$\zeta_\delta=\widetilde\Gamma_2\circ\left(id+\widetilde\Gamma_1(\cdot,1)\right)^{-1}.$$
Then by \eqref{1457}-\eqref{1434} and scaling property,
\begin{align*}
&\left\|\widetilde v(\cdot/\delta)\right\|_{H^{k_0}(\Omega)}\lesssim\delta^{1-k_0}\left\|\widetilde v\right\|_{H^{k_0}(\Omega/\delta)}\\
&\lesssim\delta^{1-k_0}\left(\sum_{i=1}^n\left\|\widetilde\phi_i\right\|_{H^{k_0}(\slab_\delta)}+\left\|\widetilde{\psi}\right\|_{H^{k_0}(\slab_\delta)}+\left\|\widetilde{\Gamma}_{\mathrm{s}}\right\|_{H^{k_0}(\R)}\right)\lesssim\delta^{\beta}\exp\left(-\frac{2}{n\delta}\right),
\end{align*}
\begin{align*}
\|\zeta_\delta\|_{H^{k_0}(\mathbb R)}\lesssim\left\|\widetilde{\Gamma}_{\mathrm{s}}\right\|_{H^{k_0}(\R)}\lesssim\delta^{\beta}|\log\delta|\exp\left(-\frac{2}{n\delta}\right),
\end{align*}
where the constant $\beta$  depends on the asymptotic configuration in \eqref{1526} and on $k_0$. This completes the proof of the main theorem. 
\end{proof}

\section*{Acknowledgements}
Y. Wang is partially supported by NSF of China (No. 12471110). Y. Zheng is supported by NSF of China (No. 12671246).

\bibliography{RefDatabase}{}
\bibliographystyle{plain}

\end{document}